\documentclass[reqno,twoside,11pt,english]{amsart}
\usepackage{amsmath,amsfonts,amssymb,amsthm,epsfig}
\newcommand{\RNum}[1]{\uppercase\expandafter{\romannumeral #1\relax}}
\usepackage{comment}
\usepackage{ulem}
\usepackage{marginnote}
\usepackage{color}
\usepackage[final,colorlinks,linkcolor=blue,anchorcolor=red,citecolor=blue]{hyperref}
\usepackage{graphicx}
\usepackage{titletoc,epsf}
\usepackage{amsmath,amsfonts,latexsym,amsthm,amsxtra,amssymb,bbm}
\allowdisplaybreaks
\usepackage{graphicx,float}
\newcommand{\Si}{\Sigma}
\newcommand{\lam}{\lambda}
\newcommand{\La}{\Lambda}

\newcommand{\ba}{\backslash}

\newcommand{\la}{\lambda}
\newcommand{\de}{\delta}
\newcommand{\De}{\Delta}
\newcommand{\na}{\nabla}
\newcommand{\Om}{\Omega}

\newcommand{\oo}{\infty}

\newcommand{\pa}{\partial}

\newcommand{\si}{\sigma}
\newcommand{\va}{\varphi}

\newcommand{\non}{\nonumber}

\newcommand{\al}{\alpha}
\newcommand{\ga}{\gamma}

\newcommand{\dist}{\text{\rm dist}}
\newcommand{\medint}{-\kern -,375cm\int}         %  average integral
\newcommand{\medintinrigo}{-\kern -,315cm\int}
\newcommand{\wto}{\rightharpoonup}                %  weak convergence

  \newcommand{\wq}{\infty}
 \newcommand{\cal}{\mathcal}

\newcommand{\3}{{\mathcal C}}

\newcommand{\HH}{{\mathcal H}}

\newcommand{\R}{{\mathbb R}}

\newcommand{\N}{{\mathbb N}}

\newcommand{\ep}{\varepsilon}
\newcommand{\su}{\subset}

\newcommand{\ri}{\rightarrow}

\def\supp{\textup{supp}}

\def\loc{\text{\rm loc}}

\newcommand{\reg}{\text{\rm reg}}
\newcommand{\sing}{\text{\rm sing}}
\newcommand{\spt}{\text{\rm spt}}

\newcommand{\Span}{\text{\rm span}}
\numberwithin{equation}{section}
\newtheorem{theorem}{Theorem}[section]

\newtheorem{corollary}[theorem]{Corollary}

\newtheorem{definition}[theorem]{Definition}

\newtheorem{lemma}[theorem]{Lemma}

\newtheorem{proposition}[theorem]{Proposition}
\newtheorem{remark}[theorem]{Remark}
\newtheorem*{remark*}{Remark}

\theoremstyle{definition}

\makeatother

\begin{document}

\title[Quantitative regularity of 1/2-harmonic maps]{Quantitative stratification and global regularity for 1/2-harmonic mappings}

 \author[C.Y. Guo, G.C. Jiang, C. Wang, C.L. Xiang, G.F. Zheng ]{Changyu Guo, Guichun Jiang$^*$, Changyou Wang,  Changlin Xiang, \\
 Gaofeng Zheng}

\address[Changyu Guo]{Research Center for Mathematics and Interdisciplinary Sciences, Shandong University, 266237, Qingdao, P. R. China, and Department of Physics and Mathematics, University of Eastern Finland, 80101, Joensuu, Finland}
\email{changyu.guo@sdu.edu.cn}

\address[Guichun Jiang]{School of Mathematics and Statistics, Hubei Normal University, and Huangshi Key Laboratory of Metaverse and Virtual Simulation, Huangshi 435002, China}
\email{gcjiang@hbnu.edu.cn}

\address[Changyou Wang]{Department of Mathematics, Purdue University, West Lafayette, IN 47907, USA}
\email{wang2482@purdue.edu}

\address[Changlin Xiang]{Three Gorges Mathematical Research Center, China Three Gorges University,  443002, Yichang,  P. R. China}
\email{changlin.xiang@ctgu.edu.cn}

\address[Gaofeng Zheng]{School of Mathematics and Statistics, and Key Lab NAA--MOE, Central China Normal University, Wuhan 430079,  P. R.  China}
\email{gfzheng@ccnu.edu.cn}

	
\thanks{$^*$Corresponding author: Guichun Jiang}
%\thanks{C.-Y. Guo is supported by the Young Scientist Program of the Ministry of Science and Technology of China (No.~2021YFA1002200), the NSF of China (No.~12101362 and 12311530037), the Taishan Scholar Project and the NSF of Shandong Province (No.~ZR2022YQ01). C. Y. Wang is partially supported by NSF grants 1764417 and 2101224.  C.-L. Xiang is financially supported by the NSFC  (No.~12271296) and  the NSF of Hubei province (No. 2024AFA061). G.-F. Zheng is supported by the National Natural Science Foundation of China (No.~12271195 and No.~12171180). Jiang, Xiang and Zheng are also partly supported by the Open Research Fund of Key Laboratory of Nonlinear Analysis \& Applications  (Central China Normal University) (No.NAA2023ORG003), Ministry of Education, P. R. China.}

\date{\today}

\begin{abstract}
In this paper, we extend the celebrated global regularity theory of Naber-Valtorta [Ann. Math. 2017] to 1/2-harmonic mappings into manifolds. Inspired by their work, we first adapt Lin's defect measure theory [Ann. Math. 1999] to such maps building on the partial regularity established by Millot-Pegon-Schikorra [Arch. Ration. Mech. Anal. 2021]. Then apply it to show that the set of singular points of such maps can be quantitatively stratified via a new notion of boundary symmetry with the aid of 
{the celebrated 
harmonic extension method by Caffarelli-Silverstre}. As in that of Naber-Valtorta, developing the necessary 
quantitative regularity estimates, and then combining it with the Reifenberg type theorems and a delicate covering argument allow us to get sharp growth estimates on the volume of tubular neighborhood around singular points and establish 
the rectifiability of each singular stratum. 
\end{abstract}

%{\small
%	\noindent {\bf Keywords}: Biharmonic maps; regularity; compactness; dimension reduction; quantitative strata; Minkowski dimension.
%	\smallskip
%	\newline}

%\bigskip

\maketitle
\tableofcontents

%{\small
%	\keywords {\noindent {\bf Keywords:} 1/2-harmonic maps; regularity theory; quantitative stratification; quantitative symmetry; singular set.}
%\smallskip
%\newline
%\subjclass{\noindent {\bf 2020 Mathematics Subject Classification:} 53C43, 35J48}
%	\tableofcontents}
%\bigskip

\section{Introduction and main results}
In recent years, $1/2$-harmonic maps and the associated {heat} flows have attracted {considerable attention}, see for instance \cite{Da and Ri, Da and Ri 2, HSSW-2022-CPDE, Millot-Pegon-Schikorra-2021-ARMA, Millot-Pegon-2020, Millot-Sire-15, Struwe-2024-APDE} and the references therein. Indeed, it is not only a natural extension of harmonic maps, 
more importantly, but also has close connection
with many different subjects arising from geometry and mathematical physics, such as minimal surfaces with free boundary, free boundary valued harmonic maps, 
{the} Ginzburg-Landau {theory}  for {super}conductivity (see e.g. \cite{Berlyand, Da Lio-book-2018, F. L. T. Riviere, F. Da Lio and  A. Pigati, A. M. Fraser, 
Schikorra-2018-Anal, Struwe-1984}). Meanwhile,
{to handle the} nonlocal nature, many important local mathematical structures and theories have also been extended to this setting, see for instances \cite{DaLio-etal-2022-AMPA, Da and Ri, Da and Ri 2,Mazowiecka-Schikorra-2018}. In particular, the conservation law of Rivi\`ere \cite{Riviere-2007-Invent} and the integration by compensation theory of  \cite{CLMS-1993} 
have been successfully generalized to the nonlocal settings.

In this paper, we study quantitative stratification of the singular set of 1/2-harmonic maps into 
{compact Riemannian} manifolds {without boundary}.  Let $\Omega\subset\mathbb{R}^n$ be a bounded open set and $N\hookrightarrow\mathbb{R}^d$ be a smooth 
{compact Riemannian submanifold without boundary}. Set $\ga_{n}:=\pi^{-\frac{n+1}{2}} {\Gamma(\frac{n+1}{2})}$. Given a measurable map $u:\mathbb{R}^n\ri N$, 
{the}  $1/2$-Dirichlet energy $\mathcal{E}(u,\Om)$ in $\Omega$ is defined as
$$\mathcal{E}(u,\Om):=\frac{\ga_{n}}{4}\iint_{(\mathbb{R}^n\times\mathbb{R}^n)\ba(\Om^c\times\Om^c)}\frac{|u(x)-u(y)|^2}{|x-y|^{n+1}}dxdy. $$
To introduce the formal definition of 1/2-harmonic maps, let us first recall  the spaces
\[\widehat{H}^{1/2}(\Om, \mathbb{R}^d):=\Big\{u\in L_{\loc}^2(\mathbb{R}^n)\ \big|\ \mathcal{E}(u,\Om)<\oo\Big\}\]
and 
\[\widehat{H}^{1/2}(\Om, N):=\Big\{u\in \widehat{H}^{1/2}(\Om, \mathbb{R}^d)\ \big|\  u(x)\in N \text{ for almost every }   x\in\R^n\Big\}.\]
{An} $1/2$-harmonic map {is} defined as {a} critical point of the energy functional $\mathcal{E}(u,\Om).$ 
More precisely,  we have the following {definition}.
\begin{definition}\label{uy}
\begin{enumerate}
\item	{A map} $u\in \widehat{H}^{1/2}(\Om, N)$ is called a weakly $1/2$-harmonic map if
	\begin{eqnarray}\label{1.1}
	\frac{d}{dt}\Big|_{t=0}\mathcal{E}\big(  \pi_N({u+t\va}),\Om\big)=0\qquad \text{ for all }\,\va\in C_0^\infty(\Omega,\mathbb{R}^d).
	\end{eqnarray}
	Here {$\pi_N: N_\delta=\big\{y\in \R^d: {\rm{dist}}(y,N)<\delta\big\}\to N$} denotes the nearest point projection {map from $N_\delta$ to $N$, 
	which is smooth provided $\delta=\delta(N)>0$ is sufficiently small}.	
\item	If a weakly $1/2$-harmonic map $u\in \widehat{H}^{1/2}(\Om, N)$  satisfies {additionally}
	\begin{eqnarray}\label{1.1}
	\frac{d}{dt}\Big|_{t=0}\mathcal{E}\big(u(x+t\psi(x))\big)=0\qquad \text{for all }\,\psi\in C_0^\infty(\Omega,\mathbb{R}^n),
	\end{eqnarray}
	then it is called a stationary $1/2$-harmmonic map.
\item Finally, {a map} $u\in \widehat{H}^{1/2}(\Om, N)$ is called a minimizing $1/2$-harmonic map if
	\begin{eqnarray}\label{1.7}
	\mathcal{E}(u,\Om)\leq \mathcal{E}(v,\Om)
	\end{eqnarray}
	for all $v\in \widehat{H}^{1/2}(\Omega,N)$ with $v=u$ on $\R^d\setminus\Omega$.
	\end{enumerate}
\end{definition}

In {an} interesting work \cite{Da and Ri}, Da Lio and  Rivi\`{e}re showed that weakly $1/2$-harmonic maps from 1-dimensional line $\R$ into the unit sphere $\mathbb{S}^{m-1}$ are smooth. This result was generalized to general closed target $N$ in \cite{Da and Ri 2}. Regarding a $1/2$-harmonic map as the free boundary of a 
conformal minimal harmonic map, F. Da Lio and  A. Pigati proved in \cite[Theorem 1.4]{F. Da Lio and  A. Pigati} that weakly $1/2$-harmonic maps from a one dimensional submanifold (such as the unit circle $S^1$) into $N$ are smooth.

For a map $u:\R^n\to N$ and an open set $\Omega\subset \R^n$, we define its regular set $\reg(u)$ in $\Omega$ by
\[\reg(u)=\Big\{x\in \Omega: u \text{ is  continuous in a neighborhood of }x\Big\},\]
and its singular set $\sing(u)$ of $u$ by
\[
\sing(u):=\Om\backslash \reg(u).
\]
While weakly $1/2$-harmonic maps {smooth} in the (critical) dimension $n=1$. But,  in dimensions $n\ge 2$ (supercritical), weakly $1/2$-harmonic maps may be discontinuous everywhere as observed in \cite[Remark 1.4]{Millot-Pegon-2020} and thus a partial regularity theory for weakly $1/2$-harmonic maps is not possible. On the other hand, making use of the well-known harmonic extension approach of \cite{L. Caffarelli and L.Silvestre}, and also the regularity theory of free boundary-valued harmonic maps (see e.g. \cite{F. Duzaar and K. Steffen, F. Duzaar and K. Steffen 2, R. Hardt and F. H. Lin, Scheven-2006-MZ}), a partial regularity theory for stationary $1/2$-harmonic maps has been developed in \cite{Millot-Pegon-2020, Millot-Sire-15}. 

For our later reference, we  recall  the relevant (partial) regularity  results from \cite[Theorem 4.18 and Remark 4.24]{Millot-Sire-15} and \cite{Millot-Pegon-Schikorra-2021-ARMA} as follows.

\begin{theorem}[\cite{Millot-Sire-15}  \cite{Millot-Pegon-Schikorra-2021-ARMA}]\label{thm3}
	\begin{enumerate}
		\item If $u\in \widehat{H}^{1/2}(\Omega,N)$ is a weakly  $1/2$-harmonic map and $n=1$, then $u$ is smooth in $\Omega$;
		
		\item If $n\ge 2$ and $u\in \widehat{H}^{1/2}(\Omega,N)$ is a stationary  $1/2$-harmonic map, then $u\in C^\infty(\Om\ba \sing(u))$ and $\HH^{n-1}({\sing}(u))=0$ for  $n\geq2$;
		
		\item Furthermore, if $u\in \widehat{H}^{1/2}(\Omega,N)$ is a minimizing  $1/2$-harmonic map, then $\sing(u)$ is a locally finite 
		for $n=2$, and {the Hausdorff dimension} $\dim({\sing}(u))\le n-2$ for  $n\geq3$. 
	\end{enumerate}
	
%	(1) If $u\in \widehat{H}^{1/2}(\Omega,N)$ is a weakly  1/2-harmonic map and $n=1$, then $u$ is smooth in $\Omega$;
%	
% (2)	If $n\ge 2$ and $u\in \widehat{H}^{1/2}(\Omega,N)$ is a stationary  1/2-harmonic map, then $u\in C^\infty(\Om\ba \sing(u))$ and $\HH^{n-1}({\sing}(u))=0$ for  $n\geq2$;
%
% (3) Furthermore, if $u\in \widehat{H}^{1/2}(\Omega,N)$ is a minimizing  1/2-harmonic map, then $\sing(u)$ is a locally finite in $\Om$ for $n=2$, and $\dim({\sing}(u))\le n-2$ for  $n\geq3$. 
\end{theorem}

This theory is analogous to the classical regularity theory of  harmonic maps, see e.g.  R. Schoen and K. Uhlenbeck  \cite{Schoen-Uhlenbeck-1982} for minimizing harmonic maps, 
{H\'elein \cite{Helein1991} and} Evans \cite{Evans-1991} and Bethuel \cite{Bethuel-1993} for stationary harmonic maps, and also the important work of L. Simon  \cite{Simon-book-1996} which  {initiated} the study the structure of singular sets of minimizing harmonic maps, and the seminar work of Lin \cite{Lin-1999-Annals} on the structure of singular sets of stationary harmonic maps. 

Far beyond the partial regularity theory, a major breakthrough was made recently by  Naber and Valtorta \cite{Naber-V-2017}, 
{who} proved the rectifiability (of every stratum) of the singular sets of minimizing and stationary harmonic maps (see Definition \ref{ddd} below for singular stratums). Based on this rectifiablity theory, they further obtained an a priori estimates for harmonic maps, which solves partially the regularity conjecture of Rivi\`ere \cite{Riviere-2007-Invent} on weakly harmonic maps. It is worth mentioning that this work plays an important role in their recent breakthrough \cite{Naber-Val-2025-arXiv} on the energy identity of harmonic maps in supercriticial dimensions. The techniques of \cite{Naber-V-2017}  have been recently extended  to  minimizing  and stationary biharmonic maps (see \cite{Breiner-Lamm-2015, Guo-Jiang-Xiang-Zheng}), 
giving similar rectifiability results on the singular stratums and optimal global regularity on the gradient of those maps. 

In view of the above progresses on harmonic maps,  and taking into account of its potential applications on the energy identity of stationary $1/2$-harmonic maps, it is natural to establish a quantitative stratification theory for singular sets of minimizing and stationary $1/2$-harmonic maps and then obtain  optimal  a priori estimates for gradients of  
$1/2$-harmonic maps. 
 
 To state our results, we introduce, for any given constant $\Lambda>0$, the subset
\begin{equation}\label{eq: bounded energy space}
\widehat{H}^{1/2}_\La(\Om, N)=\Big\{u\in \widehat{H}^{1/2}(\Om, N)\ \big|\  \mathcal{E}(u,\Om)<\La\Big\}.
\end{equation}
In this article, we shall always consider the problem in a ball
\[D_R(x_0)=\big\{x\in \mathbb{R}^n\ \big|\ |x-x_0|<R\big\}.\]
For simplicity, denote $D_R(0)$ by $D_R$ when the  ball is centered at the origin. For any set $A\subset \R^n$, we define the $r$-tubular neighborhood of $A$ in $\R^n$ as
\[ D_r(A)=\big\{x\in \R^n\ \big|\  \dist(x,A)\le r \big\}.\]

{Before stating our results, let $S_{\ep, r}^k(u)$ denote $k$-th quantitative singular stratum of $u$, defined by Definition 3.5 below,  
and 
$$S_\ep^k(u)=\bigcap_{r>0} S^k_{\ep, r}(u), \ \ S^k(u)=\bigcup_{\ep>0} S_\ep^k(u).$$}
Our first main result gives sharp growth estimates on the volume of tubular neighborhood around singular points and establishes the expected rectifiability of each stratum. 

\begin{theorem}[Stratification of stationary $1/2$-harmonic maps]\label{thm80}
Given $\Lambda>0$, let  $u\in \widehat{H}^{1/2}_\La(D_8, N)$ be a stationary $1/2$-harmonic map. Then, for each $\ep>0$ there exists $C_\ep=C_\ep(n,N,\Lambda,\ep)$ such that
\begin{align}\label{J}
\mbox{\rm Vol}(D_r(S^k_{\ep,r}(u))\cap D_1)\leq C_\ep r^{n-k},\qquad \text{for all }\,r\in(0,1].
\end{align}
Consequently, for all $r\in(0,1]$,
 \begin{align}\label{Jk}
 \mbox{\rm Vol}(D_r(S^k_{\ep}(u))\cap D_1)\leq C_\ep r^{n-k}.
 \end{align}
 Moreover, for each $k$, $S^k_\ep(u)$ and  $S^k(u)$ are $k$-rectifiable and  upper Ahlfors $k$-regular,  and  for $\HH^k$-almost every $x\in S^k(u)$, there exists a unique $k$-plane $V^k\subset \R^n$ such that every tangent map of $u$ at $x$ is $k$-symmetric with respect to $V^k$.
\end{theorem}
Recall that a  subset $A\subset \R^n$ is said to be upper Ahlfors $k$-regular, if there is a constant $M>0$ such that
\[\cal{H}^k(A\cap D_r(x))\le Mr^k\qquad \text{for all }\,x\in A \text{ and } 0<r<\text{diam}(A). \]

As an application of the volume estimate  in Theorem \ref{thm80}, we obtain optimal first order regularity estimate for minimizing 
$1/2$-harmonic maps. To state it, following Naber and Valtorta \cite{Naber-V-2017}, we introduce the notion of (scale invariant) regularity scale.

\begin{definition}[Regularity scale]\label{pppv}
Given a map $u\colon D_2(0)\ri N$  and $x\in D_1(0)$, we define the regularity scale $r_u(x)$ of $u$ at $x$ by
$$r_u(x):=\max\Big\{0\leq r\leq 1\ \big|\ \sup_{D_r(x)}\left(r|\nabla u(x)|\right)\leq1\Big\}.$$
\end{definition}

Our second main theorem provides a regularity estimate on the gradient of a minimizing $1/2$-harmonic map,
{as well as} an estimate on its regularity scale.
%By elliptic regularity theory for fractional Laplacian, if $u$ is a 1/2-harmonic mapping and $r_u(x)=r>0$, then $u\in C^\infty(D_{{r}/{2}}(x))$.

\begin{theorem}[Regularity estimates on minimizing $1/2$-harmonic maps]\label{thm84} There exists a positive constant  $C=C(n,N,\La)$ such that,
for any minimizing $1/2$-harmonic map $u\in \widehat{H}^{1/2}_\La(D_2(0), N)$, {for $0<r<1$} there holds
\begin{equation}\label{emhm}
\quad {\rm Vol}\big(\big\{x\in D_1(0) \big|\ |\nabla u(x)|>r^{-1}\big\}\big)\leq{\rm Vol}\big(\big\{x\in D_1(0) \big|\ r_u(x)<r\big\}\big)\leq Cr^2.
\end{equation}
In particular, we have $\na u\in L^{2,\wq}(D_1(0))$, the weak $L^2$-space on $D_1(0)$.
\end{theorem}

As was already observed in \cite[Theorem 1.4]{Millot-Pegon-2020}, the map $u\colon D_1^2\to S^1$ with $u(x)=\frac{x}{|x|}$ is a minimizing $1/2$-harmonic map 
with $|\nabla u(x)|\sim 1/r$, which belongs to $L^{2,\wq}(D_1)$ but not in $L^p$ for any  $p\ge 2$. This shows the sharpness of Theorem \ref{thm84} at least when $n=2$. 

As in the case of  harmonic maps, see e.g. \cite{Cheeger-Naber-2013-CPAM, Naber-V-2017}, the above regularity results can be  improved for stationary $1/2$-harmonic maps  under extra geometrical assumptions on the target similar to the case of stationary harmonic maps. This observation  first appeared in the celebrated work of F.H. Lin \cite{Lin-1999-Annals} 
dealing harmonic maps and then was generalized to bi-harmonic maps by Scheven \cite{Scheven-2008-ACV}  and Breiner-Lamm \cite{Breiner-Lamm-2015}. More precisely, we shall prove that a stationary/minimizing $1/2$-harmonic maps enjoy higher regularity if there are no non-constant smooth  ``$1/2$-harmonic spheres" of certain dimensions, that is, {non-constant smooth} $1/2$-harmonic maps from $\R^{\ell+1}\backslash \{0\}$ to $N$ {homogeneous of degree zero}. 
%But to deduce such regularity, we need first established a compactness result  for stationary 1/2-harmonic maps, see Theorem \ref{thm: compactness of stationry HM}. Then, following the approach of 
Similar to the case of harmonic maps (see Naber and Valtorta \cite{Naber-V-2017,Naber-V-2018}), we have the following improved regularity estimates for 
$1/2$-harmonic maps.

\begin{theorem}[Improved estimates on $1/2$-harmonic maps]\label{thm85}
Let $u\in \widehat{H}_\Lambda^{1/2}(D_8(0),N)$ be a stationary (minimizing resp.) $1/2$-harmonic map. Assume that for some $k\geq1$ there exists no  nonconstant smooth 0-homogeneous stationary (minimizing resp.)  1/2-harmonic maps $\R^{\ell+1}\backslash\{0\}\to N$ for all $1\leq \ell\leq k$. Then there exists a constant $C=C(m,K_N,\La)>0$ such that
\[
{\rm Vol}\big(\big\{x\in B_1(0):r|\nabla u|>1\big\}\big)\leq{\rm Vol}\big(\big\{x\in B_1(0):r_u(x)<r\big\}\big)\leq Cr^{2+k}.
\]
In particular, both $|\na u|$ and $r_u$ have uniform bounds in $L^{2+k,\infty}(B_1(0))$.
\end{theorem}

Before ending this section, let us briefly discuss the proofs of these theorems. As in the case of harmonic maps, Theorem \ref{thm80} would be the key result, and after it, all the later regularity theorems follow routinely. For the proof of Theorem \ref{thm80}, we will adapt the approach of A. Naber and D. Valtorta \cite{Naber-V-2017,Naber-V-2018}: 
\begin{enumerate}
	\item Establish the necessary quantitative $\epsilon$-regularity theorem.
	\item Combine it with the Discrete Reifenberg Theorem proved by \cite{Naber-V-2017}, together with a delicate refined covering argument from \cite{Naber-V-2018}, to prove the improved volume estimate \eqref{J} and the rectifiability of the quantitative stratum. 
\end{enumerate}
%establish the necessary quantitative $\epsilon$-regularity theorems and then combine it with the Discrete Reifenberg Theorem proved in \cite{Naber-V-2017}, together with a delicate refined covering argument from \cite{Naber-V-2018}, to prove the improved volume estimate \eqref{J} and rectifiability of the quantitative stratum.  
In case of stationary harmonic or biharmonic maps, the proofs of these theorems rely heavily on the monotonicity formula for rescaled energies. \textbf{In our case, it is not known whether such kind of monotonicity  formula  exists.} To overcome this difficulty, we use the well-known idea of Caffarelli-Silvestre in \cite{L. Caffarelli and L.Silvestre} and consider the harmonic extension $u^e$ of a stationary $1/2$-harmonic map $u$ (see \eqref{extension}). Then $u^e$ is harmonic function with Dirichlet boundary value $u$. The stationarity of $u$ implies $u^e$ is also stationary by \cite[Proposition 2.15]{Millot-Pegon-Schikorra-2021-ARMA} and the rescaled energy of $u^e$ enjoys a monotonicity formula by Theorem \ref{thm1}. Since $u^e$ is smooth in the extension direction $z$ by \eqref{extension}, the singular set of $u^e$ must lie on the original space $\R^n$ and hence coincides with the singular set of $u.$ \textbf{However, we do not know whether a similar relation holds for the quantitative stratification of the singular set of $u$ and $u^e$.} 
%At least we have not shown this yet. 
In order to overcome this difficulty, we introduce a new definition of quantitative stratification for the singular set of $u^e$, whose quantitative symmetry lies on the flat boundary of a half ball in $\R^{n+1}_+$ (see Definition \ref{def:quantitative symmetry}). Using this as the definition of quantitative stratification for the singular set of $u$, 
the quantitative symmetry of $u$ and $u^e$ coincide and the quantitative analysis can be 
achieved through the monotonicity formula of $u^e$.

The paper is arranged as follows. In Section \ref{sec:partial regularity},  we first recall some preliminaries and then establish an analog of Lin's defect measure theory for stationary 
$1/2$-harmonic maps, from which we obtain the classical stratification of singular sets and weak compactness theory for stationary $1/2$-harmonic maps, see Theorem \ref{thm: compactness of stationary HM}. In Section \ref{sec:Quantitative stratification}, we introduce a new quantitative symmetry and the quantitative stratification of singular set for stationary $1/2$-harmonic maps associated to this new quantitative symmetry. Then we derive some geometric properties of each stratum. In Section \ref{sec:Reifenberg theorems}, we recall the rectifiable-Reifenberg theorem of \cite{Naber-V-2017} and estimate Jones' number. In Section \ref{sec:covering lemma}, we prove a main covering lemma based on the delicate covering argument of Naber-Valtorta \cite{Naber-V-2018}. In the last two sections, we prove the main results of this paper.

\textbf{Notations.} Throughout the paper, we will use the following notations:
\begin{itemize}
	\item $\mathbf{x} = (x,z)$ with $x\in \mathbb{R}^n$ and $z\in\mathbb{R}$ denotes point in $\mathbb{R}^{n+1}$;
	\item $\mathbb{R}^n$ is often identified with $\pa\mathbb{R}^{n+1}_+=\mathbb{R}^n\times\{0\}.$  A set $A\su\mathbb{R}^n$ is also  identified with $A\times\{0\}\subset \pa\mathbb{R}^{n+1}_+.$
	\item   $B_r(\mathbf{x} )$ denotes the
	open ball in $\mathbb{R}^{n+1}$ of radius $r$ centered at $\mathbf{x} = (x,z)$;
	\item   $B^+_r(\mathbf{x})$ usually denotes the half	open ball in $\mathbb{R}^{n+1}_+$ of radius $r$ centered at $\mathbf{x} = (x,0)$;
	\item $D_r(x)$ denote the the open ball/disk  in $\mathbb{R}^n$ centered at $x$.
	\item For any set $A\subset \R^n$, we denote the $r$-tubular neighborhood of $A$ in $\R^n$ by
	\[ D_r(A)=\{x\in \R^n: \dist(x,A)\le r  \}.\]
\end{itemize}

  If $x = 0$, we simply write $B_r(x)$ and $D_r(x)$ as $B_r$ and $D_r$,
respectively. For an arbitrary set $G\su\mathbb{R}^{n+1}$, we write
$$G^+=G\cap\mathbb{R}^{n+1}_+\ \ \ \ \mbox{and}\ \ \ \ \pa^+G=\pa G\cap\mathbb{R}^{n+1}_+.$$
If $G\su\mathbb{R}^{n+1}_+$  is a bounded open set, we shall say that $G$ is admissible if
 $\pa G$ is Lipschitz regular; and  the (relative) open set $\pa^0 G\su\mathbb{R}^{n+1}_+$ defined by
$$\pa^0G=\Big\{\mathbf{x}\in\pa G\cap\pa \mathbb{R}^{n+1}_+:B_r^+(\mathbf{x})\su G\ \ \mbox{for some }\ r>0\Big\},$$
is nonempty and has Lipschitz boundary, such that $\pa G=\pa^+G\cup\overline{\pa^0G}$.

\section{Partial regularity and defect measure}\label{sec:partial regularity}

In this section, we collect some standard facts for $1/2$-harmonic maps, where the proof can be found in \cite{Millot-Sire-15,Millot-Pegon-Schikorra-2021-ARMA}.
First, we introduce some related function spaces. The Sobolev-Slobodeckij space $H^{1/2}(\Om)$ consists of functions $u\in L^2(\Om)$ satisfying
\[[u]^2_{{H}^{1/2}(\Om)} :=\frac{\ga_{n}}{2}\iint_{\Om\times\Om}\frac{|u(x)-u(y)|^2}{|x-y|^{n+1}}dxdy<\oo.\]
It is a separable Hilbert space with norm given by$\|\cdot\|^2_{H^{1/2}(\Om)}:=\|\cdot\|^2_{L^2(\Om)}+[\cdot]^2_{{H^{1/2}(\Om)}}$. The space $H^{1/2}_{\loc}(\Om)$ consists of functions which belongs to $H^{1/2}(\Om')$ for any  open subset {$\Om'\Subset\Om$}.   
The linear subspace $H_{00}^{1/2}(\Om)\su H^{1/2}(\mathbb{R}^n)$ is defined by
$$H_{00}^{1/2}(\Om):=\Big\{u\in  H^{1/2}(\mathbb{R}^n)\big|\ u=0\ a.e. \ x\in\mathbb{R}^n\ba\Om\Big\}.$$
Endowed with the induced norm, $H_{00}^{1/2}(\Om)$ is also a Hilbert space and
\[[u]^2_{{H}^{1/2}(\mathbb{R}^n)} =2\mathcal{E}(u,\Om) \qquad \forall\, u\in H_{00}^{1/2}(\Om).\]
If $\Om$ is a bounded open domain with sufficiently smooth boundary (for example, $\pa\Om$ is Lipschitz regular), then
$$H_{00}^{1/2}(\Om)=\overline{C_{c}^{\wq}(\Om)}^{\|\cdot\|_{H^{1/2}(\R^{n})}}.$$
%(see \cite[Theorem 1.4.2.2]{Gulliver-Jost-1987}). 
The topological dual space of $H_{00}^{1/2}(\Om)$ is denoted by $H^{-1/2}(\Om),$ that is,
\[
H^{-1/2}(\Om)=\left(H_{00}^{1/2}(\Om)\right)^{\ast}.
\]
\begin{lemma}[{\cite[Lemma 2.1]{Millot-Pegon-Schikorra-2021-ARMA}}]\label{lem:2.1}
	Let $D_\rho(x_0)\su\Om$. There exists a
	constant $C_\rho=C(\rho,n)>0$ such that
	$$\int_{\mathbb{R}^n}\frac{|u(x)|^2}{(|x-x_0|+1)^{n+1}}dx\leq C_\rho\big(\mathcal{E}(u,D_\rho(x_0))+\|u\|^2_{L^2(D_\rho(x_0))}\big)$$
	for every $u\in\widehat{H}^{1/2}(\Om).$
\end{lemma}

By Lemma \ref{lem:2.1}, $\widehat{H}^{1/2}(\Om)$
is a Hilbert space for the scalar product induced by the norm 
$$\|u\|_{\widehat{H}^{1/2}(\Om)}=\big(\|u\|_{L^2(\Om)}+\mathcal{E}(u,\Om)\big)^{1/2}.$$
Given an open set $\Om\su\mathbb{R}^n$, the fractional Laplacian $(-\De)^{1/2}$ in $\Om$ is defined
as the continuous linear operator $(-\De)^{1/2}:\widehat{H}^{1/2}(\Om)\ri(\widehat{H}^{1/2})'$ induced by the
quadratic form $\mathcal{E}(\cdot,\Om)$. In other words, the weak form of the fractional Laplacian
$(-\De)^{1/2}u$ of a given function $u\in\widehat{H}^{1/2}(\Om)$ is defined through its action on $ \widehat{H}^{1/2}(\Om)$ by
\begin{equation}
\big\langle(-\De)^{1/2}u,\va\big\rangle_\Om=\frac{\ga_{n}}{2}\iint_{(\mathbb{R}^n\times\mathbb{R}^n)\ba(\Om^c\times\Om^c)}\frac{(u(x)-u(y))(\va(x)-\va(y))}{|x-y|^{n+1}}dxdy.
\end{equation}
Notice that the restriction of the linear form $(-\De)^{1/2}u$ to the subspace $H^{1/2}_{00}(\Om)$ belongs
to $H^{-1/2}(\Om)$ with the estimate $$\big\|(-\De)^{1/2}u\big\|^2_{H^{-1/2}(\Om)}\leq 2\mathcal{E}(u,\Om).$$
We say that $u_{i}\wto u$ in $\widehat{H}^{1/2}(B_{1})$, if
%\[\begin{aligned} & \lim_{i\to\wq}\iint_{(\R^{n}\times\R^{n})\backslash\Om^{c}\times\Om^{c}}\frac{(u_{i}(x)-u_{i}(y))\cdot(\va(x)-\va(y))}{|x-y|^{n+1}}dxdy\\& \qquad=\iint_{(\R^{n}\times\R^{n})\backslash\Om^{c}\times\Om^{c}}\frac{(u(x)-u(y))\cdot(\va(x)-\va(y))}{|x-y|^{n+1}}dxdy\end{aligned}\quad\quad\forall\,\va\in C_{c}^{\wq}(B^{n}).\]
\[
\big\langle(-\De)^{1/2}u_{i},\va\big\rangle_{\Om}\to\big\langle(-\De)^{1/2}u,\va\big\rangle_{\Om}\quad\quad\text{for all }\,\va\in C_{c}^{\wq}(B_1).
\]
\subsection{Monotonicity formula and partial regularity}

The following alternative {description} of weakly $1/2$-harmonic maps are well-known.
\begin{proposition}[{\cite[Remark 4.24 ]{Millot-Sire-15}}]\label{Millot-Sire}
	Let $\Om\subset\mathbb{R}^n$ be a bounded open set.
	A map $u\in \widehat{H}^{1/2}(\Omega,N)$ is a weakly $1/2$-harmonic map in $\Om$ if and only if
	\begin{equation*}
	\big\langle(-\De)^{1/2}u,\va\big\rangle_{\Om}=0
	\end{equation*}for every $\va\in H^{1/2}_{00}(\Om,\mathbb{R}^d)$ satisfying  
	{$\va(x)\in T_{u(x)} N$} for almost everywhere $x\in \Om.$
\end{proposition}

In the case $N=\mathbb{S}^{d-1}$, we can write down the equation for $1/2$-harmonic maps:
\begin{equation*}
(-\De)^{1/2}u(x)=\Big(\frac{\ga_{n}}{2}\int_{\mathbb{R}^n}\frac{|u(x)-u(y)|^2}{|x-y|^{n+1}}dy\Big)u(x)\qquad \text{in } \mathcal{D}'(\Om).
\end{equation*}
\begin{definition}
	Let $\Om\subset\mathbb{R}^n$ be a bounded open set. Given a map $u\in \widehat{H}^{1/2}(\Omega,\mathbb{R}^d)$
	and a vector field $X\in C^1_0(\Om; \mathbb{R}^n)$, the first (inner)
	variation of $\mathcal{E}(\cdot,\Om)$ at $u$ and evaluated at $X$ is defined as
	$$\de\mathcal{E}(u,\Om)[X]:=\frac{d}{dt}\Big|_{t=0}\mathcal{E}(u\circ\phi_{-t},\Om),$$
	where $\{\phi_t\}_{t\in\mathbb{R}^n}$ denotes the integral flow on $\mathbb{R}^n$ generated by $X$, that is, for every
	$x\in\mathbb{R}^n$, the map $t\mapsto\phi_t(x)$  is defined as the unique solution of the ordinary
	differential equation
	\begin{equation*}
	\begin{cases}\frac{d}{dt}\phi_t(x)=X(\phi_t(x)),\\
	\phi_0(x)=x.
	\end{cases}
	\end{equation*}
\end{definition}

Given a measurable function $u$ defined on $\mathbb{R}^n,$ we shall denote by $u^e$ its harmonic extension to the half-space $\mathbb{R}^{n+1}_+$. That is,
\begin{equation}\label{extension}
u^e(x,z):=\ga_{n}\int_{\mathbb{R}^n}\frac{zu(y)}{(|x-y|^2+z^2)^\frac{{n+1}}{2}}dy.\end{equation}
Then $u^e$ solves
\begin{eqnarray}\label{harmonic function}
\begin{cases}
\De u^e=0\ &\mbox{in}\ \mathbb{R}^{n+1}_+,\\
u^e=u\ &\mbox{on}\ \pa\mathbb{R}^{n+1}_+.
\end{cases}
\end{eqnarray}
For $\mathbf{x}_0=(x_0,0)\in\mathbb{R}^n\times\{0\}$, we define the Dirichlet energy $\mathbf{E}(u^e,B_r^+(\mathbf{\mathbf{x}_0}))$ in $B_r^+(\mathbf{\mathbf{x}_0})$ by
$$\mathbf{E}(u^e,B_r^+(\mathbf{\mathbf{x}_0})):=\frac{1}{2}\int_{B_r^+(\mathbf{x}_0)}|\na u^e|^2d\mathbf{x}.$$
%with $\mathbf{x}_0=(x_0,0)\in\mathbb{R}^n\times\{0\}.$

The following monotonicity formula plays a central role in the regularity issue of stationary $1/2$-harmonic maps.
\begin{theorem}[Monotonicity formula,~\cite{Millot-Pegon-Schikorra-2021-ARMA,Scheven-2006-MZ}]\label{thm1}
	Let $\Om\su\mathbb{R}^n$ be a bounded open set. If $u\in \widehat{H}^{1/2}(\Omega,N)$ is a stationary $1/2$-harmonic map in $\Om$, then for every $\mathbf{x}_0=(x_0,0)\in\Om\times\{0\},$  the ``density function''
		$$r\in(0,dist(x_0,\Om^c))\mapsto\Theta(u^e,B_r^+(\mathbf{\mathbf{x}_0})):=\frac{1}{r^{n-1}}\mathbf{E}(u^e,B_r^+(\mathbf{\mathbf{x}_0}))$$
	is {monotonically} nondecreasing. Moreover, for every $0<\rho<r<\dist(x_0,\Om^c)$, it holds
	\begin{eqnarray}\label{1.8}
	\Theta(u^e,B_r^+(\mathbf{\mathbf{x}_0}))-\Theta(u^e,B_\rho^+(\mathbf{\mathbf{x}_0}))= \int_{B_r^+(\mathbf{x}_0)\ba B_\rho^+(\mathbf{x}_0)}\frac{|(\mathbf{x}-\mathbf{x}_0)\cdot\na u^e|^2}{|\mathbf{x}-\mathbf{x}_0|^{n+1}}d\mathbf{x}.
	\end{eqnarray}	
	%for every $0<\rho<r<dist(x_0,\Om^c).$
	%In particular, upon redefining $\Phi_u(a,\cdot)$ on a set of measure zero, we may assume \eqref{1.8} holds for all $0<\rho<r\leq R/4$ and thus $r\mapsto\Phi_u(a,r)$ is monotonically nondecreasing.
\end{theorem}
Using this monotonicity formula, we may define the density function $\Xi(u,\cdot): \Omega\to [0,\infty)$ of a stationary $1/2$-harmonic map $u\colon \Omega\to N$ by
 \begin{equation}\label{eq: density of map}
\Xi(u,x_0):=\lim_{r\searrow 0} \Theta(u^e,B_r^+( {\mathbf{x}_0})),\qquad x_0\in \Omega.
 \end{equation}
Letting $\rho\to 0$ in \eqref{1.8}, we obtain that for every $0<r<\dist(x_0,\pa\Omega)$ and $\mathbf{x}_0=(x_0,0)\in\Om\times\{0\}$,
\[ \Theta(u^e,B_r^+( {\mathbf{x}_0}))=\Xi(u,x_0)+ \int_{B_r^+(\mathbf{x}_0)} \frac{|(\mathbf{x}-\mathbf{x}_0)\cdot\na u^e|^2}{|\mathbf{x}-\mathbf{x}_0|^{n+1}}d\mathbf{x}
\]

Next, we define the singular set of a $1/2$-harmonic map as follows.
\begin{definition}\label{singset}
	For any $1/2$-harmonic map $u\in \widehat{H}^{1/2}(\Omega,N)$, we define its regular set $\reg(u)$ by
	\[\reg(u)=\big\{x\in \Omega: u \text{ is  continuous in a neighborhood of }x\big\},\]
	and the singular set $\sing(u)$ of $u$ by
	\[
	\sing(u):=\Om\backslash \reg(u).
	\]
\end{definition}

\begin{remark}\label{singular set}
	{\rm Since $u^e$ is smooth in $\R^{n}\times (0,\infty)$ by \eqref{extension}, we have
	$$\sing(u^e)=\sing(u)\times\{0\}.$$
	This allows us to study $\sing(u)$ through the boundary singularity of $u^e$.}
\end{remark}

Thanks to the monotonicity formula, the following partial regularity result holds, 
{see the proof by} \cite[Theorem 4.18 and Remark 4.24]{Millot-Sire-15} and  \cite[Theorem
5.1]{Millot-Pegon-Schikorra-2021-ARMA}.   
%\textbf{(Need to check whether the following theorem holds for general targets)}

\begin{theorem}[Partial regularity] \label{thm: partial regularity of MPS}
	There exist $\varepsilon_{1}=\varepsilon_{1}(n)>0$ and $\kappa_{2}=\kappa_{2}(n)\in(0,1)$
	such that the following holds. Let $u\in\widehat{H}^{1/2}\left(D_{2R};N\right)$
	be a weakly $1/2$-harmonic map in $D_{2R}$ such that {for every $\mathbf{x}\in\partial^{0}B_{2R}^{+}$}, the function
	$\boldsymbol{\Theta}\left(u^{\mathrm{e}},\mathbf{x},r\right)$ 
	is {monotonically nondecreasing for $r\in(0,2R-|\mathbf{x}|)$}.
	If
	\[
	\boldsymbol{\Theta}\left(u^{\mathrm{e}},0,R\right)\leqq\varepsilon_{1},
	\]
	then $u\in C^{0,1}\left(D_{\kappa_{2}R}\right)$ and
	\[
	R^{2}\|\nabla u\|_{L^{\infty}\left(D_{\kappa_{2}R}\right)}^{2}\leqq C_{2}\boldsymbol{\Theta}\left(u^{\mathrm{e}},0,R\right),
	\]
	for a constant $C_{2}=C_{2}(n)>0$.\end{theorem}

For the purpose {later}, we restate Theorem \ref{thm: partial regularity of MPS} in terms of the regularity scale of $u$.

\begin{proposition}\label{prop: reg-scale-estimate} Suppose $u\in\widehat{H}^{1/2}(B_{2},N)$ is
	a stationary $1/2$-harmonic map. There exist $\varepsilon_{2}=\varepsilon_{2}(n)>0$
	and $\kappa_{2}=\kappa_{2}(n)\in(0,1)$ such that if ${\cal E}(u,B_{2})<\varepsilon_{2}$,
	then
	\[
	r_{u}(0)\ge\kappa_{2}.
	\]
\end{proposition}

To prove this proposition, we need the following lemma.
\begin{lemma}[{\cite[Lemma 2.9]{Millot-Pegon-Schikorra-2021-ARMA}}]\label{lem: 2.9}
	Let $u\in\widehat{H}^{1/2}(\Om)$.  There exists a constant $C=C(n)>0$
	such that for any $D_{3r}(x)\subset\Om$ with $\mathbf{x}=(x,0)$,
	\[
	\mathbf{E}(u^e,B_{r}^{+}(\mathbf{x}))\le C{\cal E}(u,D_{2r}(x)),
	\]
	and
		\[{\|u^e\|^2_{L^2(B_{r}^{+}(\mathbf{x}))}}\le C\Big( r^2 {\cal E}(u,D_{2r}(x))+r{\|u\|^2_{L^2(D_{2r}(x))}} \Big).\]
{In particular,}
	\[
	\boldsymbol{\Theta}\left(u^e,B_{r}^{+}(\mathbf{x})\right)\le C_{1}(n)\boldsymbol{\theta}\left(u,D_{2r}(x)\right)
	\]
	for some $C_{1}=C_{1}(n)>0$, where $\boldsymbol{\theta}\left(u,D_{r}(x)\right)=\frac{1}{r^{n-1}}\mathcal{E}(u,D_r(x_0)).$
\end{lemma}

%we can prove Proposition  \ref{prop: reg-scale-estimate} now.

\begin{proof}[Proof of Proposition \ref{prop: reg-scale-estimate}]
	Choose $\varepsilon_{2}\le\min\{\varepsilon_{1}/C_{1},1/(C_{1}C_{2})\}$
	such that ${\cal E}_{s}(u,B_{2})<\varepsilon_{2}$ implies
	\[
	\boldsymbol{\Theta}\left(u^e,B_{1}^{+}(\mathbf{x_0})\right)\le C_{1}\boldsymbol{\theta}\left(u, D_2(x_{0})\right)\le C_{1}\varepsilon_{2}\le\varepsilon_{1}.
	\]
	Then Theorem \ref{thm: partial regularity of MPS} yields
	\[
	\kappa_{2}^{2}\|\nabla u\|_{L^{\infty}\left(D_{\kappa_{2}}\right)}^{2}\leqq C_{2}\boldsymbol{\Theta}\left(u^{\mathrm{e}},B_1^+\right)\kappa_{2}^{2}\le C_{2}C_{1}\varepsilon_{2}\le1.
	\]
	The proof is complete in view of the definition of regularity scale.
\end{proof}

Another {important}  consequence {of }Lemma \ref{lem: 2.9} is follows.
\begin{corollary}
	Suppose $\{u_i\}$ is a bounded sequence in $\widehat{H}^{1/2}(D_4)$. Then there exists a subsequence of the extended functions $\{u_i^e\}$ which converges weakly in $H^1(B_2^+)$ and strongly in $L^2(B_2^+)$.
\end{corollary}

%
%\begin{lemma} (\cite[Lemma 2.19]{Millot-Pegon-Schikorra-2021-ARMA})
%	Let $\Omega\subseteq\mathbb{R}^{n}$ be an open set, and $u\in\widehat{H}^{s}\left(\Omega;\mathbb{R}^{d}\right)\cap L^{\infty}\left(\mathbb{R}^{n}\right)$
%	be such that $\|u\|_{L^{\infty}\left(\mathbb{R}^{n}\right)}\leqq M$.
%	For every $\varepsilon>0$, there exists $\delta=\delta(n,s,M,\varepsilon)>0$
%	and $\alpha=\alpha(n,s,M,\varepsilon)\in(0,1/4]$ such that
%	\[
%	\boldsymbol{\Theta}_{s}\left(u^{\mathrm{e}},\mathbf{x}_{0},r\right)\leqq\delta\quad\Longrightarrow\quad\boldsymbol{\theta}_{s}\left(u,x_{0},\alpha r\right)\leqq\varepsilon
%	\]
%	for every $\mathbf{x}_{0}=\left(x_{0},0\right)\in\Omega\times\{0\}$
%	and $r>0$ satisfying $\bar{D}_{r}\left(x_{0}\right)\subseteq\Omega$.\end{lemma}

\subsection{Defect measure theory and  stratification of singular sets}
In this subsection, we discuss the defect measure theory and stratification for singular sets of  stationary $1/2$-harmonic maps. To the best of our knowledge, only part of the defect measure theory have been studied in \cite{Millot-Sire-15} and \cite{Millot-Pegon-Schikorra-2021-ARMA}. In the case of stationary  harmonic maps, this was first considered in the seminal paper of Lin \cite{Lin-1999-Annals},  {which was later} extended to stationary biharmonic maps by Scheven in \cite{Scheven-2008-ACV,Scheven-2009-Poincare}.

We start with the notion of tangent maps (see e.g.  \cite{Millot-Pegon-Schikorra-2021-ARMA,Millot-Sire-15}). Assume $u\in \widehat{H}_\Lambda^{1/2}(\Omega,N)$ is a stationary 
$1/2$-harmonic map. Let $x_0\in\Om\su\mathbb{R}^n$ and $ \mathbf{x}_0=(x_0,0)$. By Lemma \ref{lem: 2.9} there holds
\begin{equation}\label{uniformly bound}
\mathbf{E}(u^e,B_r^+(\mathbf{x}_0))\leq C(n)\mathcal{E}(u,D_{2r}(x_0)),\qquad \forall\,D_{2r}(x_0)\subset\Omega.\end{equation}
Let  {$d_0=\frac14\dist(x_0,\pa\Om)$}. Then, by the monotonicity formula in Theorem \ref{thm1}, for any $\rho>0$ and  $r<{d_0}/\rho$, there holds
$$\Theta(u^e,B_{r\rho}^+(\mathbf{\mathbf{x}_0}))\leq\Theta(u^e,B^+_{d_0}(\mathbf{\mathbf{x}_0}))\leq C(n)\mathcal{E}(u,D_{2{d_0}}(x_0)).$$
 This implies that
$$\Theta(u^e_{\mathbf{x}_0,r},B_{\rho}^+ )\leq C\mathcal{E}(u,D_{2{d_0}}(x_0))\le C(n)\mathcal{E}(u,\Omega),\qquad \forall\,0<r<{d_0}/\rho.$$
Here $u^e_{\mathbf{x}_0,r}=(u_{x_0,r})^e$ and $u_{x_0,r}=u(x_0+r\cdot)$. Since $u$ is $N$-valued, we conclude that $\{u^e_{\mathbf{x}_0,r}: 0<r<{d_0}/\rho \}$ is uniformly bounded in $H^1(B_{\rho}^+)$.
By \cite[Lemma 2.8]{Millot-Pegon-Schikorra-2021-ARMA}, there exists a constant $C(n)>0$ such that
$$\rho^{1-n}[u_{x_0,r}]^2_{H^{1/2}(D_{\rho/2})}\leq C(n) \Theta(u^e_{\mathbf{x}_0,r},B_{\rho}^+ )\le C\mathcal{E}(u,\Omega), \qquad \forall\,0<r<d/\rho.$$
Thus $\{u_{x_0,r}: 0<r<{d_0}/\rho\}$ is uniformly bounded in $H^{1/2}(D_{\rho/2})$ for  fixed $\rho>0$. Then, as in the proof of \cite[Lemma 2.19]{Millot-Pegon-Schikorra-2021-ARMA}, we find that
$$\mathcal{E}(u_{x_0,r},D_{\frac{\rho}{2}} )\leq C\Big([u_{x_0,r}]^2_{H^{1/2}(D_{\rho})}+\iint_{D_{\frac{\rho}{2}}\times D_{\rho}^c}\frac{dxdy}{|x-y|^{n+1}}\Big)\leq C\rho^{n-1}\Big(\mathcal{E}(u,\Om)+1 \Big)$$
for all $0<r<{d_0}/\rho$.  Therefore, for  any sequence  $r_j\to 0$, there exists a subsequence (still denoted as $r_j$) and a map $v\in \widehat{H}^{1/2}_\loc(\mathbb{R}^n)$ such that for any $\rho>0$, it holds
$$u_{x_0,r_j}\rightharpoonup v\quad \mbox{ in } \widehat{H}^{1/2}(D_\rho)\quad\text{and}\quad u^e_{\mathbf{x}_0,r_j}\rightharpoonup v^e\quad \mbox{ in } H^1(B^+_\rho).$$
%for any $\rho>0$.

Similarly,  {we can show} that, for any $x_0\in \Omega$, if $\{u_i\}\subset \widehat{H}^{1/2}(\Omega,N)$ is 
a  uniformly bounded {family of} stationary $1/2$-harmonic maps, then for any given $\rho>0$ and $r\to 0$, the scaled sequence  $\{(u_i)_{x_0,r}\}$ and $\{(u_i^e)_{\mathbf{x}_0,r}\}$ are bounded in $\widehat{H}^{1/2}(D_{\rho},N)$ and $H^1(B^+_{\rho},\R^d)\cap L^\wq(\Om,\R^d)$ respectively.

\begin{definition}
	We say that $v\in \widehat{H}_{\loc}^{1/2}(\mathbb{R}^n,N)$ is a tangent map of a stationary $1/2$-harmonic map $u\in  \widehat{H}^{1/2}(\Omega, N)$ at the point $a\in \Omega$, if there is a sequence $r_j\to 0$ such that $u_{a,r_j}\equiv u(a+r_j\cdot)\wto v$ in $\widehat{H}_{\loc}^{1/2}(\mathbb{R}^n,N)$.
\end{definition}

To study tangent maps,  we first consider a general case, which will be important for later {purposes}.

\begin{proposition}[Defect measures]\label{prop: defect measure}
	Suppose that $\{u_{i}\}_{i\ge1}\subset \widehat{H}^{1/2}(D_{1},N)$
	is a sequence of stationary $1/2$-harmonic maps satisfying
	\[
	u_{i}\wto u\quad\text{weakly in }\widehat{H}^{1/2}(D_1),\qquad u_{i}\to u\quad\text{strongly in }L^{2}(D_{1}) \text{ and a.e. }
	\]
	and there is a Radon measure $\mu^e$ on $\R^{n+1}_+ \cup D_1$ such that
	\[
	\mu^e_i:=|\na u^e_{i}|^{2}d\mathbf{x}\wto \mu^e %\qquad\text{on }
	\]
	in the sense that for any $\varphi\in C_c(\R^{n+1}_+ \cup D_1)$ there holds
	 $$\int \varphi d \mu^e_i \to \int \varphi d \mu^e,\qquad \text{as }i\to \wq.$$	
	Then the following conclusions hold:
	\begin{itemize}
	\item[(i)] There is a closed $(n-1)$-rectifiable set $\Sigma\subset D_1$, with
	${\cal H}^{n-1}(\Sigma\cap K)<\infty$ for all compact set $K\subset D_1$, such that $u\in C^{\infty}(D_1\backslash\Sigma)$
	and
	\[
	u_{i}\to u\qquad\text{in } C_{\loc}^{1}(D_1\backslash\Sigma).
	\]
	\item[(ii)] {The map} $u$ is a weakly $1/2$-harmonic map in $D_1$  with $\sing(u)\subset\Sigma$.
	
	\item[(iii)] There exist a nonnegative Radon measure $\nu$ with ${\rm spt}(\nu)\cup \sing(u)=\Sigma$ and a  {nonnegative measurable}
	function $\Theta\colon \Sigma\to [0,\infty)$  such that
	\[
	\mu^e=|\na u^e|^{2}dx+\nu,\qquad\text{and}\qquad \nu=\Theta{\cal H}^{n-1}|_{\Sigma}.
	\]
	Moreover, there are constants $C,c$ depending only on $n,N$ such that
	\[
	c\ep_{1}\le\Theta(x)\le C\Lambda\qquad \text{for } {\cal H}^{n-1}\text{-a.e. } x\in \Sigma,
	\]
	where $\ep_1>0$ is {the constant given by Theorem 2.7}.
	\end{itemize}
	
%{\color{red}	{\upshape(iv) } (not proved yet! Maybe we do not need this) } $u$ enjoys the unique continuation property in the sense that if there is another weakly 1/2-harmonic map $v$ in $D_1$ such that $u=v$ almost everywhere on an open set, and $v$ is smooth away from a set $\Sigma'$ of finite $\mathcal{H}^{n-1}$-measure, then $u\equiv v$ on $D_1$.
\end{proposition}

We call $\nu$ {as a} {\it defect measure} of the sequence $\{u_i\}$, and $\Sigma$ the {\it energy {concentration} set} of the sequence $\{u_i\}$.

\begin{proof} The proof is rather standard. {For the convenience of readers, we sketch it}.
First of all, following \cite[proof of Theorem 7.1]{Millot-Pegon-Schikorra-2021-ARMA}, we obtain that $u_i^e\to u^e$ strongly in $H^1_{\loc}(B^+_1)$. Thus the energy concentrating set $\Sigma$ lies in $\pa^0 B_1^+$. {In fact,}  $\Sigma$ {can be defined} by
	\begin{align*}
	\Sigma:&=\Big\{x\in D_1: \lim_{r\to 0}\liminf_{i\ri\oo}r^{1-n}\mu^e_i(B^+_{r}(\mathbf{x}))  \ge \ep_0  \Big\}\\
	&=\Big\{x\in D_1: \lim_{r\to 0}\liminf_{i\ri\oo}\Theta(u^e_i,B^+_{r}(\mathbf{x}))  \ge \ep_0  \Big\}\\
	&=\Big\{x\in D_1: \lim_{r\to 0} r^{1-n}\mu^e(B^+_{r}(\mathbf{x}))  \ge \ep_0  \Big\}.
	%\\&=\left\{x\in D_1: \Xi(\mu, {x})  \ge \ep_0  \right\}.
	\end{align*}
	Then, {thanks to the monotonicity formula \eqref{1.8}},  $\Si$ is closed and has locally finite $(n-1)$-dimensional Hausdorff measure by a standard covering argument.
	Assume $x\in D_1\ba\Si,$ then there exists $0<r<1$ such that
	\[
	\liminf_{i\ri\oo}r^{1-n}\int_{B^+_r(\textbf{x})}|\na u^e_i|^2d\textbf{x}<\ep_0.
	\]
	By \cite[Theorem 2.1]{Scheven-2006-MZ}, $u^e_i\in C^{0,\al}(B^+_{r/2}(\textbf{x}))$. Higher regularity in \cite{Gulliver-Jost1987} implies $u^e_i$ is uniformly bounded in $C^{1,\al}(B^+_{r/4}(\textbf{x}))$. Hence by a covering argument, we see that $u_{i}\to u$ in  $C_{\loc}^{1}(D_1\backslash\Sigma)$. 
	This proves assertion (i) except the assertion that $\Sigma$ is $(n-1)$-rectifiable.
	
    Next we prove assertion	(ii) by  {a method similar to} that of \cite{Naber-V-2018}. Indeed, by  the Dirichlet-to-Neumann characterization of \cite[Lemma 2.9]{Millot-Sire-15} and strong convergence in (i), we find that $u$ is a weakly $1/2$-harmonic map in $D_1\ba\Sigma$. In order to show this fact is also true in the whole domain $D_1$, we take any test vector field $w\in C_c^\wq(D_1,\R^d)$ satisfying $w(x)\in T_{u(x)}N$ for  a.e. $x\in D_1$. Denote by $W$  any smooth extension of $w$ compactly supported in $B^+_1\cup D_1$. Since $\HH^{n-1}_{\rm{loc}}(\Si)<\oo,$ $\Si$ has vanishing $(1,2)$-capacity. More precisely, since $\Si\cap \mbox{spt}(w)$ is compact, by \cite[Theorem 5.1.9 and Corollary 3.3.4]{Adams-1996}, there exists $\psi_k\in C^\wq_c(\R^{n})$, with $\Psi_k$ being any smooth extension of {$\psi_k$} compactly supported in $B^+_1\cup D_1$
    {satisfying}
	$$ \Psi_k\equiv1  \mbox{ on } \Si\cap \mbox{spt}(w) \quad \text{and}\quad \|\Psi_k\|_{W^{1,2}(B^+_1)}\to 0\ \mbox{as}\ k\ri\oo.$$
	Set $\phi_k=1-\psi_k$ and $\Phi_k=1-\Psi_k$ such that $\phi_k\equiv\Phi_k\equiv0$ on $\Si\cap \mbox{spt}(w)$. Multiplying $\Phi_k W$ on both sides of  \eqref{harmonic function} and integrating we have
	\begin{align*}
	0=&\int_{B^+_1}\na u^e\cdot\na(\Phi_k W)d\mathbf{x}=\int_{B^+_1}\na u^e\cdot\na((1-\Psi_k) W) d\mathbf{x}\\
	&=\int_{B^+_1}\na u^e\cdot\na  W d\mathbf{x}-\int_{B^+_1}\na u^e\cdot\na  (\Psi_k W) d\mathbf{x}.
	\end{align*}
	The choice of $\Psi_k$ implies that
	\begin{align*}
	\int_{B^+_1}\na u^e\cdot\na  (\Psi_kW) d\mathbf{x}&= \int_{B^+_1}\na u^e\cdot(\na  \Psi_k W+\Psi_k \na W) d\mathbf{x}\\
	&\leq C\|u^e\|_{W^{1,2}}\|W\|_{C^1}\|\Psi_k\|_{W^{1,2}}\ri0\ \ \mbox{as}\ n\ri\oo.
	\end{align*}
	Hence we obtain
	$$\int_{B^+_1}\na u^e\cdot\na  W d\mathbf{x}=0.$$
	Applying the Dirichlet-to-Neumann characterization of \cite[Lemma 2.9]{Millot-Sire-15} again, we deduce that
	\begin{align*}
	\langle(-\De )^{1/2}u,w\rangle_{D_1}=\int_{B^+_1}\na u^e\cdot \na W d\mathbf{x}=0,
	\end{align*}
	This shows that $u$ is a weakly 1/2-harmonic map on $D_1 $.
	Since $u$ is smooth outside $\Sigma$ by assertion (i), $\sing(u)\subset \Sigma$. 	
	
	Now we can prove assertion (iii). The fact that $\sing(u)\cup\spt(\nu)\subset\Sigma$ can be inferred  from assertion (i), since for each  $x_0\in D_1\ba\Si$,   we can find a radius $0<r<\dist(x_0,\pa D_1)$ such that $u_i^e\to u^e$ strongly in $H^1(B_r^+)$, which then implies that $\nu(\bar{B}_{r/2}^+)=0$. Thus $\sing(u)\cup\spt(\nu)\subset\Sigma$.
	
	Next assume $x_0\in \Si\backslash  \sing(u)$. Then $\na u^e$ is bounded in a neighborhood of $\mathbf{x}_0$ and so $r^{1-n}\int_{B^+_r(\mathbf{x}_0)}|\na u^e|^2\to 0$ as $r\to 0$. Hence the definition of $\Si$ yields
	$${\lim_{r\ri 0}}r^{1-n}\nu^e(B^+_r(\mathbf{x}_0))\geq\ep_0,$$
	which implies $\mathbf{x}_0\in\spt(\nu^e)$. This shows $\sing(u)\cup\spt(\nu^e)\supset\Sigma$.
	
		The remaining assertions and  the countably $(n-1)$-rectifiability of $\Sigma$ in assertion (i) follow from the same argument as that of 
		{Lin \cite{Lin-1999-Annals} and}  Scheven \cite{Scheven-2008-ACV}. We omit the details. \end{proof}

Using Proposition \ref{prop: defect measure}, we can now explore properties of tangent maps  of stationary  $1/2$-harmonic maps.

\begin{proposition} \label{prop: tangent measure1}
	Let $u\in \widehat{H}^{1/2}(D_1, N)$ be a stationary weakly $1/2$-harmonic map. {For $x_0\in D_1$, suppose} there is a sequence $r_j\to 0$ such that
	{$u_{j}=u_{x_0,r_j}:=u(x_0+r_j\cdot)$ satisfies}
	\[
	u_{j}\rightharpoonup v\quad\text{in }\widehat{H}_\loc^{1/2}(\mathbb{R}^n, N) \text{ and almost everywhere},
	\] and
	 $$|\na u^e_{j}|^{2}d\mathbf{x}\wto\mu^e $$
	  in the sense of Radon measures as  in Proposition \ref{prop: defect measure}. 	Then
	\begin{itemize}
	\item[(i)] There is a closed $(n-1)$-rectifiable set $\Sigma\subset \mathbb{R}^n$ with locally finite	
	${\cal H}^{n-1}$-measure, such that $v\in C_{\loc}^{1}(\R^n\backslash \Sigma)$
	and
	\[
	u_{j}\to v\qquad\text{in } C_{\loc}^{1}(\R^n\backslash\Sigma).
	\]
	\item [(ii)] {\rm{(symmetry of tangent map)}} $v$ is a  weakly $1/2$-harmonic map  with ${\rm sing}(v)\subset\Sigma$, and moreover, it is 
	{homogeneous of degree zero} with respect to the origin, i.e.
	$$v(\lambda x)=v(x), \qquad\text{for all } x\in\R^n,\lambda>0.$$
	\item [(iii)] There exist a  Radon measure $\nu$ in $\R^{n+1}$ with ${\rm spt}(\nu)\cup \sing(v)=\Sigma$ and a density function $\Theta_{\nu}\colon \Sigma\to [0,\infty)$  such that
	\[
	\mu^e=|\na v^e|^{2}d\mathbf{x}+\nu,\qquad\text{and}\qquad \nu=\Theta_{\nu}{\cal H}^{n-1}|_{\Sigma}.
	\]
	Moreover, there are constants $C,c$ depending only on $m,n,N$ such that
	\[
	c\ep_{1}\le\Theta_{\nu}(x)\le C\Lambda\qquad \text{for } {\cal H}^{n-1}\text{-almost every } x\in \Sigma,
	\]
	where $\ep_1$ is given in {\upshape (iii)} of Proposition \ref{prop: defect measure}.
	\item[(iv)] {\rm{(symmetry of tangent measure)}} $\mu$ is a cone measure in the sense that for any $\la>0$ and measurable set $A\subset\mathbb{R}^n$, there holds
	\[\la^{1-n}\mu(\la A)=\mu(A).\]
	Consequently, thanks to the homogeneity of $v$ in assertion {\upshape(ii)}, the measure $\nu$ is a cone measure as well.
	\end{itemize}
\end{proposition}
We call $\mu^e$ or $(v,\nu)$  from Proposition \ref{prop: tangent measure1} a {\it tangent measure} of $u$ at  $x_0\in D_1$.

%\begin{lemma}\label{qya2} Assume that $u\in W^{2,2}(B_2,N)$ is a stationary weakly 1/2-harmonic map, $v$ is a tangent map of $u$ at $a\in\Omega,$ then $v$ is 0-homogeneous with respect to the origin, i.e. $$v(x)=v(\lambda x),\qquad\forall x\in\R^n,\lambda>0.$$\end{lemma}
\begin{proof}
	
	Assertions (i)--(iii) are direct consequences of Proposition  \ref{prop: defect measure}, except the conclusion concerning the homogeneity of $v$. 
	We will prove it together with assertion (iv).	
	Denote $u^e_j:=u^e_{x_0,r_j}$. Using the monotonicity formula \eqref{1.8}, we find that
	\begin{eqnarray}\label{sinaf}
	\Theta(u^e_j,B_r^+(\mathbf{\mathbf{0}}))-\Theta(u^e_j,B_\rho^+(\mathbf{\mathbf{0}}))= \int_{B_r^+(\mathbf{0})\ba B_\rho^+(\mathbf{0})}\frac{|\mathbf{x}\cdot\na u^e_j|^2}{|\mathbf{x}|^{n+1}}d\mathbf{x}.
	\end{eqnarray}
	Note that, for any $r>0$ we have $\Theta(u^e_j,B_r^+( {\mathbf{0}}))=\Theta(u^e,B_{rr_j}^+( {\mathbf{x}_0}))$ for all $j\ge 1$.
	Hence by the definition  \eqref{eq: density of map} of density function $\Xi(u,\cdot)$, it follows that
	$$r^{1-n}\mu^e(B_r^+( {\mathbf{0}}))=\lim_{j\to\oo}\Theta(u^e_j,B_r^+( {\mathbf{0}}))=\lim_{j\to\oo}\Theta(u^e,B_{rr_j}^+( {\mathbf{x}_0}))=\Xi(u,x).$$
	
	On the other hand, note that $\na u_j^e\to \na v^e$ a.e. in $B^+_r$. Thus,  sending $j\ri\oo$ in \eqref{sinaf}    and using Fatou's lemma we obtain
\[\int_{B_r^+(\mathbf{0})\ba B_\rho^+(\mathbf{0})}\frac{|\mathbf{x}\cdot\na v^e|^2}{|\mathbf{x}|^{n+1}}d\mathbf{x} =0 \]
for any $r>\rho>0$.	This shows that $v^e$ is homogeneous {of degree zero} with respect to origin and so is $v$.

To prove assertion (iv), it is equivalent to prove that, for any $\Psi\in C_c^1(\mathbb{R}^{n+1}_+)$, there holds
\begin{equation}\label{eq: scaling invariance}
\begin{aligned}
	\frac{d}{d\la}(\mu^e_{0,\la}(\Psi))&=\frac{d}{d\la}\big(\la^{1-n}\int_{\mathbb{R}^{n+1}_+}\Psi_\la d\mu^e\big)\\
	&={-\la^{-n}}\int_{\mathbb{R}^{n+1}_+}\Big((n-1){\Psi(\frac{\mathbf{x}}{\la})}+\frac{\mathbf{x}}{\la}\cdot\na\Psi(\frac{\mathbf{x}}{\la})\Big)d\mu^e\\
	&=0,
\end{aligned}
\end{equation}
for all $\la>0$, where $\Psi_\la(\mathbf{x})=\Psi(\mathbf{x}/\la)$.

To see this,  choosing the test vector field $\Psi_\la(\mathbf{x})\mathbf{x}$ \sout{and using} in the the stationarity {identity} 
of $u^e_j$ (see \cite[Proposition 2.15]{Millot-Pegon-Schikorra-2021-ARMA}), we derive that
\begin{eqnarray}\label{eq: stationarity}
	\sum_{\al,\beta=1}^{n+1}\int_{\mathbb{R}^{n+1}_+}(\de_{\al\beta}|\na u^e_j|^2-2\langle\na_\al u^e_j,\na_\beta u^e_j\rangle)\na_\al(\Psi_\la\mathbf{x}_\beta)d\mathbf{x}=0, \qquad \forall\,j\ge 1.
\end{eqnarray}
By assumption, the first term in the above equality converges to
\begin{equation}\label{eq: first term}
 \int_{\mathbb{R}^{n+1}_+} {\rm div}(\Psi_\la(\mathbf{x})\mathbf{x})d\mu^e=\int_{\mathbb{R}^{n+1}_+} \Big((n+1)\Psi_\la(\mathbf{x})+\frac{\mathbf{x}}{\la}\cdot\na\Psi(\frac{\mathbf{x}}{\la})\Big)d\mu^e.
\end{equation}
	To deal with the second term in \eqref{eq: stationarity}, note that $ \int_{B_r^+(\mathbf{0})}{|\mathbf{x}\cdot\na u^e_j|^2}d\mathbf{x}\to 0$ holds for any $r>0$ as $j\to \wq$ by  \eqref{sinaf}. Thus, by using the integral inequality
\begin{align*}
\big(\sum_{\al,\beta=1}^{n+1}\int_{\mathbb{R}^{n+1}_+}\mathbf{x}_\al{(\Phi_\la)_\beta} \langle\na_\al u^e_j,\na_\beta u^e_j\rangle d\mathbf{x}\big)^2
\leq \big(\int_{\mathbb{R}^{n+1}_+} |{(\mathbf{\Phi}_\la)_\beta\cdot\na_\beta} u^e_j|^2 d\mathbf{x}\big)\big(\int_{\mathbb{R}^{n+1}_+} |\mathbf{x}\cdot\na u^e_j|^2 d\mathbf{x}\big)
\end{align*}
and	letting $j\ri\oo$ we obtain
	\begin{eqnarray*}
		\sum_{\al,\beta=1}^{n+1}\int_{\mathbb{R}^{n+1}_+}\mathbf{x}_\al{(\Phi_\la)_\beta} \langle\na_\al u^e_j,\na_\beta u^e_j\rangle d\mathbf{x}\to 0.
	\end{eqnarray*}
Therefore we infer that
\[\int_{\mathbb{R}^{n+1}_+}2\langle\na_\al u^e_j,\na_\beta u^e_j\rangle\na_\al(\Psi_\la\mathbf{x}_\beta)d\mathbf{x}\to  2\int_{\mathbb{R}^{n+1}_+} \Psi_\la d\mu^e.  \]
	Combining this convergence together with  \eqref{eq: stationarity} and \eqref{eq: first term}  we deduce \eqref{eq: scaling invariance}. 
	Since $\nu=\mu^e-|\na v^e|^2d\mathbf{x}$ and {$v^e$} is homogeneous {of degree zero, $\nu=\nu^e\lfloor\partial\mathbb{R}^{n+1}_+$ is a cone measure}.
	The proof is complete.
	\end{proof}
	
Now we are able to discuss the stratification of singular set for stationary $1/2$-harmonic maps. First observe that
\begin{lemma}\label{kkkl}
	Let $u\in \widehat{H}^{1/2}(\Omega,N)$ be a stationary $1/2$-harmonic map. Then $$y\in \reg(u)\quad \Longleftrightarrow\quad \Xi(u^e,y)=0.$$
\end{lemma}
\begin{proof} ($\Rightarrow$) If $y\in \reg(u)$, then $u$ is $C^1$ in a neighborhood of $y$ by  \cite[Theorem 5.1 and Proposition 6.2]{Millot-Pegon-Schikorra-2021-ARMA}, 
hence $\Xi(u^e,y)=0$. \\
($\Leftarrow$) It follows from monotonicity formula \eqref{1.8} and the $\ep$-regularity theory.
\end{proof}

The next lemma  extends the monotonicity formula \eqref{1.8} of stationary $1/2$-harmonic maps to tangent measures.
\begin{lemma}[Monotonicity of tangent measures]\label{kun}
	Suppose $u\in \widehat{H}^{1/2}(\Omega,N)$ is a stationary $1/2$-harmonic map, $v$ is a tangent map of $u$ at a point $x_0\in {\rm sing}(u)$ 
	and  $\mu^e=(v^e,\nu)$ is the  tangent measure given by Proposition \ref{prop: tangent measure1}.
	Then for every $\mathbf{y}:=(y,0)\in \mathbb{R}^{n+1}_+$, the function $\Theta(\mu^e,B^+_\rho(\mathbf{y}))$ is {monotonically} nondecreasing 
	{of $\rho$},  and
	$$\Xi(\mu^e,y):=\lim_{\rho\ri0}\Theta(\mu^e,B^+_\rho(\mathbf{y})) \quad\text{exists and}\quad \Xi(\mu^e,y)\leq\Xi(\mu^e,0).$$
\end{lemma}
\begin{proof}
	By the definition of tangent map, there exists a sequence $r_j\ri 0$  such that
	$u_j\equiv u_{x_0,r_j}\wto v$ in $\widehat{H}^{1/2}_{\loc}(\R^n,N).$ By the monotonicity formula \eqref{1.8}, we have
	\begin{eqnarray}\label{sinp} \Theta(u^e_j,B_r^+(\mathbf{\mathbf{y}}))-\Theta(u^e_j,B_\rho^+(\mathbf{\mathbf{y}}))= \int_{B_r^+(\mathbf{y})\ba B_\rho^+(\mathbf{y})}\frac{|(\mathbf{x}-\mathbf{y})\cdot\na u^e_j|^2}{|\mathbf{x}-\mathbf{y}|^{n+1}}d\mathbf{x}
	\end{eqnarray}
	for each $y\in\Om$ and $0<\rho<r<d(y,\partial\Om)$.
	Letting $j\ri\oo$ we see that  the function $\rho\mapsto \Theta(\mu^e,B^+_\rho(\mathbf{y}))$ is {monotonically} 
	nondecreasing for every $\mathbf{y}:=(y,0)\in {\Omega\times\{0\}}$,  and
	$$\Xi(\mu^e,y):=\lim_{\rho\ri0}\Theta(\mu^e,B^+_\rho(\mathbf{y})) \quad\text{exists}.$$
Furthermore, letting $\rho\ri 0$ in \eqref{sinp} yields
	\begin{eqnarray*} \Xi({u^e_j},y)+\int_{B_r^+(\mathbf{y})}\frac{|(\mathbf{x}-\mathbf{y})\cdot\na u^e_j|^2}{|\mathbf{x}-\mathbf{y}|^{n+1}}d\mathbf{x}
=\Theta(u^e_j,B_r^+(\mathbf{\mathbf{y}})).
	\end{eqnarray*}
	Letting $j\ri\oo$ we obtain
	\begin{eqnarray}\label{8u}
	\Xi(\mu^e,y)+\lim_{j\ri\oo}\int_{B_r^+(\mathbf{y})}\frac{|(\mathbf{x}-\mathbf{y})
\cdot\na u^e_j|^2}{|\mathbf{x}-\mathbf{y}|^{n+1}}d\mathbf{x}
=\Theta(\mu^e,B_r^+(\mathbf{\mathbf{y}})).
	\end{eqnarray}
	On the other hand, by Proposition \ref{prop: tangent measure1},
	\begin{eqnarray*} \Theta(\mu^e,B_r^+(\mathbf{\mathbf{y}}))\leq\frac{(r+|\mathbf{y}|)^{n-1}}{r^{n-1}}
\Theta(\mu^e,B_{r+|\mathbf{y}|}^+(\mathbf{\mathbf{0}}))
		=\frac{(r+|\mathbf{y}|)^{n-1}}{r^{n-1}}\Xi(\mu^e,0).
	\end{eqnarray*}
	Combining with \eqref{8u} and letting $r\ri\oo$ we get
	\begin{eqnarray*}	\Xi(\mu^e,y)+\lim_{j\ri\oo}\int_{B_R^+(\mathbf{y})}\frac{|(\mathbf{x}-\mathbf{y})\cdot\na u^e_j|^2}{|\mathbf{x}-\mathbf{y}|^{n+1}}d\mathbf{x}\leq\Xi(\mu^e,0), \ \forall R>0.
	\end{eqnarray*}
	The proof is complete.
\end{proof}

Using the density function $\Xi(\mu^e,\cdot)$, we can deduce the following result,  similar to that of Lin \cite{Lin-1999-Annals}.

\begin{lemma}\label{ou}
	Suppose $x_0\in\sing(u) $ and $\mu$, $v$ and $\Xi(\mu,\cdot)$ are given by Lemma {\rm\ref{kun}}. Define $S(\mu^e)$ by
	$$S(\mu^e)=\big\{y\in\R^n\ \big|\ \Xi(\mu^e,y)=\Xi(\mu^e,0)\big\}.$$
	Then $S(\mu^e)$ is a linear subspace of $\R^n$ and $\Xi(\mu^e,\cdot)$ is translation invariant along $S(\mu^e)$, i.e.
	$$\Xi(\mu^e,x+y)=\Xi(\mu^e,x)\qquad \text{for all }\,x\in\R^n,\,y\in S(\mu^e).$$ Moreover,  $S(\mu^e)\subset \sing(v)\cup \spt(\nu)$.
\end{lemma}

From the theory of defect measures in Proposition \ref{prop: tangent measure1}, we know that the Hausdorff dimension ${\rm dim} S(\mu^e)\leq n-1$, as $S(\mu^e)\subset \sing(v)\cup {\rm spt}(\nu)$ by the above lemma.
Therefore, for $j=0,1,\cdots,n-1$,  we define the stratification of ${\rm sing}(u)$ by {letting}
$$\Si^j(u):=\big\{x_0\in{\rm sing}(u):\mbox{dim}(S(\mu^e))\leq j\ \mbox{ for all tangent measures } \mu^e  \mbox{ of } u  \mbox{ at } x_0 \big \}.$$
It is easy to see that
$$\Si^0\su\Si^1\su\cdots\su\Si^{n-4}{\su}\Si^{n-3}\su\Si^{n-2}\su\Si^{n-1}={\rm sing}(u).$$
This is the so-called classical stratification of singular set of $u$. In the case $u$ is a minimizing $1/2$-harmonic map, 
the  compactness of these maps (see  \cite{Millot-Pegon-Schikorra-2021-ARMA}) implies that $\nu\equiv0$ and so
\[
\Sigma^k(u)=\big\{x\in \Omega:\text{no tangent maps of } u \text{ at } x \mbox{ is } (k+1)\mbox{-symmetric}\big\}
\]
where the notion of $k$-symmetry is standard (see Definition \ref{def:symmetry} below).

\subsection{Compactness theory of stationary $1/2$-harmonic maps}

In this subsection, we extend the compactness theory of stationary harmonic maps developed by Lin \cite[Lemma 3.1]{Lin-1999-Annals} to the fractional case. To this end, denote by $\mathcal{M}(B_5^+)$ the set of all measures $\mu^e$ which arise as weak limits of stationary $1/2$-harmonic maps in  $D_5$, and by $\pi(\mu^{e})$ the energy 
concentration set corresponding to $\mu^e$ (see Proposition \ref{prop: tangent measure1} for the definition). It follows by standard arguments that $\mathcal{M}(B_5^+)$ is weakly closed with respect to convergence of Radon measures in $B_5^+\cup D_5$. Our main theorem of this section is

\begin{theorem}\label{thm: compactness of stationary HM} Let $\mu^{e}\in{\cal M}(B_{5})$ and $\Si=\pi(\mu^{e})$.
\begin{itemize}
\item [(1)] If ${\cal H}^{n-1}(\Si)>0$, there exists a  smooth $1/2$-harmonic
	$\mathbb{S}^1$ in $N$, i.e. a {non-constant} $1/2$-harmonic map from $\mathbb{S}^{1}$ into $N$.
\item[(2)](Compactness) If $N$ does not admit any non-constant smooth $1/2$-harmonic
	$\mathbb{S}^1$, then every bounded sequence of stationary $1/2$-harmonic
	maps has a strongly convergent subsequence.
	\end{itemize}
	\end{theorem}
\begin{proof}
	Assertion (2) follows directly from assertion (1). The proof of assertion
	(1) is divided into  four steps:
	\begin{itemize}
	\item[\textbf{Step1}] We may assume that $\mu^{e}$ is the weak limit of
	a sequence $\{u_{i}\}\subset\widehat{H}_{\La}^{1/2}(D_{1},N)$ of
	stationary 1/2-harmonic maps for some $\La>0$, and $0\in\Si$ is
	a concentration point. Moreover, by a diagonal argument and blow-up
	procedure at the concentration point $0$, we can assume that
	\[
	u_{i}^{e}\wto{\rm constant}\qquad\text{in }H^{1}(B_{2}^{+},N)
	\]
	and
	\[
	|\na u_{i}^{e}|^{2}d\mathbf{x}\to\mu^{e}=c{\cal H}^{n-1}\lfloor\Sigma_{\ast}\qquad\text{in }B_{2}^{+}\cup D_{2}
	\]
 as weak convergence of Radon measures for some positive constant $c>0$,
	where
	$$\Si_{\ast}=D_{5}^{n-1}\times\{0\}\subset\R^{n-1}\times\{0\}.$$
	To clarify the dimension of disks in the base space $\R^{n}$, here and hereafter  we use $D_{r}^{k}=\{x\in\R^{k}:|x|<r\}$ to denote the $k$-dimensional unit ball in $\R^k$, and $\mu^{e}$ is translation invariant along $\R^{n-1}$.
	
	\item[\textbf{Step2}] Claim 1. There holds
	\[
	\int_{B_{1}^{+}}\sum_{j=1}^{n-1}\left|\pa_{x_{j}}u_{i}^{e}\right|^{2}d\mathbf{x}\to0\qquad\text{as }i\to\wq.
	\]
	To see this, write $e_{0}=0$ and let $e_{j}$ be the standard basis
	of $\R^{n}$. By the monotonicity formula \eqref{1.8}, for all $0\le j\le n-1$
	and all $0<r<R\le 4$, we have
	\[
	\int_{B_{R}^{+}(e_{j})\backslash B_{r}^{+}(e_{j})}\frac{|(\mathbf{x}-e_{j})\cdot\na u_{i}^{e}|^{2}}{|\mathbf{x}-e_{j}|^{n+1}}d\mathbf{x}=\Theta(u_{i}^{e},B_{R}^{+}(e_{j}))-\Theta(u_{i}^{e},B_{r}^{+}(e_{j})).
	\]
	Since $\mu^{e}$ is tranlation invariant along $\R^{n-1}$, sending
	$i\to\wq$ gives
	\[
	\lim_{i\to\wq}\int_{B_{R}^{+}(e_{j})\backslash B_{r}^{+}(e_{j})}\frac{|(\mathbf{x}-e_{j})\cdot\na u_{i}^{e}|^{2}}{|\mathbf{x}-e_{j}|^{n+1}}d\mathbf{x}=\Theta(\mu^{e},B_{R}^{+}(e_{j}))-\Theta(\mu^{e},B_{r}^{+}(e_{j}))=0
	\]
	{holds for all $1\le j\le n-1$}.
	To proceed, note that
	\[
	|e_{j}\cdot\na u_{i}^e|^{2}\le2|(x-e_{j})\cdot\na u_{i}^e|^{2}+2|(x-e_{0})\cdot\na u_i^e|^{2}.
	\]
	Thus, we obtain {that for $0<\sigma<\frac12$,}
	\[
	\begin{aligned}\int_{B_{1}^{+}}|e_{j}\cdot\na u_{i}^e|^{2} & \le2\int_{B_{1}^{+}}|(\mathbf{x}-e_{j})\cdot\na u_{i}^e|^{2}+2\int_{B_{1}^{+}}|(\mathbf{x}-e_{0})\cdot\na u_{i}^e|^{2}\\
	& \le2\int_{B_{2}^{+}(e_{j})\backslash B_{\si}^{+}(e_{j})}|(\mathbf{x}-e_{j})\cdot\na u_{i}^e|^{2}+{2\int_{B^+_\sigma(e_j)}|(\mathbf{x}-e_{j})\cdot\na u_{i}^e|^2} \\
	& \quad+2\int_{B_{1}^{+}\backslash B_{\si}^{+}}|(\mathbf{x}-e_{0})\cdot\na u_{i}^e|^{2}+2\int_{B_{\si}^+}|(\mathbf{x}-e_{0})\cdot\na u_{i}^e|^{2}\\
	& =o_{i}(1)+o_{i}(1)+2\int_{B_{\si}^+}|(\mathbf{x}-e_{0})\cdot\na u_{i}^e|^{2}+{2\int_{B^+_\sigma(e_j)}|(\mathbf{x}-e_{j})\cdot\na u_{i}^e|^2}.
	\end{aligned}
	\]
	The last two terms can be  estimated by
	\[
	\int_{B_{\si}^+}|(\mathbf{x}-e_{0})\cdot\na u_{i}^e|^{2}\le\si^{2}\int_{B^+_{\si}}|\na u_{i}^e|^{2}\le\si^{2}\cdot C(\La)\si^{n-2}=O(\si^{n}),
	\]
	and
	\[
	\int_{B^+_{\si}(e_j)}|(\mathbf{x}-e_{j})\cdot\na u_{i}^e|^{2}\le\si^{2}\int_{B^+_{\si}(e_j)}|\na u_{i}^e|^{2}\le O(\si^{n}).
	\]
	Hence combining the above estimates
	gives $\int_{B_{1}^{+}}|e_{j}\cdot\na u_{i}^e|^{2}=o_{i}(1)$ for all
	$0\le j\le n-1$. This proves the claim.
	
	\item[\textbf{Step3}] Locate the blow-up point. Denote $X_{1}=(x_{1},\cdots,x_{n-1})$
	and define
	\[
	f_{i}(X_{1})=\int_{[-1,1]\times(0,1/2)}\left|\na_{X_{1}}u_{i}^{e}(X_{1},x_{n},z)\right|^{2}dx_{n}dz
	\]
	for $i\ge1$. The claim in Step 2 implies that 
	$$\int_{D^{n-1}}f_{i}(X_{1})dX_{1}\to0, \  {\rm{as}}\  i\to\wq.$$
	We have the following observations:
	\begin{itemize}
	\item[(i)] By the partial regularity theory of stationary 1/2-harmonic maps,
	for each $i\ge1$, there exists a closed subset $E_{i}\subset D_{1}$, {
	with ${\cal H}^{n-1}(E_i)=0$},  such that $u_{i}\in C^{\wq}(D_{1}\backslash E_{i})$.
	
	\item[(ii)] By the Hardy-Littlewood {theorem}, there exists $C(n)>0$ such
	that
	\[
	\Big|\big\{X_{1}\in D^{n-1}:Mf_{i}(X_{1})>\la\big\}\Big|\le\frac{C}{\la}\|f_{i}\|_{L^{1}(D^{n-1})}
	\]
	for all $i\ge1$. {Here $Mf_i$ denotes the Hardy-Littlewood maximal function of $f_i$}.
	\end{itemize}
	Therefore, by taking $\la_{i}=\|f_{i}\|_{L^{1}(D^{n-1})}^{1/2}$ we have 
	 $$\big|\{X_{1}:Mf_{i}(X_{1})>\la_i\}\big|\le C\la_{i}\to0, \ {\rm{as}}\  i\to\wq,$$
	so that there exist $X_{1}^{i}\in D^{n-1}\backslash E_{i}$ {and $s_i>0$} such that
	\begin{equation}
	u_{i}\in C^{\wq}(D_{s_{i}}^{n-1}(X_{1}^{i})\times(-1,1))\label{eq: smoothness-1}
	\end{equation}
	and
	\[
	\sup_{0<r<1}r^{1-n}\int_{D_{r}^{n-1}(X_{1}^{i})}f_{i}dX_{1} \le Mf_i(X_1^i)\le \la_i \to0\qquad\text{as }i\to\wq.
	\]
	
	\noindent \text{Claim 2.} There exist $\de_{i}\to0$ and $x_{n}^{i}\to0$
	such that
	\[
	\sup_{x_{n}\in(-1/2,1/2)}\de_{i}^{1-n}\int_{D_{\de_{i}}^{n-1}(X_{1}^{i})\times D_{\de_{i}}^{1}(x_{n})\times(0,1/2)}|\na u_{i}^{e}|^{2}=\frac{\ep_{0}}{c(n)}\qquad\text{for }i\gg1
	\]
	is achieved at $x_{n}=x_{n}^{i}$, where $c(n)\ge1$ is a
	constant that will be determined later.
	
	To prove Claim 2, we argue by contradiction.  First, note that
	by the smoothness property (\ref{eq: smoothness-1}), for each $i\ge1$
	and each $x_{n}\in(-1/2,1/2)$, there has
	\[
	\de^{1-n}\int_{D_{\de}^{n-1}(X_{1}^{i})\times D_{\de}^{1}(x_{n})\times(0,1/2)}|\na u_{i}^{e}|^{2}\le C\de\to0\qquad\text{as }\de\to0.
	\]
	However, for each fixed $\de>0$, if
	\[
	\sup_{x_{n}\in(-1/2,1/2)}\de^{1-n}\int_{D_{\de}^{n-1}(X_{1}^{i})\times D_{\de}^{1}(x_{n})\times(0,1/2)}|\na u_{i}^{e}|^{2}<\frac{\ep_{0}}{c(n)}, \ i\gg 1.
	\]
	Then, by the $\epsilon$-regularity theorem, we can conclude that $u_{i}^{e}\in C^{\wq}(D_{\de}^{n-1}(X_{1}^{i})\times D_{\de}^{1}(0))$
	converges strongly in $H^{1}(B_{\de}^{+})$, which contradicts with
	the assumption that $D_{\de}^{n-1}(X_{1}^{i})\times D_{\de}^{1}(0)\cap\Si_{\ast}$
	has positive measure. Therefore, there exists $\de_{i}>0$ such that
	\[
	\sup_{x_{n}\in(-1/2,1/2)}\de_{i}^{1-n}\int_{D_{\de_{i}}^{n-1}(X_{1}^{i})\times D_{\de_{i}}^{1}(x_{n})\times(0,1/2)}|\na u_{i}^{e}|^{2}=
	\frac{\ep_{0}}{c(n)}\qquad\text{for }i\gg1.
	\]
	
	Now, suppose for each $i\ge1$, the above supremum is achieved at a
	point $x_{n}^{i}\in (-1/2,1/2)$. We need to show that $x_{n}^{i}\to0$.
	In fact, if $x_{n}^{i}\ge\de_0>0$ for instance for some $\de_0>0$, then
	away from $\Si_{\ast}$ we have
	\[
	\int_{D^{n-1}(X_{1}^{i})\times(D^{1}\backslash D_{\de_0/4}^{1}(x_{n}^{i}))\times(0,1/2)}|\na u_{i}^{e}|^{2}\ge C(\ep_{0},n)>0,
	\]
	which contradicts the assumption that $u_{i}^{e}\to\text{const.}$
	in $H^{1}$ away from $\Si_{\ast}$. Claim 2 is proved.
	
\item[\textbf{Step4}] Now we blow up $u_{i}^{e}$ at $p^{i}=(X_{1}^{i},x_{n}^{i},0)$
	by setting
	\[
	v_{i}({\bf x})=u_{i}^{e}(p^{i}+\de_{i}{\bf x}),\qquad {\bf x}\in\Om_{i}\equiv D_{R_{i}}^{n-1}\times D_{R_{i}}^{1}\times\left(0,R_{i}\right),
	\]
	where $R_{i}=1/(2\de_{i})$. Note that $v_{i}$ is a stationary 1/2-harmonic
	map satisfying, for any $0<R<R_{i}$,
	\begin{eqnarray}
	&  & R^{1-n}\int_{D_{R}^{n-1}\times D_{R_{i}}^{1}\times(0,R_{i})}\left|\na_{X_{1}}v_{i}\right|^{2}\to0,\quad\text{as }i\to\wq,\label{eq: 3.5}\\
	&  & \int_{D^{n-1}\times D^{1}\times(0,1/2)}|\na v_{i}|^{2}=\max_{b\in D_{R_{i}-1}^{1}}\int_{D^{n-1}(0)\times D^{1}(b)\times(0,1/2)}|\na v_{i}|^{2}=\frac{\ep_{0}}{c(n)}.\label{eq: 3.6}\\
	&  & \sup_{i\ge1}\int_{D_{R}^{n-2}\times D_{R}^{1}\times(0,R)}|\na v_{i}|^{2}\le C(\La)R^{n-1}.\label{eq: 3.7}
	\end{eqnarray}
	{\eqref{eq: 3.5} follows from Claim 1 and}  the scaling invariance:
	\[
	R^{1-n}\int_{D_{R}^{n-1}\times D_{R_{i}}^{1}\times(0,R_{i})}\left|\na_{X_{1}}v_{i}\right|^{2}=(\de_{i}R)^{1-n}\int_{D_{\de_{i}R}^{n-2}(X_{1}^{i})\times D_{1/2}^{1}(x_{n}^{i})\times(0,1/2)}|\na_{X_{1}}u_{i}^{e}|^{2},
	\]
	\eqref{eq: 3.6} follows from Claim 2, and {\eqref{eq: 3.7} follows from} the fact that $u_{i}\in\widehat{H}_{\La}^{1/2}$.
	
	\noindent Letting $i\to\wq$, we deduce from the $\ep$-regularity theory and
	(\ref{eq: 3.6}) that
	\[
	v_{i}\to v_{\wq}^{e}\qquad\text{in }C_{\loc}^{1}(\overline{\R_{+}^{n+1}})
	\]
	for some smooth 1/2-harmonic mapping $v_{\wq}$. Moreover, by (\ref{eq: 3.5})
	we find that $\na_{X_{1}}v_{\wq}=0$, i.e., $v_{\wq}$ depends only
	$x_{n}$ and $z$. By (\ref{eq: 3.7}) we deduce that
	\[
	%{\cal E}(v_{\wq},\R)=\int_{\R\times(0,\wq)}|\na v_{\wq}^{e}|^{2}\le C(\La).
	\sup_{R>1} \frac{1}{R^{n-1}}\int_{B_R^+}|\na v_{\wq}^{e}|^{2}\le C(\La).
	\]
	This  implies that $v_{\wq}$ is a {non-constant}
	smooth $1/2$-harmonic map from
	$\R$ into $N$ with {finite $1/2$-Dirichlet energy}. Hence $v_{\wq}$  is a $1/2$-harmonic $\mathbb{S}^1$ in $N$. 
	\end{itemize}
	The proof is complete.
\end{proof}

\section{Quantitative stratification of singular set}\label{sec:Quantitative stratification}

\subsection{Quantitative symmetry and cone splitting principle}
Given $u\in\widehat{H}^{1/2}(\Om,N)$, recall that
$u^{e}$ is the Poisson extension \eqref{extension} of $u$ in $\mathbb{R}_{+}^{n+1}$.
First, we recall the notion of quantitative symmetry introduced by Cheeger-Naber \cite{Cheeger-Naber-2013-CPAM}.
\begin{definition}[Symmetry]\label{def:symmetry} Given a measurable
	map $u:\R^{n}\to\R$,  we say that
	\begin{itemize}
	\item[(1)] $u$ is 0-homogeneous {or 0-symmetric} with respect to a point $p\in\R^{n}$,
	if $u(p+\lambda v)=u(p+v)$ for all $\lambda>0$ and $v\in\mathbb{R}^{n}$.
	
	\item[(2)] $u$ is $k$-symmetric if $u$ is 0-homogeneous with respect to
	the origin, and $u$ is translation invariant with respect to a $k$-dimensional
	subspace $V\subset\mathbb{R}^{n}$, i.e.,
	\[
	u(x+v)=u(x)\qquad\text{for all }x\in\mathbb{R}^{n},\,v\in V.
	\]
	\end{itemize}
\end{definition}

A map $u\in\widehat{H}^{1/2}(\Om,N)$  is $k$-symmetric if and only if $u^e$ is $0$-homogeneous in $\R_{+}^{n+1}$ with respect to the origin, and {is }translation invariant along a $k$-dimensional subspace $V\subset\mathbb{R}^{n}=\pa \R_{+}^{n+1}$. For convenience, we introduce

\begin{definition}[Boundary symmetry]\label{def: boundary symmetry} Given a
	map $h:\overline{\R_{+}^{n+1}}\to\R$. We say that $h$ is boundary
	$k$-symmetric, if $h$ is 0-homogeneous in $\R_{+}^{n+1}$ in the
	sense that	
	\[
	h(\lambda v)=h(v),\qquad\forall\,\la>0\quad\text{and}\quad v\in\overline{\mathbb{R}_{+}^{n+1}};
	\]
	and if there is a $k$-dimensional subspace $V\subset\mathbb{R}^{n}=\pa\R_{+}^{n+1}$
	such that
	\[
	h(x+v)=h(x),\qquad\forall\,x\in\R_{+}^{n+1}\text{ and }v\in V.
	\]
\end{definition}
In view of the above definitions,  $u$ is $k$-symmetric in $\R^n$ if and only if $u^{e}$
is boundary $k$-symmetric. The quantitative symmetry is then defined via comparison with boundary $k$-symmetric functions.
\begin{definition}[Quantitative symmetry]\label{def:quantitative symmetry}
	Given a map $u\in\widehat{H}^{1/2}(\Omega,N)$, $\ep>0$ and a nonnegative
	integer $k$, we say that $u$ is $(k,\ep)$-symmetric on $D_{r}(x)\subset\Omega$,
	if there exists a boundary $k$-symmetric function $h\colon\R_{+}^{n+1}\to\mathbb{R}$
	such that
	\[
	\medint_{B_{r}^{+}(\mathbf{x})}|u^{e}(\mathbf{y})-h(\mathbf{y}-\mathbf{x})|^{2}d\mathbf{y}\leq\ep,
	\]
	where $\mathbf{x}=(x,0)$. Alternatively, we say that $u^{e}$ is
	boundary $(k,\ep)$-symmetric on $B_{r}^{+}(\mathbf{x})$.  
\end{definition}

Equivalently, $u$ is $(k,\ep)$-symmetric on $D_{r}(x)$ if and only
if the scaled map $u_{x,r}(y)=u(x+ry)$ is $(k,\ep)$-symmetric on
$D_{1}(0)$. A good compactness property of \sout{this} the quantitative symmetry is stated in the following remark.

\begin{remark}\label{rmk: compactness-of-symmetry}
	Suppose $\{u_i\}\subset \widehat{H}^{1/2}(D_2)$ converges weakly to some function $u\in \widehat{H}^{1/2}(D_2)$. If $u_i$ is $(k,\ep_i)$-symmetric in $D_1$ for some $\ep_i\to 0$, then $u$ is $k$-symmetric in  $D_1$.
\end{remark}

The proof of this remark is exactly the same as that of \cite[Remark 3.2]{He-Z-X-24} and is omitted here.
Given the definition of quantitative symmetry, we now introduce a
quantitative stratification for points of a function according to
how much it is symmetric around those points.

\begin{definition}[Quantitative stratification]\label{ddd} For any
	map  $u\in\widehat{H}^{1/2}(\Omega,N)$, $r,\ep>0$ and $k\in\{0,1,\cdots,n\}$,
	we define the $k$-th quantitative singular stratum $S_{\ep,r}^{k}(u)$
	by
	\[
	S_{\ep,r}^{k}(u)\equiv\Big\{ x\in\Omega\ \big|\ u\text{ is not }(k+1,\ep)\text{-symmetric on }D_{s}(x)\text{ for any }r\le s<1\Big\}.
	\]
	Furthermore, we set
	\[
	S_{\ep}^{k}(u):=\bigcap_{r>0}S_{\ep,r}^{k}(u)\quad\text{ and }\quad S^{k}(u)=\bigcup_{\ep>0}S_{\epsilon}^{k}(u).
	\]
\end{definition}

Clearly, by Definition \ref{def:quantitative symmetry}, we have
	\[
S_{\ep,r}^{k}(u)\equiv\Big\{ x\in\Omega: u^e\text{ is not boundary }(k+1,\ep)\text{-symmetric on }B^+_{s}(\mathbf{x})\text{ for any }r\le s<1\Big\}.
\]
It is straightforward to check that $$k'\leq k\text{ or }\ep'\geq\ep\text{ or } r'\leq r\ \Longrightarrow \ S^{k'}_{\ep',r'}(u)\subseteq S^k_{\ep,r}(u).$$ In particular, we have \[S^0(u)\subset S^1(u)\subset\cdots\subset S^n(u)=\Omega.\]
The following lemma shows that $S^k_\ep(u)$ is indeed a quantitative stratification for singular sets of stationary $1/2$-harmonic maps.

\begin{lemma}\label{equi} Suppose $u\in \widehat{H}^{1/2}(\Om,N)$ is a stationary  $1/2$-harmonic map. Then
	\[
	S^{k}(u)=\Big\{x\in\Omega\ \big|\ \text{no tangent maps of }u\text{ at }x\text{ is }(k+1)\text{-symmetric}\Big\}.
	\]
	Consequently, we have
	\[
	S^{0}(u)\subset S^{1}(u)\subset\cdots\subset S^{n-1}(u)\subset{\rm sing}(u).
	\]
\end{lemma}

\begin{proof} For the moment we write
	\[
	\Si^{k}(u)=\Big\{x\in\Omega\ \big|\ \text{no tangent maps of }u\text{ at }x\mbox{ is }(k+1)\mbox{-symmetric}\Big\}.
	\]	
	Suppose $x\in S^{k}(u)$. Then $x\in S_{\ep}^{k}(u)$ for some $\ep>0$.
	Thus, for any boundary $(k+1)$-symmetric map $h$ and
	any $r>0$, there holds
	\[
	\medint_{B^+_{1} }|u^e_{x,r}-h|^{2}d {\bf y}\geq\ep.
	\]
	If $v$ is a tangent map of $u$ at $x$, there exists a sequence
	$r_{i}\to0$ such that $u^e_{x,r_{i}}\to v^e$ in $L_{\loc}^{2}(\R^{n+1}_+)$.
	Then it follows that
	\[
	\medint_{B^+_{1}}|v^e_{x,r}-h|^{2}d {\bf y}\geq\ep,
	\]
	which implies that $v$ is not $(k+1)$-symmetric. Hence $S^{k}(u)\subset\Sigma^{k}(u)$.
	
	For the \sout{reverse} opposite direction, suppose $x\notin S^{k}(u)$. Then there
	exist  sequences of positive numbers $r_{i}>0$ and boundary $(k+1)$-symmetric
	maps $h_{i}\colon\R_{+}^{n+1}\to N$ such that
\begin{equation}\label{eq: temp-1}
	\medint_{B^+_{1}}|u^e_{x,r_i}-h_{i}|^{2}d {\bf y}\leq i^{-1}.
\end{equation}

	Up to a subsequence we can assume that $u_{x,r_{i}}\rightharpoonup v$
	in $\widehat{H}^{1/2}(D_2)$, which implies that $u^e_{x,r_{i}}\to v^e$
	in $L^{2}(B^+_1)$. Then the weak lower semi-continuity of
	$L^{2}$-norm implies that
	\[
	\int_{B^+_{1}}|v^e-h|^{2}\,d{\bf y}\leq\liminf_{i\to\infty}\int_{B^+_{1}}|u^e_{x,r_i}-h_{i}|^{2}\,d{\bf y}=0
	\]
for some  boundary $(k+1)$-symmetric map $h$.
	
	If $r_{i}\to0$, then $v$ is a tangent map and thus is $(k+1)$-symmetric,
	which shows that $x\not\in\Sigma^{k}(u)$. If $r_{i}\to r>0$, then 	by sending $i\to\infty$
and using the lower semi-continuity of $L^{2}$-norm, we can infer from \eqref{eq: temp-1} that
	\[
	\medint_{B^+_{r}(x)}|u^e ( {\bf z})-h({\bf z}-x)|^{2}d{\bf z}=0
	\]
	for some $(k+1)$-symmetric map $h$.
	This will imply that all tangent maps of $u$ at $x$ are $(k+1)$-symmetric,
i.e., $x\notin\Sigma^{k}(u)$. Thus  $S^{k}(u)\supset\Sigma^{k}(u)$. The proof is complete. \end{proof}

The following  Proposition shows that stationary $1/2$-harmonic maps can satisfy $(0,\ep)$-symmetry property  naturally.  
Indeed,  if  $u$ is a stationary  $1/2$-harmonic map and $$\Theta(u^e,B^+_1(\textbf{x}))=\Theta(u^e,B^+_{1/2}(\textbf{x})),$$
 then by the monotonicity formula \eqref{1.8} and the unique continuation property (see e.g. \cite[Theorem 1.2]{Garofalo-Lin-86-Indiana}), $u^e$ must be $0$-symmetric with respect to $\textbf{x}$ in $\R^{n+1}_+$. 
Such a property can be quantitatively preserved under small perturbations (also called rigidity property).

\begin{proposition}\label{pro1}	Fix $\La>0$. For any $\ep>0,$ there exists $\delta_1=\delta_1(n,N,\Lambda,\ep)$ such that, if  $u\in \widehat{H}^{1/2}_\La(D_{8},N)$ 
is a stationary  $1/2$-harmonic map \sout{with} satisfying	
$$\Theta(u^e,B^+_r(\mathbf{x}))-\Theta(u^e,B^+_{r/2}(\mathbf{x}))<\delta_1$$
	for some $x\in D_1$ and $0<r<1$, then $u^e$ is boundary $(0,\ep)$-symmetric on $B^+_r(\mathbf{x})$.
\end{proposition}
\begin{proof} {We argue by contradiction}.
	Suppose that there exists $\ep_0>0$ and
	a sequence of  stationary $1/2$-harmonic  maps $u_i\in \widehat{H}^{1/2}_\La(D_{8},N)$ {and $x_i\in D_1$ such that}
	$$\Theta(u_i^e,B^+_{r_i}(\textbf{x}_i))-\Theta(u_i^e,B^+_{r_i/2}(\textbf{x}_i))<i^{-1},$$  
	but $u_i^e$  is not boundary $(0,\ep_0)$-symmetric on $B^+_{r_i}(\textbf{x}_i)$.
	Let $\bar{u}_i(y)=u_i(x_i+r_iy)$. By scaling invariance we have, for all $s>0$,
	$$\Theta(\bar{u}_i^e,B^+_s(\textbf{0})) =\Theta(u^e_i,B^+_{r_is}(\textbf{x}_i)).$$
	The  monotonicity property \eqref{1.8} of stationary $1/2$-harmonic maps implies that $\{\bar{u}^e_i\}_{i\geq 1}$ are uniformly bounded in $H^1(B^+_2).$ 
	So up to  a subsequence, we may  assume that $\bar{u}^e_i\rightharpoonup v^e$ weakly in $H^{1}(B_1^+)$ and $\bar{u}^e_i\to v^e$ strongly in $L^{2}(B_1^+)$. 
	Then, by Proposition \ref{prop: defect measure},   $v$ is weakly 1/2-harmonic in $B_1$.
	
	Now using the monotonicity formula \eqref{1.8}, for all $i\in \N$, we have
	\begin{align*}
	\int_{B^+_1\backslash B^+_{1/2}}{|\textbf{y}\cdot\nabla \bar{u}^e_i|^2}d\textbf{y} &\le C(\Theta(\bar{u}_i^e,B^+_1 )-\Theta(\bar{u}_i^e,B^+_{1/2})\\
	&{=C\big(\Theta(u_i^e,B^+_{r_i}(\textbf{x}_i))-\Theta(u_i^e,B^+_{r_i/2}(\textbf{x}_i))\big)}\le C/i.
	\end{align*}
	Sending $i\to \infty$ and using the weak convergence of $\bar{u}^e_i\to v^e$ in $H^1(B^+_1)$, we deduce 
	 $$\int_{B^+_1\backslash B^+_{1/2}}{|\textbf{y}\cdot\nabla \bar{v}^e|^2}d\textbf{y}=0,$$
	 so that $v^e$ is radially invariant on $B^+_1(0)\backslash B^+_{1/2}(0)$. This implies that $v^e$ is 0-homogeneous on $\R^{n+1}$ by the unique continuation property (see e.g. \cite[Theorem 1.2]{Garofalo-Lin-86-Indiana}). In turn, the strong convergence of $\bar{u}^e_i\rightarrow v^e$ in $L^2(B_1^+)$ implies that
	\[
	\medint_{B^+_{r_i}(\textbf{x}_i)} |u^e_i-v^e((\cdot-\textbf{x}_i)/r_i)|^2=\medint_{B^+_{1}}|\bar{u}^e_i-v^e|^2\to 0\qquad \text{as }i\to \infty,
	\]
	contradicting with the assumption that $u^e_i $ is not boundary $(0,\ep_0)$-symmetric on $B^+_{r_i}(\textbf{x}_i)$. The proof is complete.
\end{proof}

\subsection{Properties of quantitative stratum} Before discussing the stratified stratum, we \sout{first} recall the notion of  quantitative frame introduced \sout{in} by
\cite{Naber-V-2018}.
\begin{definition}
	Let $\{y_i\}_{i=0}^k \subset D_1(0)$ and $\rho>0$. We say that these points $\rho$-effectively span a $k$-dimensional affine subspace if for all $i=1,\cdots,k$,
	$$\dist\big(y_i,y_0+\Span\big\{y_1-y_0,\cdots,y_{i-1}-y_0\big\}\big)\geq2\rho.$$
	More generally, a set $F\subset D_1(0)$ is said to $\rho$-effectively span a $k$-dimensional affine subspace, if there exist points $\{y_i\}_{i=0}^k\subset F$ which $\rho$-effectively spans a $k$-dimensional affine subspace.
\end{definition}

\begin{remark}\label{rmk: quantitative frame}
{\rm {The advantage of the  quantitative frame is twofold (see the comments right below \cite[Definition 28]{Naber-V-2018}):
\begin{itemize}
\item[(1)]  If $\{y_i\}_{i=0}^k$ $\rho$-effectively spans a $k$-dimensional affine subspace, then for every point 	$x\in y_0+{\rm span}\big\{y_1-y_0,\cdots,y_k-y_0\big\}$,
there exists a unique set of numbers $\{\alpha_i\}_{i=1}^k$ such that
$$x=y_0+\sum_{i=1}^k\alpha_i(y_i-y_0) \quad \text{with}\quad  |\alpha_i|\leq C(n,\rho)|x-y_0|.$$

\item[(2)] Quantitative frame is stable under limiting process: if  $\{y_i^j\}_{i=0}^k$ $\rho$-effectively spans a $k$-dimensional affine subspace for all $j\ge 1$, and $y_i^j\to y_i$ as $j\to \infty$, then $\{y_i\}_{i=0}^k$ also $\rho$-effectively spans a $k$-dimensional affine subspace.
\end{itemize}}}
\end{remark}

Using the notion of quantitative frame, the next proposition shows that for a sufficiently pinched singular stratum of a stationary 1/2-harmonic mapping, if it is of high dimension in essential, then it  satisfies an  one-side Reifenberg approximating property.
\begin{proposition}\label{pro32}Fix $\La>0$.
	For any $\ep,\rho>0$, there exists $\delta_2=\delta_2(n,N,\Lambda,\ep,\rho)>0$ such that the following holds: for any stationary  1/2-harmonic map $u\in \widehat{H}^{1/2}_{\Lambda}(D_8,N)$, suppose $S\subset S^k_{\ep,\delta_2}(u)$ is a subset such that the collection of points
	$$\mathcal{F}:=\left\{y\in S\cap D_1\ \big|\  \Theta(u^e,B^+_{2}(\mathbf{y}))- \Theta(u^e,B^+_\rho (\mathbf{y})) <\delta_2\right\}$$
	 $\rho$-effectively spans a $k$-dimensional affine plane $V\subset\R^n$, then
	$$S\cap D_1\subset D_{2\rho}(V).$$
\end{proposition}
To prove this Proposition,  we denote, for any $k$-dimensional subspace $L$,
$$|L\cdot\nabla f|^2=\sum\limits_{i=1}^{k}|\nabla_{e_i}f|^2,$$
 where $\{e_i\}_{i=1}^{k}$ is an  orthonormal basis of $L$. This quantity measures how far away the function $f$  is translation invariant along the subspace $L$.
\begin{proof}
	By assumption, we assume that  $V$	is  the affine subspace that is spanned by the $\rho$-independent frame  $\{y_{j}\}_{j=0}^{k}\subset \mathcal{F}$, i.e.,
	\[
	V=y_{0}+{\rm span}{\big\{y_{1}-y_{0},\cdots, y_k-y_0\big\}.}
	\]
	Let $x_0\in D_{1}\backslash D_{2\rho}(V)$, $\mathbf{x}_{0}=(x_0,0)$ and $\de_{2}>0$ to be
	determined later. We need to prove that $ {x}_{0}\not\in S_{\ep,\de_{2}}^{k}(u)$. The idea is  to show  that  $u^e$ is almost translation invariant along a $(k+1)$-dimensional subspace in a neighborhood of ${x}_{0}$.
	
	By the definition of $\mathcal{F}$, there holds $$\Theta(u^e,B^+_{2}(\mathbf{y}_i))-\Theta(u^e,B^+_\rho(\mathbf{y}_i))<\de_{2},\qquad \forall\,0\le i\le k.$$
Thus, for $0<r<\rho$ we have that $B^+_{r}(\mathbf{x}_0)\subset B^+_{2}(\mathbf{y}_i)\backslash B^+_{\rho}(\mathbf{y}_i)$
	for every $0\le i\le k$, so that by the monotonicity formula \eqref{1.8},
	\begin{align*}
	\int_{B^+_{r}(\mathbf{x}_0)}|(\textbf{z}-\mathbf{y}_i)\cdot\na u^e(\mathbf{z})|^{2}d\textbf{z}
	&\le\int_{B^+_{2}(\mathbf{y}_i)\backslash B^+_{\rho}(\mathbf{y}_i)}|(\textbf{z}-\mathbf{y}_i)\cdot\na u^e(\mathbf{z})|^{2}d\textbf{z}\\
	&\le {C(n,\rho)\big(\Theta(u^e,B^+_{2}(\mathbf{y}_i))-\Theta(u^e,B^+_\rho(\mathbf{y}_i))\big)}\\
	&\le C(n,\rho)\delta_2.
	\end{align*}
Consequently, by the triangle inequality we deduce
	\[
	\int_{B^+_{r}(\mathbf{x}_0)}|(\mathbf{y}_i-\mathbf{y}_0)\cdot\na u^e(\mathbf{z})|^{2}d\textbf{z} \le 2C(n,\rho)\de_{2},\qquad \forall\,1\le i\le k.
	\]
	Since $\{y_{j}\}_{j=0}^{k}$ is $\rho$-independent, we conclude that
	\begin{equation}
	\int_{B^+_{r}(\mathbf{x}_0)}|\hat{V}\cdot\na u^e(\mathbf{z})|^{2}d\textbf{z}\le C\de_{2}\label{eq: k-small}
	\end{equation}
	for some constant $C=C(n,\rho)>0$, where $\hat{V}={\rm span}\{ {y}_i-y_{0}\}_{i=1}^{k}\subset \R^n$.
	
	On the other hand, for each ${z}\in D_{r}({x}_0)\subset D_{2}\backslash D_{\rho}(V)$,
	let
	\[
	\pi_{V}(z)=y_{0}+\sum_{i=1}^{k}\al_{i}(z)(y_{i}-y_{0})
	\]
	be the orthogonal projection of $z$ in $V$. Then $| {z}-\pi_{V}( {z})|\ge\rho$,
	$|\al_{i}(z)|\le C(n,\rho)$ (see Remark \ref{rmk: quantitative frame}) and
	\[
	\begin{aligned}
\int_{B^+_{r}(\mathbf{x}_0)}|(\textbf{z}-\pi_{V}(\textbf{z}))\cdot\na u^e(\mathbf{z})|^{2}d\textbf{z} 
&\le\int_{B^+_{2}(\mathbf{y}_0)\backslash B^+_{\rho}(\mathbf{y}_0)}|(\textbf{z}-\mathbf{y}_{0})\cdot\na u^e(\mathbf{z})|^{2}d\textbf{z}\\
  &  + C(n, \rho)\sum_{i=1}^{k}\int_{B^+_{2}(\mathbf{y}_i)\backslash B^+_{\rho}(\mathbf{y}_i)}|(\mathbf{y}_{i}-\mathbf{y}_{0})\cdot\na u^e(\mathbf{z})|^{2}d\textbf{z}\\
  &\le C(n,\rho)\de_{2}.
	\end{aligned}
	\]
	Thus, by setting ${h(\mathbf{z})=\displaystyle\frac{\mathbf{z}-\pi_{V}(\mathbf{z})}{|\mathbf{z}-\pi_{V}(\mathbf{z})|}}$, it follows
	that
	\[
	\begin{aligned}&\int_{B^+_{r}(\mathbf{x}_0)}|h(\mathbf{x}_0)\cdot\na u^e(\mathbf{z})|^{2}d\textbf{z} \\
	& \le\int_{B^+_{r}(\mathbf{x}_0)}|h(\textbf{z})\cdot\na u^e(\mathbf{z})|^{2}+\int_{B^+_{r}(\mathbf{x}_0)}|(h(\textbf{z})-h(\mathbf{x}_0))\cdot\na u^e(\mathbf{z})|^{2}d\textbf{z}\\
	& \le C\de_{2}+Cr^{2}\int_{B^+_{r}(\mathbf{x}_0)}|\na u^e(\mathbf{z})|^{2}d\textbf{z}\\
	& \le C\de_{2}+C(n,\rho,\La)r^{n+1}.
	\end{aligned}
	\]
	Now we choose $r=r(n,\La,\rho)\ll\rho$ such that $C(n,\rho,\La)r^{n+1}\le C\de_{2}$.
	Thus
	\[
	\int_{B^+_{r}(\mathbf{x}_0)}|h(\mathbf{x}_0)\cdot\na u^e(\mathbf{z})|^{2}d\textbf{z}\le2C\de_{2}.
	\]
	Together with \eqref{eq: k-small},  we find that
	\[
	\int_{B^+_{r}(\mathbf{x}_0)}|P\cdot\na u^e(\mathbf{z})|^{2}d\textbf{z}\le C(n,\rho,\La)\de_{2}
	\]
for the $(k+1)$-dimensional subspace $P=\hat{V}\oplus \R h(\mathbf{x}_0)\subset\R^n$.	This shows that,  in a small  neighborhood of $\mathbf{x}_{0}$,  $u^e$ is almost translation invariant along a $(k+1)$-dimensional subspace.	
By Lemma \ref{5.3} below, we can choose $\de_{2}=\de_{2}(n,\rho,\La)$ sufficiently small so
	that $x\not\in S^k_{\ep,\de_2}(u)$. This proves Proposition \ref{pro32}.
\end{proof}
\begin{lemma}\label{5.3}
	For any $\ep>0$, there exists $\delta_3=\delta_3(n,N,\Lambda,\ep)>0$ such that if  $u\in \widehat{H}^{1/2}_{\Lambda}(D_8,N)$ is a stationary  $1/2$-harmonic map satisfying
	\begin{align}\label{81}
	\int_{B^+_1 }|P\cdot\nabla u^e|^2<\delta_3
	\end{align}
	for some $(k+1)$-dimensional subspace $P\subset\R^n$, then
	$S^k_{\ep,\overline{r}}(u)\cap D_{1/2}=\emptyset$ for  $\overline{r}=\delta_3^{\frac{1}{2(n-1)}}$. In particular, $0\not\in S^k_{\ep,\overline{r}}(u)$.
\end{lemma}
\begin{proof}	The proof is exactly same as Naber-Valtorta \cite{Naber-V-2018}. Here we sketch it for convenience of the readers.
	First we claim that there is a constant $C_2(n,N,\Lambda)>0$ such that for every $x\in D_{1/2}$, there exists $r_x\in[\overline{r},1/2]$  such that
	\begin{align}\label{80}
	\Theta(u^e,B^+_{r_x} (\mathbf{x}))-\Theta(u^e,B^+_{{r_x}/2} (\mathbf{x}))<\frac{C_2(n,N,\Lambda)}{|\log\delta_3|}.
	\end{align}
	Indeed, if this is not true, then for  $u\in \widehat{H}^{1/2}_{\Lambda}(D_8,N)$ we may assume by choosing a good radii that $\Theta(u^e,B^+_{1/2} (\mathbf{x})) \le C_1(n,N,\Lambda)$. Then
	\begin{align*}
	C_1(n,N,\Lambda)\geq\Theta(u^e,B^+_{1/2} (\mathbf{x}))&\geq\sum_{i=1}^{|\log\overline{r}|+1}(\Theta(u^e,B^+_{2^{-i}} (\mathbf{x}))-\Theta(u^e,B^+_{2^{-i-1}} (\mathbf{x}))\\
	&\geq c(n)C_2(n,N,\Lambda),
	\end{align*}
	which is impossible if we take $C_2(n,N,\Lambda)={2C_1(n,N,\Lambda)}/{{c}(n)}.$ This proves the claim.

	%We now use a contradiction argument to show that, for every $x\in B^+_{1/2}(0)$, $u$ is $(k+1,\ep)$-symmetric on $B^+_{r_x}(x)$.
	We now argue by contradiction. 
	Suppose there exist an $\ep>0$ and a sequence of stationary  $1/2$-harmonic maps  $u_i\in \widehat{H}^{1/2}_\Lambda(D_8)$, together with $(k+1)$-dimensional subspaces $P_i\subset\R^n$, $\de_{3,i}\to 0$, $x_i\in D_{1/2}(0)$ and $r_i\in[\bar{r}_i,1]$, such that $u^e_i$ is not boundary $(k+1,\ep)$-symmetric on $B^+_{r_i}(\mathbf{x}_i)$, and \eqref{80} holds for $u^e_i$ at $x=x_i$, where $\bar{r}_i=\de_{3,i}^{1/2(n-2)}$ and $r_i=r_{x_i}$. Note also that by the definition of $r_x$, we have
	\begin{align}\label{79}
	r_i^{1-n}\int_{B^+_{r_i}(\mathbf{x}_i)}|P_i\cdot\nabla u^e_i|^2<r_i^{1-n}\delta_{3,i}\leq\de_{3,i}^{1/2}.
	\end{align}
		Using a simple rotation, we may assume that the $(k+1)$-dimensional subspaces $P_i$ are fixed by $P$, i.e., $P_i=P$ for all $i$.
	
	Let $v_i(x)=u_i(x_i+r_ix)$. Then we may assume that $v^e_i$ converges weakly in $H^1$ and strongly in $L^2$ to some weakly harmonic map $v^e$. It follows from \eqref{80} that $v^e$ is 0-symmetric by unique continuation (since $v^e$ is harmonic) and is translation invariant with respect to the $(k+1)$-dimensional subspace $P$ by \eqref{79}:
		\[
	\int_{B^+_1}|P\cdot \nabla v^e|^2\leq\liminf_{i\to \infty}\int_{B^+_1}|P\cdot \nabla v^e_i|^2=\liminf_{i\to \infty}r_i^{1-n}\int_{B^+_{r_i}(\mathbf{x}_i)}|P\cdot\nabla u^e_i|^2=0.
	\]
	Since $v^e_i\to v^e$ in $L^2$, this implies that $v^e_i$ is boundary $(k+1,\ep)$-symmetric on $B^+_{1}$ if $i\gg 1$, or equivalently, $u^e_i$ is $(k+1,\ep)$-symmetric with $B^+_{r_i}(\mathbf{x}_i)$. We get the desired contradiction.
\end{proof}

The next result shows that $\Theta(u^e,B^+_\rho(\cdot))$ remains almost constant on all pinched points.
\begin{lemma}\label{5.5}
For any $0<\rho,\eta<1$,  there exists $\de_4=\de_4(n,N,\Lambda,\rho,\eta)>0$  satisfying the following property. Let  $u\in \widehat{H}^{1/2}_{\Lambda}(D_8,N)$ be a stationary  
$1/2$-harmonic map and let $S\subset S^k_{\ep,\delta_4}(u)$. Let
 $$E=\sup_{y\in S\cap D_1(0)}\Theta(u^e,B^+_{2}(\mathbf{y})) .$$
 If the set
 $$\mathcal{F}=\big\{y\in S\cap D_1\ \big|\  \Theta(u^e,B^+_\rho(\mathbf{y}))>E-\delta_4, \ \mathbf{y}=(y,0)\big\}$$
  $\rho$-effectively spans a $k$-dimensional affine subspace $V\subset\mathbb{R}^n$, then
 $$\Theta(u^e,B^+_\rho(\mathbf{x})) \ge E-\eta \qquad \text{for all }  x\in V\cap D_1(0).$$
\end{lemma}
\begin{proof}
We  argue  by contradiction. Suppose $\{u_i\}\subset \widehat{H}^{1/2}_{\Lambda}(D_8,N)$ is a sequence of stationary  
$1/2$-harmonic maps satisfying
$$\sup_{y\in S_i\cap D_1(0)}\Theta(u^e_i,B^+_{2}(\mathbf{y}))=E_i \leq C(n,\Lambda).$$
For each $i\ge 1$, the set 
$$\mathcal{F}_i:=\big\{y\in S_i\cap D_1\ \big|\ \Theta(u^e_i,B^+_\rho(\mathbf{y})) >E_i-i^{-1}\big\}$$ 
contains a subset $\{y^i_j\}_{j=0}^k$ spanning $\rho$-effectively  a $k$-dimemsional affine subspace $V_i\subset\mathbb{R}^n$, and there exists $x_i\in V_i\cap D_2(0)$ such that
\begin{align}\label{fanz}
\Theta(u^e_i,B^+_\rho(\mathbf{x}_i)) \leq E_i-\eta.
\end{align}
The first consequence of the assumption is that
\begin{equation}\label{eq:lemma 312}
\Theta(u^e_i,B^+_{2}(\mathbf{y}_j^i))-\Theta(u^e_i,B^+_\rho(\mathbf{y}_j^i)) <1/i,\qquad \forall\,i\ge 1\text{ and } 0\le j\le k.
\end{equation}
Without loss of generality, we  further assume that, for each $0\le j\le k$,
\[y^i_j\to y_j\quad\text{and}\quad x_i\to x \qquad \text{as }i\to \infty\]
and $V_i$  converges to a $k$-dimensional affine subspace $V$ passing through  $x$.

By Proposition \ref{prop: defect measure}, there exist a  weakly  harmonic map $v^e$ and a defect measure $\nu$ such that up to a subsequence,
$u^e_i\to v^e$ weakly in $H^1(B_2^+)$ and strongly in $L^2(B^+_2)$, and
$$|\na u^e_i|^2d\mathbf{x}\wto \mu^e=|\De v^e|^2d\mathbf{x}+\nu.$$
 Using the monotonicity formula \eqref{1.8} and the stationarity of $u^e_i$, adapting the argument in the proof of Proposition \ref{prop: tangent measure1},  we find that for any fixed $y$, the function
$$r\mapsto \Theta(\mu^e,B^+_r(\mathbf{y})):=\Theta(v^e,B^+_r(\mathbf{y}))+r^{1-n}\nu(B^+_r(\mathbf{y}))$$ is monotonically nondecreasing. 
Thus, by sending $i\to \infty$ in \eqref{eq:lemma 312} and \eqref{fanz}, we obtain
\[\Theta(\mu^e,B^+_{2}(\mathbf{y}_j))=\Theta(\mu^e,B^+_\rho(\mathbf{y}_j))=E,\qquad \text{for all }\, 0\le j\le k\]
and \begin{equation}\label{eq: loss of energy}
\Theta(\mu^e,B^+_\rho(\mathbf{x})) \le E-\eta.
\end{equation}
Moreover, as that of  Proposition \ref{prop: tangent measure1}, we know that  $\mu^e,v^e,\nu$ are  translation invariant along $V$. Hence $\Theta(\mu^e,B^+_\rho(\mathbf{y})) \equiv \Theta(\mu^e,B^+_{2}(\mathbf{y}_j))=E$ for all $ y\in{V\cap D_1(0)}$, which clearly contradicts with \eqref{eq: loss of energy}. The proof is complete.
\end{proof}

The following  Lemma shows that the almost symmetry is preserved under certain pinching condition.
\begin{lemma}\label{5.6}
 For any  $\ep,\rho>0$, there exists $\de_5=\de_5(n,N,\Lambda,\rho,\ep)>0$ satisfying the following property. Suppose $u\in \widehat{H}^{1/2}_{\Lambda}(D_8,N)$ is a stationary  
 $1/2$-harmonic map  satisfying
 $$\Theta(u^e,B^+_1)-\Theta(u^e,B^+_{1/2})<\de_5.$$
 If there is a point $y\in D_3\setminus\{0\}$ such that
\begin{itemize}
\item[(1)] $\Theta(u^e,B^+_1(\mathbf{y}))-\Theta(u^e,B^+_{1/2}(\mathbf{y}))<\de_5$, and
\item[(2)]  $u^e$ is not boundary $(k+1,\ep)$-symmetric on $B^+_r(\mathbf{y})$ for some $r\in[\rho,2]$,
where $\mathbf{y}=(y,0)$, 
\end{itemize}
 then $u^e$  is not boundary $(k+1,\ep/2)$-symmetric on $B^+_r(0)$, or equivalently
 $u$  is not $(k+1,\ep/2)$-symmetric on $D_r( {0})$.  In particular,  
 $$\mathbf{y}\in S^k_{\ep,\rho}(u)\cap D_3\Rightarrow 0\in S^k_{\ep/2,\rho}(u).$$
\end{lemma}

\begin{proof}
	Suppose by contradiction that  $\{u^e_i\}_{i\ge 1}$ is a sequence of stationary $1/2$-harmonic maps  satisfying 
	$$\Theta(u^e_i,B^+_1)-\Theta(u^e_i,B^+_{1/2}) \leq i^{-1},$$ 
	and there exists a sequence $0\not=\mathbf{y}_i=(y_i,0)\in B^+_3$ such that 
	$$\Theta(u^e_i,B^+_1(\mathbf{y}_i))-\Theta(u^e_i,B^+_{1/2}(\mathbf{y}_i)) \leq i^{-1},$$
	 and that for each $i\in \N$, $u^e_i$ is not boundary $(k+1,\ep)$-symmetric on $B^+_r(\mathbf{y}_i)$, but  is boundary $(k+1,\ep/2)$-symmetric on $B^+_r(\textbf{0})$.
	That is, there exists a sequence of boundary $(k+1)$-symmetric maps $h_i$ such that
	$$\medint_{B^+_r(\textbf{0})}|u^e_i-h_i|^2\leq\ep/2.$$
	Up to a subsequence if necessary, we may assume that $\mathbf{y}_i\ri \mathbf{y}\in \overline{B}_3(\textbf{0})$, $u_i\rightharpoonup v$ in $\widehat{H}^{1/2}_{\Lambda}(D_8,N)$ and $u^e_i\to v^e$ in $L^2(B_6^+)$ for some weakly $1/2$-harmonic map $v\in \widehat{H}^{1/2}_{\Lambda}(D_8,N)$, and 
	$h_i\rightarrow h$ in $L^2(B_r^+)$ for some  boundary $(k+1)$-symmetric map $h$. Sending $i\to \wq$ and using the unique continuation principle for harmonic functions, we see that $v^e$ is 0-homogeneous with respect to the origin and $\mathbf{y}$; moreover, by the property of weak convergence we deduce	
	$$\medint_{B^+_r(\textbf{0})}|v^e-h|^2\leq\limsup_{i\ri\oo}\medint_{B^+_r(\textbf{0})}|u_i^e-h_i|^2
	\leq\ep/2.$$
	If $y=0$ (i.e., $y_i\to 0$), then since $N$ is compact, $\|u^e_i\|_{L^\oo}<+\infty$, and we get
	\begin{align*}
		\lim_{i\ri\oo}\medint_{B^+_r(\mathbf{y}_i)}&|u^e_i(\mathbf{x})-h(\mathbf{x}-\mathbf{y}_i)|^2\leq 2\lim_{i\ri\oo}\medint_{B^+_r(\mathbf{y}_i)}|u^e_i(\mathbf{x})-h(\mathbf{x})|^2d\mathbf{x}\\
		&+2\lim_{i\ri\oo}\medint_{B^+_r(\mathbf{y}_i)}|h(\mathbf{x})-h(\mathbf{x}-\mathbf{y}_i)|^2d\mathbf{x}\leq \ep.
	\end{align*}
	If $y\neq0,$ we can infer from the 0-homogeneity of $v$ at 0 and $y$ that  $v$ is  1-symmetric with respect to the line $\R y$ by the standard cone splitting principle, which in turn implies that
	\begin{align*}
		\lim_{i\ri\oo}&\medint_{B^+_r(\mathbf{y}_i)}|u^e_i(\mathbf{x})-h(\mathbf{x}-\mathbf{y}_i)|^2=\lim_{i\ri\oo}\medint_{B^+_r(\textbf{0})}|u^e_i(\mathbf{y}_i+\mathbf{x})-h(\mathbf{x})|^2d\mathbf{x}\\
		&\leq C\lim_{i\ri\oo}\medint_{B^+_r(\textbf{0})}|u^e_i(\mathbf{y}_i+\mathbf{x})-v^e(\mathbf{y}+\mathbf{x})|^2d\mathbf{x}+1.5\medint_{B^+_r(\textbf{0})}|v^e(\mathbf{y}+\mathbf{x})-h(\mathbf{x})|^2d\mathbf{x}\\
		&=C\lim_{i\ri\oo}\medint_{B^+_r(\textbf{0})}|u^e_i(\mathbf{y}_i+\mathbf{x})-v^e(\mathbf{y}+\mathbf{x})|^2d\mathbf{x}+1.5\medint_{B^+_r(\textbf{0})}|v^e(\mathbf{x})-h(\mathbf{x})|^2d\mathbf{x}\\
		&\leq \ep.
	\end{align*}
	In the first inequality we used the elementary inequality $|a+b|^2\le (1+\ep)|a|^2+C_\ep |b|^2$ for any $a,b\in \R^d$ and  $\ep>0$. Hence  we reach a contradiction in both cases.
\end{proof}

%\begin{remark}\label{rmk:on Lemma 313}
%In Section \ref{sec:covering lemma}, we shall repeatedly use the following variant of Lemma \ref{5.6}: Suppose $u\in \widehat{H}^{1/2}_{\Lambda}(D_8,N)$ is a stationary  1/2-harmonic map and $\Theta(u^e,B^+_1(\mathbf{x}))-\Theta(u^e,B^+_{1/2}(\mathbf{x})) <\de_5$. If there is some $y\in D_3(x)$ with $\Theta(u^e,B^+_1(\mathbf{y}))-\Theta(u^e,B^+_{1/2}(\mathbf{y})) <\de_5$, then $\mathbf{y}\in S^k_{\ep,\rho}(u^e)\cap D_3(\mathbf{x})\Rightarrow \mathbf{x}\in S^k_{\ep/2,\rho}(u^e)$.
%\end{remark}

\section{Reifenberg theorems and estimates of Jones' number}\label{sec:Reifenberg theorems}
The main aim of this section is to extend the Jones' number estimate of \cite[Theorem 7.1]{Naber-V-2017} for stationary harmonic maps to stationary 1/2-harmonic maps. This will be one of the key {ingredients} for establishing the rectifiability of each singular strata.  So we first recall Jones'  number  $\beta_2$ which quantifies how close the support of a measure $\mu$ is to a $k$-dimensional affine subspace.

\begin{definition}[\cite{Naber-V-2017}]\label{aaa}
Let $\mu$ be a nonnegative Radon measure on $D_3$. For $k\in\N$,  the $k$-dimensional Jones' $\beta_2$ number is defined \sout{as} by
\begin{align*}
\beta^k_{2,\mu}(x,r)^2={\inf\Big\{\int_{D_r(x)}\frac{d^2(y,V)}{r^2}\frac{d\mu(y)}{r^k}\big|\ V\su\R^n \mbox{ is } k\mbox{ dimensional affine space}\Big\}}
\end{align*}
where $D_r(x)\su D_3$.
\end{definition}

The importance of Jones' $\beta_2$ number can be found from the following two important quantitative Reifenberg theorems established by Naber and Valtorta \cite{Naber-V-2017}.
 \begin{theorem}[{Discrete-Reifenberg, \cite[Theorem 3.4]{Naber-V-2017}}]\label{rre}
 There exist  $\de_6=\de_6(n)>0$ and $C_R(n)>0$ such that the following property holds. Let $\{D_{r_x}(x)\}_{x\in \mathcal{C}}\su D_2\su\R^n$ be a family of pairwise disjoint balls with centers in $\mathcal{C}\subset D_1$ and let $\mu\equiv\sum_{x\in \mathcal{C}}\omega_kr_x^k\de_x$ be the associated measure. If for every 
$D_r(x)\su D_2,$
 there holds
 \begin{align}\label{eq: multiscale appro-1}
\int_{D_r(x)}\Big(\int_0^r\beta^k_{2,\mu}(y,s)^2\frac{ds}{s}\Big)d\mu(y)<\de_6^2r^k,
\end{align}
then 
 \begin{align*}
\sum_{x\in \mathcal{C}}r^k_x< C_R(n).
\end{align*}
 \end{theorem}

Another quantitative Reifenberg theorem is \sout{as} follows.
\begin{theorem} [{Rectifiable-Reifenberg, \cite[Theorem 3.3]{Naber-V-2017}}]\label{rree}
 There exist constants $\de_7=\de_7(n)$ and $C=C(n)$ {such that the following property
 holds}. Assume that $S\subset D_2\su\R^n$ is $\HH^k$-measurable, and for each $D_r(x)\subset D_2$ there holds
 \begin{align}\label{eq: multiscale appro-2}
\int_{S\cap D_r(x)}\Big(\int_0^r\beta^k_{2,\HH^k|_S}(y,s)^2\frac{ds}{s}\Big)d\HH^k(y)<\de_7^2r^k.
\end{align}
Then $S\cap D_1$ is  $k$-rectifiable, and   $\HH^k(S\cap D_r(x))\leq C r^k$ for each $x\in S\cap D_1$.
 \end{theorem}
We remark that Mi\'skiewicz \cite{Miskiewicz-18-Finn}  improved the above two theorems.
The two conditions \eqref{eq: multiscale appro-1} and \eqref{eq: multiscale appro-2} are usually called multi-scale approximation conditions. 
Similar to  stationary harmonic maps \cite[Theorem 7.1]{Naber-V-2017},  we need to establish an $L^2$-subspace approximation theorem
{for stationary $1/2$-harmonic maps}.

For $x\in D_1, \mathbf{x}:=(x,0)$, and $r>0,$ we denote 
 \begin{align*}
 W_r({x}):= W_{r,8r}({x})=\int_{B^+_{8r}({\mathbf{x}})\ba B^+_r({\mathbf{x}})}\frac{|(\mathbf{y}-\mathbf{x})\cdot\na u^e(\mathbf{y})|^2}{|\mathbf{y}-\mathbf{x}|^{n+1}}d\mathbf{y}.
\end{align*}
%The main conclusion of this subsection is the following theorem. Except for the monotonicity formula, this theorem and its proof are much the same as Theorem 6.1 and so we only list the theorem here.

\begin{theorem}\label{ji}
For any $\ep>0$, there exist $C=C(n,N,\Lambda,\ep)>0$ and $\de_8=\de_8(n,N,\Lambda,\ep)>0$ such that, for any  stationary $1/2$-harmonic map $u\in \widehat{H}^{1/2}_\La(D_{10},N)$ and $0<r\leq1$,  and $x\in D_1$ with $\mathbf{x}:=(x,0)$, if $u^e$ is boundary $(0,\de_8)$-symmetric  but not $(k+1,\ep)$-symmetric on $B^+_{8r}(\mathbf{x})$, then for any nonnegative finite measure $\mu$ on $D_r(x)$,  we have
\begin{align}\label{po}
\beta^k_{2,\mu}({x},r)^2\leq Cr^{-k}\int_{D_r(x)}W_r({y})d\mu({y}).
\end{align}
\end{theorem}

The proof of  {Theorem \ref{ji} is} similar to  \cite[Theorem 7.1]{Naber-V-2017}. We sketch the proof for the convenience of readers.  
 Assume that $x=0,r=1$ and $\mu$ is  a probability measure supported on  $D_1$. Let $ {x}_{cm}=\int_{D_1} xd\mu(x)$ 
be the mass center of $\mu$ in  $D_1$. The second moment $Q(\mu)$ of $\mu$ is the symmetric bilinear form  defined by
\begin{align*}
	Q(\mu)(v,w):=\int_{D_1}\big(( {x}- {x}_{cm})\cdot v\big)\big(({x}- {x}_{cm})\cdot w\big)d\mu( {x}),\qquad \text{for all }\ v,w\in \R^n.
\end{align*}
Let $\lambda_1(\mu)\ge \cdots\ge \lam_n(\mu)$ be nonincreasing eigenvalues of $Q(\mu)$ and $v_1(\mu),\cdots,v_n(\mu)$ be the associated eigenvectors. Then we have
\begin{align}\label{il}
	Q(\mu)(v_k)=\lam_kv_k=\int_{D_1}\big(( {x}- {x}_{cm})\cdot v_k\big) ( {x}- {x}_{cm}) d\mu( {x}).
\end{align}
Recall that the eigenvalues {can be characterized} by variational method, that is,
\begin{align*}
	\lam_1=\lam_1(\mu):=\max_{v\in \mathbb{S}^{n-1}}\int_{D_1}|( {x}- {x}_{cm})\cdot v|^2d\mu( {x}).
\end{align*}
Let $v_1=v_1(\mu)\in\mathbb{S}^{n-1}$ be any unit vector achieving such a maximum. By induction, we have
\begin{align*}
	\lam_{k+1}=\lam_{k+1}(\mu):=\max\Big\{\int_{D_1}|( {x}- {x}_{cm})\cdot v|^2d\mu( {x})\big|\ v\in \mathbb{S}^{n-1},\ v\cdot v_i=0\ \text{for all }i\leq k\Big\},
\end{align*}
 and let $v_{k+1}=v_{k+1}(\mu)\in \mathbb{S}^{n-1}$ be any unit vector achieving such a maximum. 
 Note that by definition of  $v_k$, $V_k= {x}_{cm}+\mbox{span}\{v_1,\cdots,v_k\}$ is the $k$-dimensional affine subspace achieving the minimum in the definition of $\beta_2$; see \cite[Remark 49]{Naber-V-2018}. Moreover,
\begin{align*}
\beta_{2,\mu}^k(0,1)^2=\int_{D_1} d^2( {x},V_k)d\mu( {x})=\lam_{k+1}(\mu)+\cdots+\lam_n(\mu).
\end{align*}
In fact, the second equality follows from the definition of $\la_k$ and
$$d^2( {x},V_k)=\sum_{i=k+1}^{n}(( {x}- {x}_{cm})\cdot v_i)^2.$$
The following property gives the relationship between $\lam_k,v_k$ and $W_1$. Similar to \cite[Proposition 50]{Naber-V-2018}, we have
\begin{proposition}\label{hj}
	Let  $\mu$ be a probability measure on $D_1$ and $u\in \widehat{H}^{1/2}(D_{10}, N)$. Let $\lam_k,\ v_k$ be defined  as  above. Then there exists $C(n)>0$ such that
	\begin{align*}
		\lam_k\int_{A^+_{2,4}}|(v_k,0)\cdot\na u^e(\mathbf{z})|^2d\mathbf{z}\leq C(n)\int_{D_1} W_1({x})d\mu({x}), \qquad \text{for all }\,k\ge 1,
	\end{align*}
	where $A^+_{2,4}=B^+_4(0)\backslash B^+_2(0)$.
\end{proposition}

\begin{proof} Without loss of generality, we assume  $x_{cm}=0.$ For any $\textbf{z}=(z,t)\in \R^{n+1}$  and $k=1,\cdots,n,$ multiplying 
{both sides of \eqref{il} by $\na u^e(\textbf{z})$} yields 	
\begin{align}\label{aam} 		\lam_k((v_k,0)\cdot\na u^e(\textbf{z}))=\int_{D_1}({x}\cdot v_k)(\na u^e(\textbf{z})\cdot \mathbf{x})d\mu(x)	\end{align}
with $\mathbf{x}=(x,0).$	By definition of mass center,
	\begin{align*}		\int_{D_1} \mathbf{x}\cdot \textbf{z}d\mu(x)=\int_{D_1} {x}\cdot {z}d\mu(x)={x}_{cm}\cdot {z}=0.	\end{align*}
Hence by  \eqref{aam}, 	
\begin{align*}		\lam_k((v_k,0)\cdot\na u^e(\textbf{z}))=\int_{D_1}({x}\cdot v_k)(\na u^e(\textbf{z})\cdot (\mathbf{x}-\textbf{z}))d\mu(x).	\end{align*}
By H\"{o}lder's inequality we have
	\begin{align*}		\lam_k^2|(v_k,0)\cdot\na u^e(\textbf{z})|^2\leq\lam_k\int_{D_1}|\na u^e(\textbf{z})\cdot (\mathbf{x}-\textbf{z})|^2d\mu(x).	\end{align*}
Without loss of generality we can assume $\lam_k>0$, otherwise there is nothing to prove. Direct computation gives
	\begin{align*}		\lam_k\int_{A^+_{2,4}(0)}|\na u^e(\textbf{z})\cdot (v_k,0)|^2d\textbf{z}&\leq \int_{D_1}\int_{A^+_{2,4}(0)}|\na u^e(\textbf{z})\cdot(\mathbf{x}-\textbf{z})|^2d\textbf{z}d\mu(x)\\	
	%	&\leq \int\int_{A^+_{3,4}(0)}\frac{|\na u^e(\textbf{z})\cdot(\mathbf{x}-\textbf{z})|^2}{|\mathbf{x}-\textbf{z}|^{n+1}}|\mathbf{x}-\textbf{z}|^{n+1}d\textbf{z}d\mu(x)\\	
			&\leq C(n) \int_{D_1}\int_{A^+_{1,8}(x)}\frac{|\na u^e(\textbf{z})\cdot(\mathbf{x}-\textbf{z})|^2}{|\mathbf{x}-\textbf{z}|^{n+1}} d\textbf{z}d\mu(x)\\	
				&\leq C(n)\int_{D_1}W_1(x)d\mu(x).	\end{align*}
	This completes the proof.\end{proof}

To further estimate the left hand side of the  inequality in Proposition \ref{hj}, we need the following Lemma.
\begin{lemma}\label{lem: non higher order sym}
	For any $\ep>0$, there exists  $\de_9=\de_9(n,N,\La,\ep)>0$ such that the following property
	{holds}. For any stationary $1/2$-harmonic map $u\in \widehat{H}^{1/2}_\La(D_8,N)$, if $u^e$ is boundary $(0,\de_9)$-symmetric on $B^+_1(0)$ but not  boundary $(k+1,\ep)$-symmetric, then 	
	\begin{align}\label{81}
	\int_{A^+_{2,4}}|P\cdot\nabla u^e(\mathbf{z})|^2d\mathbf{z}>\delta_9
	\end{align}
	for every $(k+1)$-dimensional subspace $P\subset\R^n$.
\end{lemma}
\begin{proof} We argue by contradiction. Suppose that there is a sequence of stationary $1/2$-harmonic maps $u_i\in \widehat{H}^{1/2}_\La(D_8,N)$ such that $u^e_i$ is boundary $(0,i^{-1})$-symmetric  on $B^+_1(0)$ but not boundary $(k+1,\ep_0)$-symmetric for some $\ep_0>0$. 
	Moreover, after an orthogonal transformation, there is a  $(k+1)$-dimensional subspace $P$ such that
	\begin{align}\label{999}
	\int_{A^+_{2,4}}|P\cdot\nabla u^e_i(\mathbf{z})|^2d\mathbf{z}\leq i^{-1},\ i\ge 1.
	\end{align}
	By extracting a subsequence, we can assume $u^e_i\rightharpoonup v^e$ in $W^{1,2}(B^+_8)$ and $u^e_i\ri v^e$ in $L^{2}(B^+_8)$ for some 0-homogeneous harmonic map $v^e\in W^{1,2}(B^+_8)$ and
	\begin{align*}
	\int_{A^+_{2,4}}|P\cdot\nabla v^e(\mathbf{z})|^2d\mathbf{z}=0.
	\end{align*}
	By the unique continuation property we know
	$$\int_{\R^{n+1}_+}|P\cdot\nabla v^e(\mathbf{z})|^2d\mathbf{z}=0.$$
	Thus $v^e$  is boundary $(k+1)$-symmetric on $B^+_1(0)$. Consequently $u^e_i$ is boundary $(k+1,\ep_0)$-symmetric for $i\gg 1$, 
	since $u^e_i$ converges strongly to $v^e$ in $L^2$. We get a desired contradiction. The proof is complete.
\end{proof}
 Now we can prove Theorem \ref{ji}.
\begin{proof}[Proof of Theorem \ref{ji}]  By scaling, we may assume $\mu(D_1(0))=1.$ Since $\lambda_k$ is nonincreasing,
\begin{align}\label{pl}
	\beta_{2,\mu}^k(0,1)^2=\lam_{k+1}+\cdots+\lam_n\leq(n-k)\lam_{k+1}.
\end{align}
Thus it suffices to estimate $\la_{k+1}$.  By Proposition \ref{hj}, we have
\begin{align*}
	\sum_{j=1}^{k+1}\lam_j\int_{A^+_{2,4}(0)}|\na u^e(\textbf{z})\cdot (v_j,0)|^2d\textbf{z}\leq(k+1)C\int_{D_1} W_1(x)d\mu(x).
\end{align*}
Let $V^{k+1}=\mbox{span}\big\{v_1,\cdots,v_{k+1}\big\}$. Then
\begin{align*}
	\lam_{k+1}\int_{A^+_{2,4}(0)}|V^{k+1}\cdot\na u^e(\textbf{z})|^2d\textbf{z}&=\lam_{k+1}\sum_{j=1}^{k+1}\int_{A^+_{2,4}(0)}|\na u^e(\textbf{z})\cdot (v_j,0)|^2d\textbf{z}
	\\
	&\le {\sum_{j=1}^{k+1}\lam_j\int_{A^+_{2,4}(0)}|\na u^e(\textbf{z})\cdot (v_j,0)|^2d\textbf{z}}\\
	&\leq C\int_{D_1} W_1(x)d\mu(x).
\end{align*}
Let $\de_9>0$ be the number {given by} Lemma \ref{lem: non higher order sym} and set $\delta_8=\delta_9$. 
Combining our assumption and  Lemma \ref{lem: non higher order sym} yields
\begin{align*}
	\int_{A^+_{2,4}(0)}|V^{k+1}\cdot\na u^e(\textbf{z})|^2d\textbf{z}\geq \de_8,
\end{align*}
and so
\begin{align*}
	\de_8\lam_{k+1}\leq\lam_{k+1}\int_{A^+_{2,4}(0)}|V^{k+1}\cdot\na u^e(\textbf{z})|^2d\textbf{z} \leq C\int _{D_1}W_1(x)d\mu(x).
\end{align*}
Since $\de_8=\de_9$ depends only on $n,N,\La,\ep$, by \eqref{pl}, we conclude
\begin{align*}
	\beta_{2,\mu}^k(0,1)^2\leq C(n,N,\La,\ep)\int_{D_1} W_1(x)d\mu(x).
\end{align*}
This completes the proof.
\end{proof}

\section{Covering lemma}\label{sec:covering lemma}
Following the  approach of Naber-Valtorta \cite[Section 6.2]{Naber-V-2018},  this section is devoted to  the following volume estimate for singular set.

\begin{lemma}[Main covering Lemma]\label{lemma:main covering lemma}
	Let $u\in \widehat{H}_{\La}^{1/2}(D_8,N)$ be a stationary weakly $1/2$-harmonic map.
	Fix any $\ep>0$ and $0<r<R\leq1$. Then there exist constants $\de=\de(n,N,\La,\ep)>0$ and $C(n)$ {such that} the following property
	{holds}. For any subset
	${S}\su S^k_{\ep,\de r}(u),$ there exists a finite covering of $S\cap D_R(0)$ satisfying
	\begin{align*}
	S\cap D_R(0)\su \bigcup_{x\in\3}D_{r_x}(x) \quad\text{with}\quad   r\le  r_x\le R
	\end{align*}
	and
		\begin{align}\label{pla}
 \sum_{x\in\3}r_x^k\leq C(n)R^k.
	\end{align}
	Moreover, the balls in $\{D_{r_x/5}(x)\}_{x\in \3}$ are pairwise disjoint and $\cal{C}\subset 	S\cap D_R(0) $.
\end{lemma}

The idea of the proof of Lemma \ref{lemma:main covering lemma} can be seen from the following  Lemma.
\begin{lemma}[Energy drop]\label{ti}
	Let $u\in \widehat{H}_\Lambda^{1/2}(D_8,N)$ be a stationary $1/2$-harmonic map.
	Fix any $\ep>0$ and $0<r<R\leq1$. There exist $\de=\de(n,N,\La,\ep)>0$ and $C_2(n)$ such that  for any subset
	${S}\su S^k_{\ep,\de r}(u)$,  there exists a finite covering of $S\cap D_R(0)$ satisfying
	\[
	S\cap D_R(0)\su \bigcup_{x\in\3}D_{r_x}(x) \quad\text{with}\quad  r\le  r_x\le R
	\]
	and
	\[\sum_{x\in\3}r_x^k\leq C_2(n)R^k.\]
	Moreover, for each $x\in\3,$ one of the following conditions  is satisfied:
	\begin{itemize}
		\item[i)] $r_x=r;$
		\item[ii)] we have the following uniform energy drop property:
		\begin{equation}\label{110}
		\sup_{y\in D_{r_x}(x)\cap S}\Theta(u^e,B^+_{2r_x}(\mathbf{y}))\leq E-\de,
		\end{equation}
		where $\mathbf{y}=(y,0)$ and $$E:=\sup_{x\in D_{R}(0)\cap{S}}\Theta(u^e,B^+_{2R}(\mathbf{x}))\qquad \text{with}\qquad \mathbf{x}=(x,0).$$
	\end{itemize}
\end{lemma}

With the help of Lemma \ref{ti}, Lemma \ref{lemma:main covering lemma} {can be proven} as follows.

\begin{proof}[Proof of Lemma \ref{lemma:main covering lemma}]	Note that the energy $E$ defined as in Lemma \ref{ti} satisfies $E\leq C\Lambda$.  So iterating Lemma \ref{ti} by at most $i=([\de^{-1}E]+1)$-times, we could obtain a covering $\{D_{r_x}(x)\}_{x\in \3^i}$ of $S\cap D_R(0)$ such that $r_x\le r$ and
	\[
	\sum_{x\in \3^i}r_x^k\leq C_2(n)^iR^k.
	\]
	We may assume that $x\in S\cap D_R(0)$ by considering the larger covering $\{D_{2r_x}(x)\}_{x\in \tilde{\cal C}^i}$. Then
	\[ S\cap D_R(0) \subset \bigcup_{x\in \tilde{\cal C}^i} D_{2 r_x}(x)\qquad \text{ and } \quad  \sum_{x\in {\tilde{\cal{C}}^i}}(2 r_x)^k\leq 2^kC_2(n)^iR^k.\]
	Finally, since $r_x\le 1$ for all $x\in \tilde{\cal C}^i$, we can use Vitali's covering lemma  to select a family of disjoint balls $\{D_{2r_x}(x)\}_{x\in \cal{C}}$ from that of $\tilde{\cal C}^i$ so that
	\[ S\cap D_R(0) \subset \bigcup_{x\in \cal{C}} D_{10r_x}(x)\qquad \text{ and } \quad  \sum_{x\in \cal{C}}(10r_x)^k\leq 10^kC_2(n)^iR^k.\]
	The proof is complete upon taking $C(n)=10^kC_2(n)^i$ and relabelling the balls.
\end{proof}

Therefore, the main task below is to prove Lemma  \ref{ti}. To this end, the auxiliary Lemma  \ref{oooa} below forms the first step. Before proceeding, let us make the following\\
\textbf{Convention}:  A family of balls $\{ B_{r_x}(x) \}_{x\in {I}} $ is called a ``Vitali covering" of a set $S$, if
\[S\subset \bigcup_{x\in I} B_{r_x}(x) \qquad\text{and}\qquad B_{r_x/5}(x)\cap B_{r_y/5}(y) =\emptyset,\quad \forall\,x,y\in I,\ {x\not=y}.  \]

\subsection{An auxiliary covering Lemma}
We first establish the following Lemma.
\begin{lemma}\label{oooa} For any $\ep>0$ and $0<\rho\leq100^{-1}$, there exist constants $\de=\de(n,N,\La,\rho,\ep)>0$ and $C_1(n)$ 
such that the following property {holds}.
For any stationary  $1/2$-harmonic map $u\in \widehat{H}_{\La}^{1/2}(D_8,N)$	 and $0<r<R\leq1$,   let
	$${S}\su S^k_{\ep,\de r}(u)\qquad \text{and}\qquad E=\sup\limits_{x\in S \cap D_{R}(0) }\Theta(u^e,B^+_{2R}(\mathbf{x})), $$ where $\mathbf{x}=(x,0)$.
Then  there exists a finite covering of $S\cap D_R(0)$ such that
	\begin{align*}
	S\cap D_R(0)\su \bigcup_{x\in\3}D_{r_x}(x) \quad\text{with}\quad  R\ge r_x\geq r
	\end{align*}
	and
		\begin{align}\label{pla}
 \sum_{x\in\3}r_x^k\leq C_1(n)R^k.
	\end{align}
	Moreover, for each $x\in\3,$ one of the following conditions  is satisfied:
\begin{itemize}
\item[i)] $r_x=r$;
\item[ii)] the set
	$$\mathcal{F}_x=\Big\{y\in S\cap D_{r_x}(x):\Theta(u^e,B^+_{\rho r_x/10}(\mathbf{y}))>E-\de\Big\}$$
	is contained in $D_{\rho r_x/5}(L)\cap D_{2r_x}(x),$ where $L$ is some  $(k-1)$-dimensional affine subspace in $\R^n$.
\end{itemize}
\end{lemma}
Later we will choose $\rho=\rho(n)$ in the proof of Lemma \ref{ti}, see \eqref{eq:  definition of rho}.
\begin{proof}
	By the scaling invariance property of $\Theta$, we may assume $R=1$. To simplify the presentation, we also assume 
	 that
	\begin{align}\label{332}
	r=\rho^{{l}}, \quad \rho=2^{-a},\quad a,\ {l}\in\N.
	\end{align}

	We divide the proof into two main  steps.
	\begin{itemize}
	\item [{\bf Step1}] (Inductive covering) Let $\eta>0$ be a constant to be determined later. Define the initial set
	\[
	\mathcal{F}^0=\{y\in D_1\cap{S}:\Theta(u^e,B^+_{\rho /10}(\mathbf{y}))\geq E-\de\}.
	\]
	If there is a $(k-1)$-dimensional affine subspace $L_0$ such that $\mathcal{F}^0\su D_{\rho/5}(L_0),$ then we say that $D_1$ is a {\it good ball},
	{and} our claim clearly holds.
Therefore, we will assume that $D_1$ is a {\it bad ball}, that is, the set $ \mathcal{F}^0$  spans $({\rho}/{10})$-effectively   a  $k$-dimensional affine subspace $V_0\subset \R^n$ .
In this case, we choose $\de\leq\de_2$ according to Proposition \ref{pro32} so that
	\begin{equation*}
 S\cap D_1\su D_{\rho/5}(V_0).
	\end{equation*}
We then choose a finite Vitali covering of $D_{\rho/5}(V_0)\cap D_1$ by balls $\{D_\rho(x)\}_{x\in\3^1}$ with $\3^1\su V_0\cap D_1$ such that
	$$S\cap D_1\su\bigcup_{x\in \3^1}D_{\rho}(x)\cap D_1.$$
	Note that by Lemma \ref{5.5}, for all $x\in\3^1\su V_0\cap D_1$, we have
	$$\Theta(u^e,B^+_{\rho /10}(\mathbf{x}))\geq E-\eta,$$
	as long as $\de\le \de_4$ which depends also on $\eta$. Under the same smallness assumption, Lemma \ref{5.6}  implies that for each $x\in \3^1$ we have $x\in S^k_{\ep/2,\rho}(u).$	

Now we divide the above covering according the following rules: For each $x\in\3^1$,  set
	\begin{align*}
\mathcal{F}_x^1=\left\{y\in S\cap D_{\rho}(x):\Theta(u^e,B^+_{\rho^2 /10}(\mathbf{y}))\geq E-\de\right\},
\end{align*} and  let
	\begin{align*}
	\3^1_g=\Big\{x\in\3^1\big| \mathcal{F}_x^1 \su D_{ \rho^2/5}(L_x^1) \
	\mbox{\rm for a}\ (k-1)\mbox{\rm -dimensional affine subspace}\ L_x^1\Big\},
	\end{align*}
	\begin{align*}
	\3^1_b=\Big\{x\in&\3^1\big| \mathcal{F}_x^1\ (\rho^2/10)\mbox{-effectively spans a}\ k\mbox{-dimensional affine subspace}\ \bar{L}_x^{1}\Big\}.
	\end{align*}
	Then we have $	\3^1=	\3^1_g\cup 	\3^1_b$ and
		$$S\cap D_1(0)\subset \bigcup_{x\in \3^1_b}D_{\rho}(x)\cup\bigcup_{x\in \3^1_g}D_{\rho}(x)\equiv D_{\rho}(\3^1_b)\bigcup D_{\rho}(\3^1_g).$$
		Hence the remaining work is to deal with bad balls. Consider any bad ball $D_\rho(x)$.
	By Proposition \ref{pro32}, we have
	$$S\cap D_\rho(x)\su D_{ \rho^2/5}(\bar{L}_x^{1})$$ for some $k$-dimensional affine subspace $\bar{L}_x^{1}$.
	Proceeding  as above we obtain
	$$S\cap D_\rho(x)\su\bigcup_{y\in \3^2_x}D_{\rho^2}(y)=\bigcup_{y\in \3^2_{x,b}}D_{\rho^2}(y)\cup\bigcup_{y\in \3^2_{x,g}}D_{\rho^2}(y)\equiv D_{\rho^2}(\3^2_{x,b})\bigcup D_{\rho^2}(\3^2_{x,g}),$$
	with $\3^2_x=\3^2_{x,b}\cup\3^2_{x,g}\su \bar{L}_x^{1}\cap D_\rho(x)$ so that $\{D_{\rho^2}(y):y\in \cal{C}^2_x \} $ is a Vitali covering.
	Moreover, {for $y\in \mathcal{C}_x^2$,  if we set}
	\begin{align*}
	\mathcal{F}_y^2=\Big\{z\in S\cap D_{2\rho}(y)\ \big|\Theta(u^e,B^+_{\rho^3 /10}(\textbf{z}))\geq E-\de\Big\},
	\end{align*}
	then
	\begin{align*}
	\3^2_{x,g}=&\Big\{y\in\3^2_{x}\ \big| \mathcal{F}_y^2 \su D_{ \rho^3/5}(L_y^2) \
	 \mbox{\rm for a}\ (k-1)\mbox{\rm -dimensional affine subspace}\ L_y^2\Big\}	
	\end{align*}
	and
	\begin{align*}
	\3^2_{x,b}=\Big\{y\in\3^2_x\ \big| \mathcal{F}_y^2\ (\rho^3/10)\mbox{-}\mbox{\rm effectively spans a}\ k\mbox{\rm -dimensional affine subspace}\ \bar{L}_y^{2}\Big\}.
	\end{align*}
	Set $\3_b^2=\bigcup\limits_{x\in\3^1_b}\3^2_{x,b}$, $\3^2_g=(\bigcup\limits_{x\in\3^1_b}\3^2_{x,g})\cup\3^1_g$, and $\3^2=\3^2_g\cup\3^2_b$. Then
	$$S\cap D_1(0)\su\bigcup_{x\in \3^2}D_{r^2_{x}}(x)=\bigcup_{x\in \3^2_g}D_{r^2_{x}}(x)\cup\bigcup_{x\in\3^2_b}D_{r^2_{x}}(x)\equiv D_{r^2_{x}}(\3^2_g)\bigcup D_{\rho^2}(\3^2_b),$$
	where $r^2_{x}=\rho$ if $x\in\3^1_g$ and $r^2_{x}=\rho^2$ if $x\in\3^2_g\ba\3^1_g$ or $\3^2_b$.
	Furthermore, with the same reasoning as the initial covering, there holds that
	for all $x\in\3^2$,  $\Theta(u^e,B^+_{r^2_{x}}(\mathbf{x}))\geq E-\eta$;  and
	for all $s\in[r^2_{x},1]$, $u$ is not $(k+1,\ep/2)$-symmetric on $D_{s}(x)$, that is, $x\in S^k_{\ep/2,r^2_{x}}(u).$
	
	Repeating this procedure, we then build a covering of the form
	\[
	S\cap D_1(0)\su \bigcup_{x\in \3^j}D_{r^j_{x}}(x)=\bigcup_{x\in \3^j_g}D_{r^j_{x}}(x)\cup\bigcup_{\3^j_b}D_{r^j_{x}}(x)\equiv D_{r^j_{x}}(\3^j_g)\bigcup D_{r^j_{x}}(\3^j_b),
	\]
	with $$\3_b^j=\bigcup\limits_{x\in\3^{j-1}_b}\3^j_{x,b}, \qquad \3^j_g=\Big(\bigcup\limits_{x\in\3^{j-1}_b}\3^j_{x,g}\Big)\cup\3^{j-1}_g$$
	 and $\3^j=\3^j_b\cup\3^j_g,$ for $j\geq2$. Notice that
	\[
	r^j_{x}=\begin{cases}
	\rho^i, & \text{ if } x\in\3^i_g\ba\3^{i-1}_g,i=2,\cdots,j,\\
	\rho^j, & \text{ if } x\in\3^j_b.
	\end{cases}
	\]
	Moreover, $\3^j_{g}$ is  given by
	\begin{align*}
	\Big\{x\in\3^j\ \big|\ r^j_{x}\geq\rho^j\ \mbox{\rm and}\ \mathcal{F}_x^j\su D_{ \rho r^j_{x}/5}(L_x^j) \
	\mbox{\rm for a}\ (k-1)\mbox{\rm -dimensional affine subspace}\ L_x^j\Big\},
	\end{align*}
	where
	\begin{align*}
	\mathcal{F}_x^j=\Big\{y\in S\cap D_{2r^j_{x}}(x)\ \big|\Theta(u^e,B^+_{\rho r^j_{x}/10}(\mathbf{y}))\geq E-\de\Big\},
	\end{align*}
	and
	\begin{align*}
	\3^j_{b}=\Big\{x\in\3^j\ \big| & r^j_{x}=\rho^j, 
	\ \mathcal{F}_x^j\ {\rho r^j_{x}}/{10}\mbox{-}\mbox{\rm effectively spans a}\ k\mbox{\rm -dimensional affine subspace}\ \bar{L}_x^{j}\Big\}.
	\end{align*}
	Furthermore,  for all $x\neq y\in\3^j\backslash \3_g^{j-1}$, we have $D_{r_x^j/5}(x)\cap D_{r_x^j/5}(y)=\emptyset$ and
	\begin{itemize}
		\item[(P1)] For all $x\in\3^j,$  $\Theta(u^e,B^+_{r^j_{x}}(\mathbf{x})))\geq E-\eta.$
		\item[(P2)] For all $x\in\3^j$ and for all $s\in[r^j_{x},1]$, $u$ is not $(k+1,\ep)$-symmetric on $D_{s}(x)$, that is, $x\in S^k_{\ep/2,r^j_{x}}(u).$
	\end{itemize}
	Taking $j=l$ and $\rho^{l}=r$ in \eqref{332} and $\3=\3^{l}$. Then the covering part of Lemma \ref{oooa} is almost complete, except that balls with centers in $\3$ does not necessarily satisfy the disjointness condition $D_{r_x^i/5}(x)\cap D_{r_y^j/5}(y)=\emptyset$ for all $x,y\in \3$. However, we may do a further covering for the collection of balls $\{D_{r_x}(x)\}$ to select a subcollection of disjoint balls $\{D_{r_x}(x)\}$ such that $\{D_{5r_x}(x)\}$ satisfies all the covering requirements of the Lemma. Thus, by relabeling these balls if necessary, we may assume the collection $\{D_{r_x}(x)\}_{x\in \3}$ satisfies the covering part of Lemma \ref{oooa}.
\item[\textbf{Step2}] (Reifenberg estimates)
	In order to prove the volume estimate \eqref{pla}, we define a measure
	\[
	\mu:=\omega_k\sum_{x\in\3}r_x^k\de_x.
	\]
	and the associated  measures
	\[
	\qquad \qquad  \mu_t:=\omega_k\sum_{x\in\3_t}r_x^k\de_x,\]
	where $\3_t=\{x\in\3:r_x\leq t\}$ for all $ 0<t\le 1$.  Note that $\cal{C}\subset D_1$ and $\mu_t$ is supported on the discrete set $\cal{C}_t$ for all $r\le t\le 1$.

	Set $r_j=2^jr,j=0,1,\cdots,al-3$. Then $r_{al-3}=1/8$ by \eqref{332}.
	Since $D_{r/5}(x)\cap D_{r/5}(y)=\emptyset$ for all $x,y\in\3_r,$ it follows \sout{easily} that 
	$$\mu_r(D_r(z))\leq c(n)r^k, \ \forall z\in D_1.$$
	Assume now for all   we have
	\begin{align}\label{1300}
	\mu_{r_j}(D_{r_j}(x))\leq C_R(n)r_j^k,\qquad \forall\,x\in D_1,
	\end{align}
	where $C_R(n)$ is the constant in Theorem \ref{rre}.  We next show \eqref{1300} is true for ${j+1}$ and hence it holds for all $j\leq al-3$ by induction.

	We first show that \eqref{1300} holds with constant $C_1(n)=c(n)C_R(n)$. To this end, we write
	\[
	\mu_{r_{j+1}}=\mu_{r_j}+\widetilde{\mu}_{r_{j+1}}:=\sum_{x\in\3_{{r_j}}}\omega_kr_x^k\de_x+\sum_{x\in\3,r_x\in({r_j},{r_{j+1}}]}\omega_kr_x^k\de_x.
	\]
	Take a covering of $D_{r_{j+1}}(x)$ by $M$ balls $\{D_{{r_j}}(y_i)\}$, $M\leq c(n)$, such that $\{D_{{r_j}/5}(y_i)\}$ are disjoint. Then by induction we have
	\[
	\mu_{{r_j}}(D_{r_{j+1}}(x))\leq\sum_{j=1}^{M}\mu_{{r_j}}(D_{{r_j}}(y_i))\leq c(n)C_R(n)r_j^k.
	\]
	By definition of $\widetilde{\mu}_{r_{j+1}}$ and pairwise disjointness of $\{D_{r_x/5}(x)\}$, we have
	\[
	\widetilde{\mu}_{r_{j+1}}(D_{r_{j+1}}(x))\leq c(n)r_{j+1}^k.
	\]
	Thus, for the measure $\mu_{r_{j+1}}$, there holds at the moment the rough estimate
	\begin{align}\label{13200}
	\mu_{r_{j+1}}(D_{r_{j+1}}(x))\leq c(n)C_R(n)r_{j+1}^k ,\qquad \forall\,x\in D_1.
	\end{align}
	
	To improve the estimate \eqref{13200}, we shall apply the discrete Reifenberg Theorem \ref{rre}. 
	Hence, let us fix  $D_{r_{j+1}}(x_0)$ and set
	\[
	{\mu_{j+1}}=\mu_{r_{j+1}}|_{D_{{r_{j+1}}}(x_0)}.
	\]
	We claim that for all $z\in \supp(\mu_{j+1})=\cal{C}_{r_{j+1}}\cap D_{r_{j+1}}(x_0)$,
	\begin{align}\label{137}
	\beta_2=\beta^k_{2,{\mu_{j+1}}}(z,s)^2\leq C_1s^{-k}\int_{D_s(z)}{W}_s(y)d{\mu_{j+1}}(y).
	\end{align}
	
	Note that \eqref{137} trivally holds if $0<s\leq r_z/5$ and we next consider the case $s\geq r_z/5$.  Since  $\Theta(u^e,B^+_{10s}(\textbf{z}))\leq E$, by (P1) we have
	$$\Theta(u^e,B^+_{10s}(\textbf{z}))-\Theta(u^e,B^+_{5s}(\textbf{z}))\leq\eta\quad \text{ for all } s\in[ r_{z}/5,1/10].$$
	Given $\delta_8(n,N,\Lambda,\rho,\ep)$ as  in Theorem \ref{ji}, it follows from Proposition \ref{pro1} that there exists $\eta_0=\eta_0(n,N,\Lambda,\rho)$ and $\de=\de_1(n,N,\Lambda,\rho,\ep,\de_8)>0$ such that if $\eta\leq\eta_0$, then $u$ is $(0,\de_8)$-symmetric on $D_{10s}(z)$. By (P2), $u$ is not $(k+1,\ep/2)$-symmetric on  $D_{10s}(z)$. Hence the claim follows from Theorem \ref{ji}.
	
	%Notice that the above inequality holds for all $0<s\leq 1/10$ since it is trivial for $0<s\leq r_x/5$ and
	We may assume without loss of generality $W_s(x)=0$ if $0<s\leq r_x/5$. Note that, by the induction assumption and \eqref{13200}, for all $j\leq al-3$ and $s\in [r,r_{j+1}],$ and $z\in D_1$,
	\begin{equation}\label{eq:rough growth for mu s}
	\mu_s(D_s(z))\leq c(n)C_R(n)s^k.
	\end{equation}
	We claim that for any $r<s\leq r_{j+1}$, we have
	\begin{equation}\label{eq:improved growth for mu j}
	\mu_{r_{j+1}}(D_s(z))\leq c(n)C_R(n)5^ks^k.
	\end{equation}
	Indeed, if $y\in D_s(z)\cap \supp(\mu)$, then $\frac{r_y}{5}\leq |y-z|\leq s$ and so $y\in \3_{5s}$, which implies $D_s(z)\cap \supp(\mu)\subset \3_{5s}$. Since $r\leq 5s\leq 5r_{j+1}$, we have
	\[
	\mu_{j+1}(D_s(z))\leq \mu_{5s}(D_s(z))\leq \mu_{5s}(D_{5s}(z))\leq c(n)C_R(n)5^ks^k.
	\]
	Fixing $s\le r\le r_{j+1}$ and integrating \eqref{137} over $D_r(y)\subset D_{r_{j+1}}(x_0)$ leads to
	\[
	\begin{aligned}
	&\int_{D_r(y)} \beta^k_{2,\mu_{j+1}}(z,s)^2d\mu_{j+1}(z)\\
	&\leq C_1s^{-k}\int_{D_r(y)}\Big[\int_{D_s(z)}{W}_s(x)d\mu_{j+1}(x)\Big]d\mu_{j+1}(z)\\
	&\leq C_1s^{-k}\int_{D_r(y)}\int_{D_{2r}(y)}\chi_{D_s(z)}(x)W_s(x)d\mu_{j+1}(x)d\mu_{j+1}(z)\\
	&\le C_1s^{-k}\int_{D_{2r}(y)}\mu_{j+1}(D_s(x))W_s(x)d\mu_{j+1}(x)\\
	&\stackrel{\eqref{eq:improved growth for mu j}}{\leq}C_1c(n)C_R(n)\int_{D_{2r}(y)}W_s(x)d\mu_{j+1}(x).
	\end{aligned}
	\]
	Hence integrating $s$ from 0 to $r$ leads to
	\[
	\int_{D_r(y)}\int_0^r \beta^k_{2,\mu_{j+1}}(z,s)^2\frac{ds}{s}d\mu_{j+1}(z)\leq c(n)C_1C_R\int_{D_{2r}(y)}\int^r_0{W}_s(x)\frac{ds}{s}d{\mu_{j+1}}(x).
	\]
	Observe that for all  $x\in \mbox{supp}(\mu)$ and $r\leq r_{j+1}\leq1/10,$ we have
	\begin{align*}
	\int_0^r{W}_s(x)\frac{ds}{s}&=\int_{r_x/5}^r{W}_s(x)\frac{ds}{s}\leq\int_{r_x/5}^{1/10}{W}_s(x)\frac{ds}{s}\\
	&\leq c[\Theta(u^e,B^+_1(\textbf{x}))-\Theta(u^e,B^+_{r_\textbf{x}/5}(\textbf{x}))]\leq c\eta.
	\end{align*}
	Using again the induction hypothesis and \eqref{13200}, we arrive at
	\begin{align}\label{1380}
	\int_{D_r(y)}\left(\int_0^r\beta^k_{2,\mu_{j+1}}(z,s)^2\frac{ds}{s}\right)d\mu_{j+1}(z)&\leq cc(n)C_1C_R^2\eta r_{j+1}^k\nonumber\\
	&\leq cc(n)C_1C_R^22^{al-3}\eta r^k
	\end{align}
	for all $y\in D_{r_{j+1}}(x)$ and $2^{3-al}r_{j+1}\leq r\leq r_{j+1}.$
	
Our desired estimate \eqref{137} follows by choosing $\eta $ small enough such that
\begin{equation}\label{equ: smallness of eta-3}
\eta\leq 2^{3-al}\frac{\de_7^2}{cc(n)C_1C_R^2},
\end{equation}
and then applying Theorem \ref{rre} to $\mu_{j+1}$. 
\end{itemize}
\end{proof}

We remark that by \cite[Theorem 1.1]{Miskiewicz-18-Finn}, the  smallness assumption \eqref{equ: smallness of eta-3} on $\eta$ is not 
necessary. However, this can not remove the smallness condition on $\eta$, since, for instance, we still need $\eta$ to be very small in the formula 
\eqref{137} below in order to apply Proposition \ref{pro1}.

\subsection{Proof of Lemma  \ref{ti}}
The proof is  is rather long and will be  divided into three steps.
For simplicity we will assume that
$$r=\rho^l, \qquad R=1.$$ We will use an induction
argument. We will use superscripts
$f,b$ to indicate {\it final} and {\it bad} balls respectively.
\begin{proof}
\begin{itemize}
	\item[\textbf{Step1}] (Recovering of bad balls: the first step)  First, Lemma \ref{oooa}  gives a {finite} covering $\{D_{r_x}(x)\}_{x\in \3}$ of $S\cap D_{1}(0)$ satisfying
	\[
	S\cap D_{1}(0)\subset\bigcup_{x\in{\cal C}_{r}^{0}}D_{r}(x)\cup\bigcup_{x\in{\cal C}_{+}^{0}}D_{r_{x}}(x),
	\]
	where
	\[
	{\cal C}_{r}^{0}=\Big\{x\in{\cal C}\ \big| r_{x}=r\Big\}\quad\text{and}\quad{\cal C}_{+}^{0}=\Big\{x\in{\cal C}\ \big| r_{x}>r\Big\},
	\]
	and
	\begin{equation}\label{eq: 6.67}
	\sum_{x\in{\cal C}_{r}^{0}\cup{\cal C}_{+}^{0}}r_{x}^{k}\le C_1(n).
	\end{equation}
	Moreover, for each $x\in{\cal C}_{+}^{0}$, the set
	\[
	\mathcal{F}_{x}=\Big\{ y\in S\cap D_{r_{x}}(x):\Theta(u^e,B^+_{ \rho r_{x}/10}(\mathbf{y}))>E-\de \Big\}
	\]
	is contained in a small neighborhood of a ($k-1$)-dimensional affine subspace. 
	
	To proceed, we only need to refine the part ${\cal C}_{+}^{0}$. Let $x\in{\cal C}_{+}^{0}$ so that $r_x>r$. We have two cases:
	\begin{itemize}
	\item[\textbf{Case1}]  $\rho r_{x}=r$. In view of the desired covering in the statement of Lemma \ref{ti}, we directly cover $S\cap D_{r_{x}}(x)$
	by a family of balls
	$$\Big\{D_{r_{y}}(y): y\in{\cal C}_{x}^{(1,r)}\subset S\cap D_{r_{x}}(x)\Big\}\qquad \text{with}\qquad r_{y}=\rho r_{x}=r,$$
and $\{D_{r_{y}/5}(y)\}_{y\in{\cal C}_{x}^{(1,r)}}$ being pairwise disjoint.  A simple volume comparison argument shows that the cardinality  $\sharp\left({\cal C}_{x}^{(1,r)}\right)\le C(n)\rho^{-n}$. Hence
	\[
	\sum_{y\in{\cal C}_{x}^{(1,r)}}r_{y}^{k}=\sharp\left({\cal C}_{x}^{(1,r)}\right)(\rho r_{x})^{k}\le C(n)\rho^{k-n}r_{x}^{k}=:C_{r}(n,\rho)r_{x}^{k}.
	\]
	Collect all such points $x$ and set
	\[
	{\cal C}^{(1,r)}={\cal C}_{r}^{0}\cup\bigcup_{x\in{\cal C}_{+}^{0},\rho r_{x}=r}{\cal C}_{x}^{(1,r)}.
	\]
	It follows from \eqref{eq: 6.67} that
	\begin{align*}
	\sum_{y\in{\cal C}^{(1,r)}}r_{y}^{k}&=\Big(\sum_{x\in{\cal C}_{r}^{0}}+\sum_{x\in{\cal C}_{+}^{0}}\sum_{y\in{\cal C}_{x}^{(1,r)}}\Big)r^{k}\\
	&\le\sum_{x\in{\cal C}_{r}^{0}}r^{k}+C_{r}(n,\rho)\sum_{x\in{\cal C}_{+}^{0}}r_{x}^{k}\\
	&\le C_{1}(n)C_{r}(n,\rho).
	\end{align*}
	
	\item[\textbf{Case2}]  $\rho r_{x}>r$.	
In this case, we  cover $S\cap D_{r_{x}}(x)$ by making  use of the fact that 
$$\mathcal{F}_{x}\subset D_{\rho r_{x}/5}(L)\cap D_{r_{x}}(x)$$
for some $(k-1)$-dimensional affine subspace $L$. First we choose a Vitali  covering for the part away from $D_{2\rho r_{x}}(\mathcal{F}_{x})$:
	\[
	S\cap D_{r_{x}}(x)\backslash D_{2\rho r_{x}}(\mathcal{F}_{x})\subset\bigcup_{y\in{\cal C}_{x}^{(1,f)}}D_{r_{y}}(y)\quad\text{with }r_{y}=\rho r_{x},
	\]
where ${\cal C}_{x}^{(1,f)}\subset S\cap D_{r_{x}}(x)$. Since $y$ is away from  $\mathcal{F}_{x}$,  the energy drop property holds, that is,
\begin{equation}\label{eq: energy-drop-2}
\Theta\Big(u^e, B^+_{r_y/10}(\mathbf{y})\Big)\le E-\de.
\end{equation}
In fact, we also have the almost desired energy drop property:
\begin{equation}\label{eq: energy-drop-3}
\sup_{z\in {S}\cap D_{r_{y}}(y)} \Theta(u^e, B^+_{\rho r_{y}/10}(\mathbf{z}))\le E-\de,  \qquad \forall\,y\in{\cal C}_{x}^{(1,f)}.
\end{equation}
Indeed, for each $z\in S\cap D_{r_{y}}(y)$,
the fact $|z-y|\le r_{y}=\rho r_{x}$ implies that $z\not\in \mathcal{F}_{x}$.
Note also that $S\cap D_{r_{y}}(y)\subset S\cap D_{2r_{x}}(x)$, since
$\rho=\rho(n)<\frac{1}{100}$. Hence $z\in S\cap D_{2r_{x}}(x)$.
Consequently, the definition of $\mathcal{F}_x$ implies that
\[
\Theta\left(u^{e},B_{\rho r_{y}/10}^{+}(\mathbf{z})\right)\le\Theta\left(u^{e},B_{\rho r_{x}/10}^{+}(\mathbf{z})\right)\le E-\de.
\]
This proves \eqref{eq: energy-drop-3}. Thanks to this energy drop  property,	
these balls $\{D_{r_{y}}(y): y\in {\cal C}_{x}^{(1,f)}\}$ will be part of the ``final'' balls. 
Moreover, the number $\sharp\left({\cal C}_{x}^{(1,f)}\right)$ is bounded from above by a constant $C(n)\rho^{-n}$. This implies that
	\[
	\sum_{y\in{\cal C}_{x}^{(1,f)}}r_{y}^{k}=\sharp\left( {\cal C}_{x}^{(1,f)}\right)(\rho r_{x})^{k}\le C(n)\rho^{k-n}r_{x}^{k}=C_{f}(n,\rho)r_{x}^{k}.
	\]
	Collecting all the subset ${\cal C}_{x}^{(1,f)}$ from the above,  we
	obtain the first generation of final balls:
	\[
	{\cal C}^{(1,f)}=\bigcup_{x\in{\cal C}_{+}^{0},\rho r_{x}>r}{\cal C}_{x}^{(1,f)},
	\]
	together with the volume estimate  from \eqref{eq: 6.67}  that
	\[
	\sum_{y\in{\cal C}^{(1,f)}}r_{y}^{k}=\sum_{x\in{\cal C}_{+}^{0}}\sum_{y\in{\cal C}_{x}^{(1,f)}}r_y^{k}\le C_{f}(n,\rho)\sum_{x\in{\cal C}_{+}^{0}}r_{x}^{k}\le C_{f}(n,\rho)C_{1}(n).
	\]
	
	For the remaining part of $S\cap D_{r_{x}}(x)$, we first choose a Vitali covering:
	\[
	S\cap D_{r_{x}}(x)\cap D_{2\rho r_{x}}(\mathcal{F}_{x})\subset\bigcup_{y\in{\cal C}_{x}^{1,b}}D_{r_{y}}(y)\quad\text{with }r_{y}=\rho r_{x},
	\]
 where ``$b$'' means bad balls, on which $r_y>r$ and
the	energy drop property \eqref{eq: energy-drop-2}   can not be  determined.
	However, since 
	$$\mathcal{F}_{x}\subset D_{\rho r_{x}/5}(L)\cap D_{2r_{x}}(x)$$
	for some ($k-1$)-dimensional affine subspace $L$, we have a better cardinality estimate by volume comparison argument:
	\[
\sharp\left(\3_{x}^{(1,b)}\right)\le C_b(n)\rho^{1-k}.
	\]
	This means that relatively there are  quite few  bad balls are. As a result,
	\[
	\sum_{y\in{\cal C}_{x}^{(1,b)}}r_{y}^{k}=\sharp\left(\3_{x}^{(1,b)}\right)\rho r_{x}^{k}  \le C_{b}(n)\cdot \rho \cdot r_{x}^{k}.
	\]
Collect all the bad balls together to get the first generation of bad balls
	\[
	{\cal C}^{(1,b)}=\bigcup_{x\in{\cal C}_{+}^{0},\rho r_{x}>r}{\cal C}_{x}^{(1,b)},
	\]
together with the volume estimate by \eqref{eq: 6.67}
	\[
	\sum_{y\in{\cal C}^{(1,b)}}r_{y}^{k}=\sum_{x\in{\cal C}_{+}^{0}}\sum_{y\in{\cal C}_{x}^{(1,b)}}r_{y}^{k}\le C_{b}(n)\rho\sum_{x\in{\cal C}_{+}^{0}}r_{x}^{k}\le C_{1}(n)C_{b}(n)\cdot \rho.
	\]
Now we choose $\rho=\rho(n)$ by requiring
\begin{equation}\label{eq:  definition of rho}
	0<\rho<\min\Big\{ 100^{-1},\frac{1}{2C_{1}(n)C_{b}(n)}\Big\}
\end{equation}
	such that
	\[
	\sum_{y\in{\cal C}^{(1,b)}}(\rho r_{x})^{k}\le \frac12.
	\]
Note that $r<r_{y}=\rho r_x\le\rho$ for all $y\in{\cal C}^{(1,b)}$.
\end{itemize}
		
\item[\textbf{Step2}] (Induction step) To proceed, denote
	\[
	C^\prime_{2}(n)=2C_{1}(n)\left(C_{r}(n,\rho)+C_{f}(n,\rho)\right),
	\]
	such that
	\[
\sum_{y\in{\cal C}^{(1,r)}\cup{\cal C}^{(1,f)}}r_{y}^{k} \le(C_{f}(n,\rho)+C_{r}(n,\rho))\sum_{x\in{\cal C}}r_{x}^{k}\le\frac{1}{2}C^\prime_{2}(n)
	\] and that
	\[
\sum_{y\in{\cal C}^{(1,b)}}(\rho r_{x})^{k}\le \frac12\le   \frac12C^\prime_{2}(n).
\]	
		Our aim is to derive the following
	covering: For each $1\le i\le l$, there holds
	\[
	S\cap D_{1}(0)\subset\bigcup_{x\in{\cal C}^{(i,r)}}D_{r_x}(x)\cup\bigcup_{x\in{\cal C}^{(i,f)}}D_{r_{x}}(x)\cup\bigcup_{x\in{\cal C}^{(i,b)}}D_{r_{x}}(x).
	\]
Moreover,  the following properties hold:
	\begin{itemize}
	\item[(1)] For each $x\in{\cal C}^{(i,r)}$, $r_{x}=r$;
	\item[(2)] For each $x\in{\cal C}^{(i,f)}$, the energy drop condition holds:
\begin{equation}\label{eq: energy-drop-4}
\sup_{z\in \cal{S}\cap D_{r_{x}}(x)} \Theta(u^e, B^+_{\rho r_{x}/10}(\mathbf{z}))\le E-\de;
\end{equation}
	these balls will be called final balls;
	\item[(3)] For each $x\in{\cal C}^{(i,b)}$, we have $r<r_{x}\le\rho^{i}$.
	On these ``bad'' balls, none of the above two  conditions is
	verified.
	\item[(4)] There holds
	\begin{equation}\label{eq: 6.70}
	\sum_{y\in{\cal C}^{(i,r)}\cup{\cal C}^{(i,f)}}r_{y}^{k}\le\big(\sum_{j=1}^{i}2^{-j}\big)C^\prime_{2}(n); \quad \sum_{y\in{\cal C}^{(i,b)}}r_{y}^{k}\le2^{-i}\le2^{-i}C^\prime_{2}(n).
	\end{equation}
	\end{itemize}
	We have proved the above induction for $i=1$ in Step 1. 	Suppose now
it holds for some $1\le i<l$. We need to prove that it holds for
$i+1$.		It is clear that we only need to decompose those bad balls.	
	Fix a point $x\in{\cal C}^{(1,b)}$, i.e., $D_{r_{x}}(x)$ is a bad
	ball and $r<r_{x}\le\rho^{i}$.
	\begin{itemize}
	\item[\textbf{Case1}]($\rho r_{x}=r$). Then we cover it simply by balls
	of radius $\rho r_{x}$ as in step 1, and obtain a covering of $r$-balls
	whose centers lie in ${\cal C}_{x}^{(i+1,r)}$, together with the estimate
	\[
	\sum_{y\in{\cal C}_{x}^{(i+1,r)}.}(\rho r_{x})^{k}\le C(n)\rho^{k-n}r_{x}^{k}=C_{r}(n,\rho)r_{x}^{k}.
	\]
	
	\item[\textbf{Case2}]($\rho r_{x}>r$). In this case we apply the previous Lemma
	to get a covering $\{D_{r_{y}}(y)\}_{y\in{\cal C}^{x}}$ such that
	\[
	S\cap D_{r_{x}}(x)\subset\bigcup_{y\in{\cal C}_{r}^{x}}D_{r}(y)\cup\bigcup_{y\in{\cal C}_{+}^{x}}D_{r_{y}}(y)\qquad\text{with }r_{y}\ge r
	\]
	and
	\[
	\sum_{y\in{\cal C}_{r}^{x}\cup{\cal C}_{+}^{x}} r_{y}^{k}\le C_{1}(n)r_{x}^{k}.
	\]
	Moreover, for each $y\in{\cal C}_{+}^{x}$, there is a ($k-1$)-dimensional
	affine subspace $L_{y}$ such that
	\[
	\mathcal{F}_{y}\equiv\Big\{ z\in S\cap D_{r_{y}}(y)\ \big|\Theta(u^e,B^+_{\rho r_{y}/10}(\textbf{z}))\ge E-\de\Big\} \subset D_{\rho r_{y}/5}(L_{y})\cap D_{r_{y}}(y).
	\]
	Of course we will reserve ${\cal C}_{+}^{x}$ as part of ${\cal C}^{(i+1,r)}$.
	Thus below we assume that $y\in{\cal C}_{+}^{x}$. The method is totally similar.
	\begin{itemize}
	\item[\textbf{Case 2.1}]($\rho r_{y}=r$). Then as that of Case 1, we get
	a simple covering of at most $C(n)\rho^{-n}$ $r$-balls $\{D_{r_z}(z)\}_{z\in{\cal C}_{y}^{(i+1,r)}}$
	of $D_{r_{y}}(y)$ with $r_z=r$. So we define
	\[
	{\cal C}^{(i+1,r)}={\cal C}^{(i,r)}\cup\bigcup_{x\in{\cal C}^{(i,b)},\rho r_{x}=r}{\cal C}_{x}^{(i+1,r)}\cup\bigcup_{x\in{\cal C}^{(i,b)},\rho r_{x}>r}\Big({\cal C}_{r}^{x}\cup\bigcup_{y\in{\cal C}_{+}^{x},\rho r_{y}=r}{\cal C}_{y}^{(i+1,r)}\Big).
	\]
	This shows that how much more is ${\cal C}^{(i+1,r)}$ than that of
	${\cal C}^{(i,r)}$. The newly generated $r$-balls have measures
		\[
	\begin{aligned} & \sum_{x\in{\cal C}^{(i,b)},\rho r_{x}=r}\big(\sum_{z\in{\cal C}_{x}^{(i+1,r)}}r_{z}^{k}\big)
+\sum_{x\in{\cal C}^{(i,b)},\rho r_{x}>r}\big(\sum_{y\in{\cal C}_{r}^{x}}r_{y}^{k}+\sum_{y\in{\cal C}_{+}^{x}}\sum_{z\in{\cal C}_{y}^{(i+1,r)}}r_{z}^{k}\big)\\
	&\quad\le C_{r}(n,\rho)\Big(\sum_{x\in{\cal C}^{(i,b)},\rho r_{x}=r}r_{x}^{k}+\sum_{x\in{\cal C}^{(i,b)},\rho r_{x}>r}\big(\sum_{y\in{\cal C}^{x}}r_{y}^{k}\big)\Big)\\
	&\quad\le C_{r}(n,\rho)C_{1}(n)\sum_{x\in{\cal C}^{(i,b)}}r_{x}^{k}\\
	&\quad\le 2^{-i}C_{r}(n,\rho)C_{1}(n).
	\end{aligned}
	\]
Thus the total measure of $r$-balls in the $(i+1)$-th generation is estimated by
	\[
	\begin{aligned}\sum_{z\in{\cal C}^{(i+1,r)}}r_{z}^{k}
	& \le\sum_{z\in{\cal C}^{(i,r)}}r_{z}^{k}+2^{-i}C_{r}(n,\rho)C_{1}(n).
	\end{aligned}
	\]
	\item[\textbf{Case 2.2}]($\rho r_{y}>r$). We cover $D_{r_{y}}(y)$ by
	\[
	\begin{aligned} & S\cap D_{r_{y}}(y)\backslash D_{2\rho r_{y}}(\mathcal{F}_{y})\subset\bigcup_{z\in{\cal C}_{y}^{(i+1,f)}}D_{r_{z}}(z),\\
	& S\cap D_{r_{y}}(y)\cap D_{2\rho r_{y}}(\mathcal{F}_{y})\subset\bigcup_{z\in{\cal C}_{y}^{(i+1,b)}}D_{r_{z}}(z),
	\end{aligned}
	\quad\text{with }r_{z}=\rho r_{y}.
	\]
As that of \eqref{eq: energy-drop-3},  each ball $D_{r_{z}}(z)$ in the first covering ${\cal C}_{y}^{(i+1,f)}\subset S\cap D_{r_{y}}(y)$ satisfies
	the energy drop property \eqref{eq: energy-drop-4} in the inductive assumption. Moreover,
	\[
	\sum_{z\in{\cal C}_{y}^{(i+1,f)}}r_{z}^{k}\le C_{f}(n,\rho)r_{y}^{k}, \quad \sum_{z\in{\cal C}_{y}^{(i+1,b)}}r_{z}^{k}\le C_{b}(n)\cdot \rho\cdot r_{y}^{k}.
	\]
	Note that in this case no $r$-balls occur anymore. Set
	\[
	{\cal C}^{(i+1,f)}={\cal C}^{(i,f)}\cup\bigcup_{x\in{\cal C}^{(i,b)}}\bigcup_{y\in{\cal C}_{+}^{x}}{\cal C}_{y}^{(i+1,f)},\quad{\cal C}^{(i+1,b)}=\bigcup_{x\in{\cal C}^{(i,b)}}\bigcup_{y\in{\cal C}_{+}^{x}}{\cal C}_{y}^{(i+1,b)}.
	\]
We have
	\[
	\begin{aligned}\sum_{z\in{\cal C}^{(i+1,f)}}r_{z}^{k} & =\sum_{z\in{\cal C}^{(i,f)}}r_{z}^{k}+\sum_{x\in{\cal C}^{(i,b)}}\sum_{y\in{\cal C}_{+}^{x}}\sum_{z\in{\cal C}_{y}^{(i+1,f)}}r_{z}^{k}\\
	& \le\sum_{z\in{\cal C}^{(i,f)}}r_{z}^{k}+C_{f}(n,\rho)\sum_{x\in{\cal C}^{(i,b)}}\sum_{y\in{\cal C}_{+}^{x}}r_{y}^{k}\\
	& \le\sum_{z\in{\cal C}^{(i,f)}}r_{z}^{k}+C_{f}(n,\rho)C_{1}(n)\sum_{x\in{\cal C}^{(i,b)}}r_{x}^{k}\\
	& \le\sum_{z\in{\cal C}^{(i,f)}}r_{z}^{k}+2^{-i}C_{f}(n,\rho)C_{1}(n).
	\end{aligned}
	\]
	Recall that $C^\prime_{2}(n)=2\left(C_{r}(n,\rho)+C_{f}(n,\rho)\right)C_{1}(n)$.
	Hence by the inductive assumption we obtain
	\begin{align*}
	\big(\sum_{z\in{\cal C}^{(i+1,r)}}+\sum_{z\in{\cal C}^{(i+1,f)}}\big)r_{z}^{k}
	&\le\big(\sum_{z\in{\cal C}^{(i,r)}}+\sum_{z\in{\cal C}^{(i,f)}}\big)r_{z}^{k}
	+2^{-i-1}C^\prime_{2}(n)\\
	&\le C^\prime_{2}(n)\sum_{j=1}^{i+1}2^{-j},
	\end{align*}
	and also by the choice of $\rho$, we obtain
	\[
	\begin{aligned}\sum_{z\in{\cal C}^{(i+1,b)}}r_{z}^{k}= & \sum_{x\in{\cal C}^{(i,b)}}\sum_{y\in{\cal C}_{+}^{x}}\sum_{z\in{\cal C}_{y}^{(i+1,b)}}r_{z}^{k}\le\sum_{x\in{\cal C}^{(i,b)}}\sum_{y\in{\cal C}_{+}^{x}}C_{b}(n)\rho r_{y}^{k}\\
	& \le\sum_{x\in{\cal C}^{(i,b)}}C_{b}(n)\rho C_{1}(n)r_{y}^{k}\le2^{-1}\sum_{x\in{\cal C}^{(i,b)}}r_{y}^{k}\le2^{-i-1}.
	\end{aligned}
	\]
	This proves \eqref{eq: 6.70} for $i+1$.
	
	Finally,  note that for every $z\in{\cal C}^{(i+1,b)}$, there exists $x\in{\cal C}^{(i,b)}$ such that
	$$r<r_{z}=\rho r_{y}\le\rho r_{x}\le\rho^{i+1}, $$
since $r_{x}\le\rho^{i}$ by inductive	 assumption.   The proof of the induction is complete.
Since $r=\rho^{l}$, the above procedure will stop at $i=l$.
\end{itemize}
\end{itemize}
	\item[\textbf{Step3}](Final refinement)  Now, we have obtained a covering of $S\cap D_1(0)$
	\[
	S\cap D_1(0)\su \bigcup_{x\in\3}D_{r_x}(x)=\bigcup_{x\in\3_{r}}D_{r_x}(x)\cup\bigcup_{x\in\3_{+}}D_{r_x}(x),
	\]
	with (by \eqref{eq: 6.70})
	\[\sum_{x\in\3}r_x^k\leq C^\prime_{2}(n),\]
	where $\3_{r}$ consists of centers of  $r$-balls in  $\3$, and  $\3_{+}$ consists of centers of  balls $D_{r_x}(x)$ with $r_x> r$ and
	\[
	\begin{aligned}
	\sup_{y\in S\cap D_{r_x}(x)}\Theta(u^e,B^+_{\rho r_{x}/10}(\mathbf{y}))\leq E-\de.
	\end{aligned}
	\]
	
	To deduce the final convering of $S\cap D_1(0)$, we need to decompose the balls in $\3_{+}$ into smaller ones, but in a controllable way. So fix a point $x\in\3_{+}$.  If $\rho^2 r_x\le r$, then directly choose a Vitali covering of at most $C(n,\rho)$ balls of $\{D_r(y)\}_{y\in \mathcal{C}^x_r}$ to cover $S\cap D_{r_x}(x)$. The volume   is well controlled:
	\[
	\sum_{y\in \mathcal{C}^x_r} r_y^k\le C(n,\rho) r_x^k.
	\]
	In the case $\rho^2 r_x> r$, we cover $S\cap D_{r_x}(x)$ by a Vitali covering of at most $C(n,\rho)$ balls of $\{D_{r_y}(y)\}_{y\in \mathcal{C}^x_+}$ with $r_y=\rho^2 r_x$. Then, we have the following inclusion relationship
	\[
	S\cap D_{r_x}(x)\subset   \bigcup_{y\in \mathcal{C}^x_+}D_{r_y}(y)\quad\text{and}\quad D_{r_y}(y)\subset D_{r_x}(x).
	\]
	This \sout{makes sure} ensures that for each $z\in D_{r_y}(y)$, there holds
	\[B^+_{2r_y}(z)=B^+_{2\rho^2 r_x}(z)\subset B^+_{\rho r_x/10}(z). \]
	Consequently, by the monotonicity of $\Theta$, we deduce, for each $y\in \mathcal{C}^x_+$, that
	\[
	\begin{aligned}
	\sup_{z\in S\cap D_{r_y}(y)}\Theta(u^e,B^+_{2 r_{y}}(\mathbf{z}))\leq E-\de.
	\end{aligned}
	\]
	With the help of the above decomposition, we obtain the covering
		\[S\cap D_1(0)\subset\bigcup_{x\in\3_{r}}D_{r}(x)\cup\bigcup_{x\in\3_{+}, \rho^2 r_x\le r}\bigcup_{ y\in \mathcal{C}^x_r}D_{r}(y)\cup\bigcup_{x\in\3_{+}, \rho^2 r_x> r}\bigcup_{y\in  \mathcal{C}^x_+}D_{r_y}(y),
		 \]
together with the volume estimate
	\[
	\sum_{x\in \mathcal{C}_r}r_x^k + \sum_{x\in \3_+}\sum_{y\in \mathcal{C}^x_r\cup \mathcal{C}^x_+ }r_y^k\leq C(n,\rho)\sum_{x\in \3}r_x^k\leq C(n,\rho)C_2^\prime(n).
	\]
	This completes the proof of Lemma \ref{ti} upon taking 	$C_2(n)=C(n,\rho)C_2^\prime(n)$.
	\end{itemize}
\end{proof}

\section{Proof of the main results}\label{sec:proof of main results}
\subsection{Proof of Theorem \ref{thm80}}
The proof of Theorem \ref{thm80} follows from the Main Covering Lemma  \ref{lemma:main covering lemma} and the rectifiable-Reifenberg Theorem \ref{rree}.
\begin{proof}[Proof of Theorem \ref{thm80}] Let $\de>0$ be defined as in Lemma \ref{lemma:main covering lemma}. It follows that, for all $0<r<\delta$,
\begin{align}\label{z3}
	\mbox{Vol}\Big(T_r(S^k_{\ep,r}(u))\cap D_1(0)\Big)\le \mbox{Vol}\Big(T_r\big(S^k_{\ep,\de }(u)\cap D_1(0)\big)\Big) \leq C'_\ep r^{n-k},
\end{align}
with $C'_{\ep}=C'_\ep(n,N,\Lambda,\ep)>0$.   This proves the  volume estimate \eqref{J} for all $0<r<\de$. In general, since $\de$ also depends on $n,N,\Lambda,\ep$, for any $\de\le r<1$ we have
\[\label{z3}
\mbox{Vol}\Big(T_r\big(S^k_{\ep,r}(u)\cap D_1(0)\big)\Big) \le \mbox{Vol}(T_1( D_1(0)) \leq C_m \le C_m \left(\frac{r}{\de} \right)^{n-k}.\]
The proof of  \eqref{J} is complete. The volume estimate \eqref{Jk} follows from  \eqref{J} by noting that $S^k_{\ep}(u)\subset S^k_{\ep, r}(u)$ for any $r>0$.

To  prove  the rectifiability of $S^k(u)$, it is sufficient to prove the rectifiability of $S^k_\ep(u)$ for each $\ep>0$, as $S^k(u)=\bigcup_{i\ge 1}S^k_{1/i}(u)$.

 By the volume estimate \eqref{Jk}, we have $\HH^k(S^k_\ep(u)\cap D_1(0))\leq C_\ep.$ Applying the same estimates on $D_r(x)$ with $x\in D_1(0)$ and $r\leq 1$ gives the Alhfors upper bound estimate
\begin{equation}\label{jkj}
\HH^k(S^k_\ep(u)\cap D_r(x))\leq C_\ep r^k.
\end{equation}

 Let $S\su S_\ep^k(u)\cap D_1(0)$ be an arbitrary measurable subset with $\HH^k(S)>0$. Set
 $$g(x,r)=\Theta(u^e,B^+_r(\mathbf{x}))-\Xi(u,x),\quad\forall\, x\in D_1(0)\text{ and }\ r\leq 1.$$
Since $u\in \widehat{H}_{\La}^{1/2}(\Omega,N)$, we know that $g$ is uniformly bounded and converges to zero everywhere,  the dominated convergence theorem implies that for each $\de>0,$ there exists $\overline{r}>0$ such that
 \begin{equation*}\label{139}
\medint_S g(x,10\overline{r})d\HH^k(x)\leq\de^2.
\end{equation*}
So we can find a measurable subset $E\su S$ with $\HH^k(E)\leq\de\HH^k(S)$ and $g(x,10\overline{r})\leq\de$ for all $x\in F=S\ba E$. Cover $F$ by a finite number of balls $D_{\overline{r}}(x_i)$ centered on $F.$ Rescaling if necessary, we may assume that $D_{\overline{r}}(x_i)=D_1(0)$. Then $g(x,10)\leq\de$ for $x\in F.$ Similar to \eqref{137}, choose $\de$ sufficiently small so that $u$ is $(0,\de_9)$ symmetric in $D_{10}$. Theorem \ref{ji} implies
\begin{equation*}
\beta_{2,\HH^k|_F}(z,s)^2\leq C_1s^{-k}\int_{D_s(z)}W_s(t)d\HH^k|_F(t) \quad \text{for all }z\in F,\ s\leq1.
\end{equation*}
Integrating the  estimate over $z\in D_r(x)$ and using \eqref{jkj} yield that for all  $x\in D_1(0)$ and $s\leq r\leq1,$ there holds
\begin{align*}
\int_{D_r(x)}\beta_{2,\HH^k|_F}(z,s)^2d\HH^k|_F(z)&\leq C_1s^{-k}\int_{D_r(x)}\int_{D_s(z)}W_s(t)d\HH^k|_F(t)d\HH^k|_F(z)\non\\
&\leq C_1C_\ep\int_{D_{r+s}(x)}W_s(z)d\HH^k|_F(z).
\end{align*}
Integrating again with respect to $s\in [0,r]$, similar to \eqref{1380}, we obtain that for all $x\in D_1(0)$ and $r\leq1,$
\begin{align*}
\int_{D_r(x)}\int_0^r\beta_{2,\HH^k|_F}(z,s)^2\frac{ds}{s}d\HH^k|_F(z)&\leq C_1C_\ep\int_{D_{2r}(x)}[\Theta(u^e,B^+_{10r}(\textbf{z}))-\Xi(u,z)]d\HH^k|_F(z)\non\\
&\leq c(n)C_1C_\ep^2\de r^k.
\end{align*}
Choosing
$$\de\le\frac{\de_7^2}{c(n)C_1C_\ep^2},$$
 we deduce from Theorem \ref{rree} that $F\cap D_1(0)$ is $k$-rectifiable.

Repeating the above argument with $E$ in place of $F$, we could find another measurable set $E_1\subset E$ with $\HH^k(E_1)\leq \delta \HH^k(E)$, and that $F_1:=E\ba E_1$ is $k$-rectifiable. Continuing this process, we eventually conclude that $S$ is $k$-rectifiable.

The proofs of other assertions are similar to that of \cite{Naber-V-2017} and are omitted.
\end{proof}

%\section{Proofs of  Regularity theorems Theorems}

\subsection{Symmetry implies regularity}
To prove Theorems \ref{thm84} and \ref{thm85}, we  first deduce the following $\ep$-regularity theorem in the spirit of Cheeger-Naber \cite{Cheeger-Naber-2013-CPAM}.

\begin{theorem}\label{thm: sym-to-reg} Given $\La>0$. There exists
	a constant $\de(n,\La,s)>0$ such that, if $u\in\hat{H}_{\Lambda}^{1/2}(D_{3},N)$
	is a stationary $1/2$-harmonic map, then
	\[
	r_{u}(0)\ge\kappa_{2},
	\]
	where $\kappa_{2}=\kappa_{2}(n)>0$ is the constant given by Theorem
	\ref{thm: partial regularity of MPS}, provided one of the following
	conditions is satisfied:
	\begin{enumerate}
		\item $u$ is an $(n-1,\de)$-symmetric minimizing $1/2$-harmonic map;
		\item The target manifold $N$ does not admit any non-constant smooth 0-homogeneous
		stationary $1/2$-harmonic map from $\R^{2}\backslash\{0\}$ to $N$,
		and $u$ is an $(n-2,\de)$-symmetric stationary $1/2$-harmonic map;
		\item For some $k\ge1$, $N$ does not admit any non-constant smooth 0-homogeneous
		minimizing $1/2$-harmonic map from $\R^{l+1}\backslash\{0\}$ to $N$
		for all $1\le l\le k$, and $u$ is an $(n-k-1,\de)$-symmetric minimizing 		
		$1/2-harmonic map$;
		\item For some $k\ge1$, $N$ does not admit any non-constant smooth 0-homogeneous
		stationary $1/2$-harmonic map from $\R^{l+1}\backslash\{0\}$ to $N$
		for all $1\le l\le k$, and $u$ is an $(n-k-1,\de)$-symmetric stationary
		$1/2$-harmonic map.
	\end{enumerate}
\end{theorem}

The proof is divided into the following three Lemmas, which may be of
independent interest. The first one is a compactness result.

\begin{lemma}[Compactness]\label{lem: compactness} Let $\{u_{i}\}_{i\ge1}\subset\widehat{H}_\Lambda^{1/2}(D_3,N)$
	be a sequence of uniformly bounded stationary $1/2$-harmonic map, and
	$u_{i}\wto u$ in $\hat{H}^{1/2}(D_{3},N)$. Then $u$ is a
	weakly $1/2$-harmonic map. 	Moreover, there holds
	\[
	u_{i}\to u\qquad\text{strongly in }\hat{H}^{1/2}(D_{1},N),
	\]
	provided one of the following conditions is satisfied:
	\begin{enumerate}
		\item $\{u_{i}\}_{i\geq1}$ is a sequence of minimizing $1/2$-harmonic maps.
		In this case, $u$ is also a minimizing $1/2$-harmonic map.
		\item The target manifold $N$ does not admit any non-constant smooth 0-homogeneous
		stationary $1/2$-harmonic map from $\R^{2}\backslash\{0\}$ to $N$.
		In this case, $u$ is also a stationary $1/2$-harmonic map.
	\end{enumerate}
\end{lemma}

\begin{proof}
	Assertion (1) is proved in the case $N=\mathbb{S}^{m-1}$ in \cite[Theorem 7.3]{Millot-Pegon-Schikorra-2021-ARMA}. {The general case can be proved similar to that of Lin \cite{Lin-1999-Annals}.}
	
Assertion (2) is proved in Theorem \ref{thm: compactness of stationary HM}.
\end{proof}
\begin{lemma}[$(n,\epsilon)$-regularity]\label{lemma: new epsilon regularity}
	There exists $\ep>0$ depending only on $n,\La,m$ such that
	\[
	r_{u}(0)\ge\kappa_{2},
	\]
	provided one of the following conditions is satisfied:
	\begin{enumerate}
		\item $u\in\widehat{H}_\Lambda^{1/2}(D_3,N)$ is an $(n,\ep)$-symmetric
		minimizing $1/2$-harmonic map;
		\item $u\in\widehat{H}_\Lambda^{1/2}(D_3,N)$ is an $(n,\ep)$-symmetric
		stationary $1/2$-harmonic map and the target manifold $N$ does not
		admit any non-constant smooth $0$-homogeneous stationary biharmonic map
		from $\R^{2}\backslash\{0\}$ to $N$.
	\end{enumerate}
\end{lemma}
\begin{proof}
	We argue by contradiction. Suppose condition (1) holds. For any $k\ge1$, there exists an $(n,1/k)$-symmetric
	minimizing $1/2$-harmonic maps $u_{k}\in\widehat{H}_\Lambda^{1/2}(D_3,N)$
	such that $r_{u_{k}}(0)<1/2$. By Lemma \ref{lem: compactness}, we
	can assume that $u_{k}\to u$ strongly in $\widehat{H}_\Lambda^{1/2}(D_3,N)$
	for some minimizing $1/2$-harmonic map $u\in\widehat{H}_\Lambda^{1/2}(D_3,N)$.
	Moreover, $u$ is $n$-symmetric. Thus $u\equiv const.$ As a result, the strong convergence implies that ${\cal E}(u_{k},D_{1})<\varepsilon_{2}$ for $k\gg1$,
	which in turn implies that $r_{u_{k}}(0)\ge\kappa_{2}$ by $\ep$-regularity theory in Section 2. We reach 	a contradiction, and assertion (1) is proved.
	Assertion (2) is proved similarly in view of the compactness Lemma \ref{lem: compactness}.
\end{proof}
The last ingredient \sout{of} is the following symmetry self-improvement
Lemma.

\begin{lemma}[symmetry self-improvement] For any $\ep>0$, there
	exists $\de>0$ such that for any stationary $1/2$-harmonic map $u\in\widehat{H}_\Lambda^{1/2}(D_3,N)$,
	if $u$ satisfies one of the four conditions in Theorem \ref{thm: sym-to-reg},
	then $u$ is also $(n,\ep)$-symmetric on $D_{1}$. \end{lemma}
\begin{proof}
	\textbf{Case 1.} Condition (1) holds.
	Suppose, for some $\ep_{0}>0$,  there exists a sequence of minimizing
	$1/2$-harmonic maps $u_{k}\in\widehat{H}_\Lambda^{1/2}(D_3,N)$
	which is $(n-1,1/k)$-symmetric for each $k\ge1$ but  not $(n,\ep_{0})$-symmetric.	By Lemma \ref{lem: compactness}, we can assume that $u_{k}\wto u$
	in $\widehat{H}_\Lambda^{1/2}(D_1,N)$ for some weakly $1/2$-harmonic
	map $u\in\widehat{H}_\Lambda^{1/2}(D_3,N)$.   Then $u$ is $n-1$ symmetric on $D_{1}$. (Note that here we  do not even need to know whether $u$ is minimizing or not, since we have very high 	symmetry). However this
	implies that $u$ is a  $1/2$-harmonic map on the one dimensional
	interval $(-1,1)$ and thus smooth. But then the homogeneity of $i$
	 in turn implies that $u\equiv const$.
	This contradicts to the fact that $u_{k}$ is not $(n,\ep_{0})$-symmetric
	on $D_{1}$ since $u_{k}\to u$ strongly in $L^{2}(D_{1})$. The proof
	is complete. The remaining cases are totally similar and the details are omitted.
\end{proof}
Now we can prove Theorem \ref{thm: sym-to-reg}.

\begin{proof}[Proof of Theorem \ref{thm: sym-to-reg}] It follows
	from the above three Lemmas. The proof is complete.\end{proof}

\subsection{Proofs of Theorems \ref{thm84} and \ref{thm85} }
Now we sketch the proof of Theorem \ref{thm84}.
\begin{proof}[Proof of Theorem \ref{thm84}]
By a scaling argument, 	Theorem \ref{thm: sym-to-reg} implies
	$$\{x\in D_1(0) : r_u(x)<r\}\su S_{\ep,r}^{n-2}(u).$$
	Thus by Theorem \ref{thm80} there exists $C>0$ such that for each $0<r<1$ we have
	\[
	{\rm Vol}(D_r(\{x\in D_1(0):r_u(x)<r\}))\leq{\rm Vol}(D_r(S_{\ep,r}^{n-2}(u)))\leq Cr^2,
	\]
	which gives the second estimate of \eqref{emhm}. Moreover, note that $\dim_{\mathcal{H}}(S_{\ep,r}^{n-2})\le n-2$ implies that
	$$\dim_{\mathcal{H}}({\rm sing}(u))\le n-2.$$
	
	 To estimate $\na u$, simply observe that by Definition \ref{pppv} there holds
	\begin{align*}
	\{x\in D_1(0):|\nabla u|>r^{-1}\}\su\{x\in D_1(0):r_u(x)<r\}.
	\end{align*}
	The proof is thus complete.
\end{proof}

\begin{proof}[Proof Theorem \ref{thm85}]
It is similar to that of Theorem \ref{thm84} and hence is omitted.
\end{proof}
\medskip 

\textbf{Acknowledgement.} C.Y. Guo is supported by the Young Scientist Program of the Ministry of Science and Technology of China (No.~2021YFA1002200), the NSF of China (No.~12101362 and 12311530037), the Taishan Scholar Project and the NSF of Shandong Province (No.~ZR2022YQ01). C. Y. Wang is partially supported by NSF 
DMS 2453789 and Simons Travel Grant TSM-00007723.  C.L. Xiang is financially supported by the NSFC  (No.~12271296) and  the NSF of Hubei province (No. 2024AFA061). G. F. Zheng is supported by the National Natural Science Foundation of China (No.~12271195 and No.~12171180). Guo and Xiang are partly supported by the Jiangsu Provincial
Scientific Research Center of Applied Mathematics (Grant No.~BK20233002). Jiang, Xiang and Zheng are also partly supported by the Open Research Fund of Key Laboratory of Nonlinear Analysis \& Applications  (Central China Normal University) (No.NAA2023ORG003), Ministry of Education, P. R. China.


\begin{thebibliography}{99}

 %\bibitem{Angelsberg-2006}
% \textsc{G. Angelsberg}, \emph{A monotonicity formula for stationary weakly biharmonic maps}, Math. Z. \textbf{252} (2006), no. 2, 287-293.

\bibitem{Adams-1996}
D. R. Adams, L. Hedberg, Function Spaces and Potential Theory. Springer, Berlin, 1996.


\bibitem{Berlyand}
\textsc{L. Berlyand, P. Mironescu, V. Rybalko and E. Sandier}, {\it Minimax critical points in Ginzburg-Landau problems with semi-stiff boundary conditions: existence and bubbling}, Commun. Partial Differ. Eqs. \textbf{9} (2014), 946-1005.


 \bibitem{Bethuel-1993}
 F. Bethuel, {\it On the singular set of stationary harmonic maps}, Manuscripta Math. \textbf{78} (1993), 417-443.

% \bibitem{Branding-Oniciuc-2019}
 %\textsc{V. Branding and C. Oniciuc}, \emph{Unique continuation theorems for biharmonic maps}, Bull. Lond. math. Soc. \textbf{51} (2019), no. 4, 603-621.

 \bibitem{Breiner-Lamm-2015}
 \textsc{C. Breiner and T. Lamm}, {\it Quantitative stratification and higher regularity for biharmonic maps}, Manuscripta Math. \textbf{148} (2015), no. 3-4, 379-398.

 \bibitem{L. Caffarelli and L.Silvestre}
  \textsc{L. Caffarelli and L.Silvestre}, {\it An extension problem related to the fractional Laplacian}, Comm. Partial Ddiffetential Equations. \textbf{32} (2007), no. 7-9, 1245-1260.

% \bibitem{Chang-W-Y-1999} \textsc{S.-Y. Chang, L. Wang and P. Yang,}
 %\emph{A regularity theory of biharmonic maps.} Comm. Pure Appl.
 %Math. \textbf{52} (1999), no. 9, 1113-1137.

 %\bibitem{Cheeger-H-N-2013}
 %\textsc{J. Cheeger, R.Haslhofer and A. Naber}, \emph{Quantitative stratification and the regularity of mean curvature flow}, Geom. Funct. Anal. \textbf{23} (2013), no. 3, 828-847.

 %\bibitem{Cheeger-H-N-2015}
 %\textsc{J. Cheeger, R.Haslhofer and A. Naber}, \emph{Quantitative stratification and the regularity of harmonic map flow}, Cal. Var. Partial Differential Equations \textbf{53} (2015), no.  1-2, 365-381.


% \bibitem{Cheeger-Naber-2013-Invent}
 %\textsc{J. Cheeger and A. Naber}, \emph{Lower bounds on Ricci curvature and quantitative behavior of sigular sets}, Invent. Math. \textbf{191} (2013), 321-339.

 \bibitem{Cheeger-Naber-2013-CPAM}
 \textsc{J. Cheeger and A. Naber}, {\it Qantitative stratification and the regularity of harmonic maps and minimal currents}, Comm. Pure. Appl. Math. \textbf{66} (2013), no.6, 965-990.

 %\bibitem{Cheeger-N-V-2015}
 %\textsc{J. Cheeger, A. Naber and D.Valtorta}, \emph{Critical sets of elliptic equations}, Comm. Pure Appl. Math. \textbf{68} (2015), no. 2, 173-209.

% \bibitem{Cheeger-J-N-2021}
 %\textsc{J. Cheeger, W. Jiang and A. Naber},  \emph{Rectifiability of singular sets of noncollapsed limit spaces with Ricci curvature bounded below}, Ann. of Math. (2) \textbf{193} (2021), no. 2, 407-538.

% \bibitem{Chen-Yin-2019}
 %\textsc{Y. Chen and H. Yin}, \emph{Uniqueness of tangent cone for biharmonic map with isolated singularity}, Pacific J. Math. \textbf{299} (2019), no. 2, 401-426.

 %\bibitem{Chen-Zhu-2023}
% \textsc{Y. Chen and M. Zhu}, \emph{Bubbling analysis for extrinsic biharmonic maps from general Riemannian 4-manifolds}, Sci. China Math. 66 (2023), no. 3, 581-600.

% \bibitem{DeLellis-book} \textsc{C. De Lellis}, \emph{Rectifiable sets, densities and tangent measures.} Zur. Lect. Adv. Math. European Mathematical Society (EMS), Z\"urich, 2008.

\bibitem{CLMS-1993} \textsc{R. Coifman, P. L. Lions, Y. Meyer and S. Semmes,} {\it Compensated compactness and Hardy spaces.} J. Math. Pures Appl. (9) \textbf{72} (1993),  247-286.


\bibitem{Da Lio-book-2018}
\textsc{F. Da Lio}, Fractional Harmonic Maps, {\it Recent Developments in Nonlocal Theory}, De Gruyter, Berlin (2018).

    \bibitem{DaLio-etal-2022-AMPA} \textsc{F. Da Lio, K. Mazowiecka, A. Schikorra and L. Wang}, {\it A fractional version of Rivi\`ere's GL(N)-gauge.} Annali di Matematica \textbf{201}  (2022), 1817-1853.

 \bibitem{F. L. T. Riviere}
\textsc{F. Da Lio, L. Martinazzi and T. Rivière}, {\it Fractional Harmonic Maps, Blow-up analysis of a nonlocal Liouville-type equation}, Anal. PDE. \textbf{8} (2015), 1757-1805.

\bibitem{F. Da Lio and  A. Pigati}
\textsc{F. Da Lio and  A. Pigati}, {\it Free boundary minimal surfaces: a nonlocal approach}, to appear in Ann. Sc. Norm.Super. Pisa Cl. Sci. (5), preprint. arXiv:1712.04683

\bibitem{Da and Ri}
\textsc{F. Da Lio and T. Riviere}, {\it Three-term commutator estimates and the regularity of 1/2-harmonic maps into spheres}, Anal. PDE. \textbf{4}(1) (2011), 149-190.

\bibitem{Da and Ri 2}
\textsc{F. Da Lio and  T. Riviere}, {\it Sub-criticality of non-local Schrödinger systems with antisymmetric potentialsand applications to half-harmonic maps}, Adv. Math. \textbf{227}(3) (2011), 1300-1348.

\bibitem{F. Duzaar and K. Steffen}
\textsc{F. Duzaar and K. Steffen}, {\it A partial regularity theorem for harmonic maps at a freeboundary}, Asymptot. Anal. \textbf{2} (1989), 299–343.

\bibitem{F. Duzaar and K. Steffen 2}
\textsc{F. Duzaar and K. Steffen}, {\it An optimal estimate for the singular set of a harmonic mapin the free boundary}, J. Reine Angew. Math. \textbf{401} (1989), 157–187.


 %\bibitem{Evans-G-book-2015}
 %\textsc{L. C. Evans and R. F. Gariepy}, \emph{Measure theory and fine properties of functions}, Revised edition. Textbooks in Mathematics. CRC Press, Boca Raton, FL, 2015.
 %Studies in Advanced Mathematics. CRC Press, Boca Raton, FL, 1992.

 \bibitem{Evans-1991}
 \textsc{L. C. Evans}, {\it Partial regularity for stationary harmonic maps into spheres}, Arch. Rat. Mech. Anal. \textbf{116} (1991), 101-163.

 %\bibitem{Evans-book-2010}
% \textsc{L. C. Evans}, \emph{Partial differential equations} (2nd edition), Amer. Math. Soc. 2010.

%\bibitem{Fall-Felli-2014}
%\textsc{M. M. Fall  and V. Felli}  \emph{Unique continuation property and local asymptotics of solutions to
%fractional elliptic equations}, Commun. PDE. \textbf{39} (2014), 354–97.

 \bibitem{Federer-book-1969}
 \textsc{H. Federer}, {\it Geometric Measure Theory}, Springer-Verlag, New York Inc. 1969.


\bibitem{A. M. Fraser}
\textsc{A. M. Fraser}, {\it On the free boundary variational problem for minimal disks}, Commun. Pure Appl. Math. \textbf{53} (2000), 931–971.

\bibitem{Garofalo-Lin-86-Indiana}   \textsc{N. 		Garofalo and F. H.  Lin}, {\it Monotonicity properties of variational integrals, Ap weights and unique continuation.}   Indiana Univ. Math. J. \textbf{35} (1986), no. 2, 245-268.

 %\bibitem{Giaquinta-book-1983}
% \textsc{M. Giaquinta}, \emph{Multiple integrals in the calculus of variations and nonlinear elliptic systems}, Princeton University Press, 1983.

 %\bibitem{Giaquinta-Giusti-1983}
 %\textsc{M. Giaquinta and E. Giusti}, \emph{Differentiability of minima of non-differentiable functionals}, Invent. Math. \textbf{72} (1983), 285-298.

% \bibitem{Giaquinta-I-1987}
 %\textsc{M. Giaquinta and P. A. Ivert}, \emph{Partial regularity for minima of variational integrals}, Archiv for Math. \textbf{25} (1987), 221-229.

 %\bibitem{Giaquinta-M-book-2012}
% \textsc{M. Giaquinta and L. Martinazzi}, \emph{An Introduction to the Regularity Theory for Elliptic Systems, Harmonic Maps and Minimal Graphs}, Appunti. Scuola Normale Suoeriore di Pisa (Nuova Serie) [Lecture Notes. Scuola Normale Superiore di Pisa (Nuova Serie)], \textbf{11}. Edizioni della Normale, Pisa, 2012.

 %\bibitem{Grisvard-1985}
%\textsc{P. Grisvard}, \emph{Elliptic Problems in Nonsmooth Domains}, Pitman, Boston, MA (1985)

 %\bibitem{Gruter-1981}
 %\textsc{M. Gr\"{u}ter}, \emph{Regularity of weak $H$-surfaces}, J.reine angew. Math. \textbf{329} (1981),  1-15.

  \bibitem{Guo-Jiang-Xiang-Zheng}
\textsc{C.Y. Guo, G.C. Jiang, C.L. Xiang and G.F. Zheng}, {\it Optimaal higher regularity for biharmonic maps via quantitative stratification}, to appear in Peking Math. J., available at \url{https://link.springer.com/article/10.1007/s42543-025-00107-0}, 2025.

 %\bibitem{Guo-Wang-Xiang-2023-CVPDE}
 %\textsc{C.-Y. Guo, C. Y. Wang and C.-L. Xiang}, \emph{$L^p$-regularity for fourth order elliptic systems with antisymmetric potentials in higher dimensions}, Calc. Var. Partial Differential Equations \textbf{62} (2023), no. 1, Paper No. 31.

 %\bibitem{Guo-Xiang-2020-JLMS}
% \textsc{C.-Y. Guo and C.-L. Xiang}, \emph{Regularity of solutions for a fourth-order elliptic system via conservation law}, J. Lond. Math. Soc. (2) 101 (2020), no. 3, 907–922.

\bibitem{Gulliver-Jost1987}
\textsc{ R. Gulliver and J. Jost}, {\it Harmonic maps which solve a free boundary problem}. J. Reine
Angew. Math. \textbf{381} (1987), 61–89.



 %\bibitem{Guo-Xiang-2021-TAMS}
% \textsc{C. Y. Guo and C. L. Xiang}, \emph{Regularity of weak solutions to higher order elliptic systems in critical dimensions}, Tran. Amer. Math. Soc. (2) \textbf{374} (2021), no. 3-4, 491-524.


 %\bibitem{Guo-Xiang-Zheng-2021-CVPDE}
% \textsc{C.-Y. Guo, C.-L. Xiang and G.-F. Zheng}, \emph{The Lamm-Rivi\`ere system I: $L^P$ regularity theory}, Calc. Var. Partial Differential Equations 60 (2021), no. 6, Paper No. 213.

 %\bibitem{Guo-Xiang-Zheng-2022-JMPA}
% \textsc{C.-Y. Guo, C.-L. Xiang and G.-F. Zheng}, \emph{$L^p$ regularity theory for even order elliptic systems with antisymmetric first order potentials}, J. Math. Pures Appl. (9) \textbf{165} (2022), 286-324.

\bibitem{R. Hardt and F. H. Lin}
\textsc{R. Hardt and F. H. Lin}, {\it Partially constrained boundary conditions with energy minimiz-ing mappings}, Commun. Pure Appl. Math. \textbf{42} (1989), 309–334.

\bibitem{He-Z-X-24} \textsc{Y. He, G. F. Zheng and C L. Xiang},  {\it A global regularity theory for sphere-valued fractional harmonic maps}, Math. Z. 311 (2025), no. 4, Paper No. 89. [math.AP].

 %\bibitem{Helein-1990}
 %\textsc{F. H\'{e}lein}, \emph{R\'{e}gularit\'{e} des applications faiblement harmoniques entre une surface et une sph\'{e}re}, (French) [Regularity of weakly harmonic maps between a surface and $n$-sphere] C. R. Acad. Sci. Paris S\'{e}r. I Math. \textbf{311} (1990), no. 9, 519-524.

 \bibitem{Helein1991}
 \textsc{F. H\'{e}lein}, {\it R\'{e}gularit\'{e} des applications faiblement harmoniques entre une surface et une vari\'{e}t\'{e} riemannienne}, (French) [Regularity of weakly harmonic maps between a surface and a Riemannian manifold] C. R. Acad. Sci. Paris S\'{e}r. I Math. \textbf{312} (1991).


 %\bibitem{Helein-1991-Eng}
 %\textsc{F. H\'{e}lein}, \emph{Regularity of weakly harmonic maps from a surface into a manifold with symmetries}, Manuscripta Math. \textbf{70} (1991), no. {2}, 203-218.

% \bibitem{Helein-2002}
 %\textsc{F. H\'{e}lein}, \emph{Harmonic maps, conservation law and moving frames}, Translated from the 1996 French original. With a foreword by James Eells. Second edition. Cambridge Tracts in Mathematics, \textbf{150}. Cambridge University Press, Cambridge, 2002.

 %\bibitem{Hirsch-S-V-2019}
 %\textsc{J. Hirsch, S. Stuvard and D. Valtorta}, \emph{Rectifiability of the singular set of multiple valued energy minimizing harmonic maps}, Trans. Amer. Math. Soc. \textbf{371} (2019), no. {6}, 4303-4352.

 % \bibitem{Hitchhiker-2011}
% \textsc{D. N. Eleonora, P. Giampiero and V. Enrico},  \textsc{Hitchhiker’s guide to the
%fractional Sobolev spaces}, Bull. Sci. Math., \textbf{136} (2012), no. \{5\}, 521–573.


\bibitem{HSSW-2022-CPDE} \textsc{A. Hyder, A. Segatti, Y. Sire and C. Wang} {\it Partial regularity of the heat flow of half-harmonic maps and applications to harmonic maps with free boundary}. Communications in PDE, \textbf{47} (2022), 1845-1882.

 \bibitem{Laurain-Riviere-2013}
 \textsc{P. Laurain and T. Riviere}, {\it Energy quantization for biharmonic maps}, Adv. Calc. Var. \textbf{6} (2013), no. 2, 191-216.

\bibitem{Lin-1999-Annals}
\textsc{F. H. Lin}, {\it Gradient estimates and blow-up analysis for stationary harmonic maps}, Ann. of Math. (2) \textbf{149} (1999), no. 3, 785-829.

%\bibitem{Lin-Wand}
%\textsc{F. H. Lin and C. Y. Wang}, \emph{The Analysis of Harmonic Maps and Their Heat Flows}, World Scientific Publishing Co. Pte. Ltd., Hackensack, NJ, 2008.

% \bibitem{Liu-Yin-2016}
 %\textsc{L. Liu and H. Yin}, \emph{Neck analysis for biharmonic maps}, Math. Z. \textbf{283} (2016), no. 3-4, 807-834.


%\bibitem{Mattila-book-1995}
%\textsc{P. Mattila}, \emph{Geometry of sets and measures in Euclidean space}, Cambridge Studies in Advanced Mathematics, \textbf{44}. Cambridge University Press, Cambridge, 1995.

\bibitem{Mazowiecka-Schikorra-2018} \textsc{K. Mazowiecka and A. Schikorra,}
{\it Fractional div-curl quantities and applications to nonlocal geometric equations.} J. Funct. Anal. \textbf{275} (2018), no. 1, 1-44.

\bibitem{Millot-Pegon-Schikorra-2021-ARMA} \textsc{  V.  Millot, M.  Pegon and A. Schikorra, }
{\it Partial Regularity for Fractional Harmonic Maps into Spheres.} Arch. Ration. Mech. Anal. \textbf{242} (2021),  747-825.

\bibitem{Millot-Pegon-2020}
\textsc{V. Millot and M. Pegon}, {\it Minimizing 1/2-harmonic maps into spheres}, Calc. Var.
Partial Differ. Equ.  {59}, Art. \textbf{55}, (2020).

\bibitem{Millot-Sire-15} \textsc{ V. Millot and Y. Sire,}
{\it On a fractional Ginzburg-Landau equation and 1/2-harmonic maps into spheres.} Arch. Ration. Mech. Anal. \textbf{215} (2015),  125-210.

%\bibitem{Millot-Sire-Wang-19} \textsc{ V. Millot, Y. Sire and K. Wang,}
%\emph{Asymptotics for the fractional Allen-Cahn equation and stationary nonlocal minimal surfaces},
%Arch. Ration. Mech. Anal.\textbf{ 231} (2019), no. 2, 1129-1216.


% \bibitem{Moser-2008-CPDE}
 %\textsc{R. Moser}, \emph{A variational problem pertaining to biharmonic maps}, Comm. Partial Differential Equations \textbf{33} (2008), no. 7-9, 1654-1689.



\bibitem{Miskiewicz-18-Finn} \textsc{ M. Mi\'skiewicz}, {\it Discrete Reifenberg-type theorem.}
Ann. Acad. Sci. Fenn. Math. \textbf{ 43 } (2018),  3-19.


 %\bibitem{Morrey-1968}
 %\textsc{C. B. Morrey}, \emph{Partial regularity results for nonlinear elliptic systems}, Journ. Math. and Mech. \textbf{17} (1968), 649-670.

 %\bibitem{Morrey-book-1986}
 %\textsc{C. B. Morrey}, \emph{Multiple integrals in the calculus of variations of variations}, Springer-verlag, Berlin, 1986.


% \bibitem{Moser-book-2005}
 %\textsc{R. Moser}, \emph{Partial regularity for harmonic maps and related problem}, World Scientific Publishing Co. Pte. Ltd., Hackensack, NJ, 2005.

% \bibitem{Moser-2015-TAMS}
 %\textsc{R. Moser}, \emph{An $L^p$ regularity theory for harmonic maps}, Trans. Amer. Math. Soc. \textbf{367} (2015), no. 1, 1-30.


% \bibitem{Naber-Notes-2020}
 %\textsc{A. Naber}, \emph{Lecture notes on rectifiable Reifenberg for measures}, Harmonic analysis and applications, 289-345, IAS/Park City Math. Ser., \textbf{27}, Amer. Math. Soc., Providence, RI, 2020.


\bibitem{Naber-V-2017}
\textsc{A. Naber and D. Valtorta}, {\it Reifenberg-rectifiable and the regularity of stationary and minimizing harmonic maps}, Ann. of Math. (2) \textbf{185} (2017), no.1, 131-227.

 \bibitem{Naber-V-2018}
 \textsc{A. Naber and D. Valtorta}, {\it Stratification for sigular set of approximate harmonic maps}, Math. Z. \textbf{290} (2018), no. 3-4, 1415-1455.

\bibitem{Naber-Val-2025-arXiv} \textsc{A. Naber and D. Valtorta,}  {\it Energy identity for stationary harmonic maps}. Preprint at arXiv:2401.02242v2 [math.AP].


% \bibitem{Naber-V-V-2019}
 %\textsc{A. Naber, D. Valtorta and G. Veronelli}, \emph{Quantitative regularity for $p$-harmonic maps}, Comm. Anal. Geom. \textbf{27} (2019), no. 1, 111-159.

 %\bibitem{Nirenberg-1996}
% \textsc{L. Nirenberg}, \emph{An extended interpolation inequality}, Ann. Scuola Norm. Sup. Pisa CI. Sci. (3) \textbf{20} (1996),  no. 4,  733-737.


%\bibitem{Preiss-1987}
%D. Preiss, \emph{Geometry of measures in $\R^n$: distribution, rectifiability, and densities}, Ann. of Math. (2) \textbf{125} (1987), no. 3, 537-643.
% \bibitem{Riviere-1995-Acta}
 %\textsc{T. Rivi\`{e}re}, \emph{Everywhere discontinuous harmonic maps into spheres}, Acta Math. \textbf{175} (1995), 197-226.

 \bibitem{Riviere-2007-Invent}
 \textsc{T. Riviere}, {\it Conservation laws for conformally invariant variational problems}, Invent. Math. \textbf{168} (2007), 1-22.

 %\bibitem{Riviere-2012}
% \textsc{T. Rivi\`{e}re}, \emph{The role of conservation laws in the analysis of conformally invariant problems}, Edizioni della Normale, Pisa, 2012.

 %\bibitem{Rivere-Struwe-2008}
% \textsc{T. Rivi\`{e}re and M. Struwe}, \emph{Partial regularity for harmonic maps and related problems}, Commun.Pure Appl. Math. \textbf{61} (4) (2008), 451-463.

\bibitem{Scheven-2006-MZ}
\textsc{C. Scheven}, {\it Partial regularity for stationary harmonic maps at a free boundary}, Math. Z. \textbf{253}. (2006), 135-157.

%\bibitem{Sampson-1978}
 %\textsc{J. H. Sampson}. \emph{Some properties and applications of harmonic mappings}, Ann. Sci. Ecole Norm. ´Sup. (4), \textbf{11}(2) (1978), 211-228.

\bibitem{Scheven-2008-ACV}
\textsc{C. Scheven}, {\it Dimension reduction for the singular set of biharmonic maps}, Adv. Cal. Var. \textbf{1} (2008), no. 1, 53-91.

\bibitem{Scheven-2009-Poincare} \textsc{C. Scheven}, {\it An optimal partial regularity result for minimizers of an intrinsically defined second-order functional}, Ann. Inst. H. Poincar\'e Anal. Non Lin\'eaire \textbf{26} (5)  (2009), 1585–1605.

% \bibitem{Schikorra-2010-Poincare}
 %\textsc{A. Schikorra}, \emph{A remark on guage transformation and the moving frame method}, Ann. Inst. H. Poincar\'{e} Anal. Non Lin\'{e}aire \textbf{27} (2010), 503-515.


 \bibitem{Schikorra-2018-Anal}
 \textsc{A. Schikorra}, {\it Boundary equations and regularity theory for geometric variational systems with Neumanndata}, Arch. Ration. Mech. Anal. \textbf{229} (2018), 709–788.


% \bibitem{Schoen-1984}
 %\textsc{R. Schoen}, \emph{Analytic aspects of the harmonic map problem}, Seminar on Nonlinear Partial Differential Equations(S.S.Chern editor), MSRI Publications, vol.2, Springer-Verlag, New York, 1984.

 \bibitem{Schoen-Uhlenbeck-1982}
 \textsc{R. Schoen and K. Uhlenbeck}, {\it  A regularity and singularity theory for harmonic maps}, J. Diff. Geom. \textbf{17} (1982), 307-335.

% \bibitem{Simon-book-1983}
 %\textsc{L.Simon}, \emph{Lectures on Geometric Measure Theory}, Proceedings of the Centre for Mathematical Analysis, Australian National University, \textbf{3}. Australian National University, Centre for Mathematical Analysis, Canberra, 1983.
%
 %\bibitem{Simon-1996-Regularity}
 %\textsc{L. Simon}, \emph{Proof of the basic regularity theorem for harmonic maps}, Nonlinear partial differential equations in differential geometry (Park City, UT, 1992), 225-256, IAS/Park City Math. Ser., \textbf{2}, Amer. Maath. Soc., Province, RI, 1996.

% \bibitem{Simon-1996}
 %\textsc{L. Simon}, \emph{Singularities of geometric variational problems}, Nonlinear partial differential equations in differential geometry (Park City, UT, 1992), 225-256, IAS/Park City Math. Ser., \textbf{2}, Amer. Maath. Soc., Province, RI, 1996.

\bibitem{Simon-book-1996}
\textsc{L. Simon}, {\it Theorems on The Regularity and Singularity of Energy Minimizing Maps}, ETH Lecture Notes, Birkh\"{a}user, Z\"{u}rich, 1996.

 %\bibitem{Strzelecki-2003-CV} \textsc{P. Strzelecki,}
 %\emph{On biharmonic maps and their generalizations}. Calc. Var. Partial Differential Equations \textbf{18} (2003),  401-432.

 %\bibitem{Strzelecki-2006}
 %\textsc{P. Strzelecki}, \emph{Gagliardo-Nirenberg inequalities with a BMO term}, Bull. Lond. Math. Soc. \textbf {38}  (2006), 294-300.


 \bibitem{Struwe-1984}
\textsc{M. Struwe}, {\it On a free boundary problem for minimal surfaces}, Invent. Math. \textbf{75} (1984), 547-560.

\bibitem{Struwe-2024-APDE} \textsc{M. Struwe,} {\it  Plateau flow or the heat flow for half-harmonic maps}. Analysis and PDE \textbf{17} (2024), 1397-1438.


% \bibitem{Struwe-2008}
% \textsc{M. Struwe}, \emph{Partial regularity for biharmonic maps revisited}, Calc. Var. Partial Differential Equations \textbf{33} (2008), no. 2, 249-262.

 %\bibitem{Wang-2004-MathZ}
% \textsc{C. Y. Wang}, \emph{Biharmonic maps from $\R^4$ into a Riemannian manifold}, Math. Z. \textbf{247}, (2004), 65-87.

% \bibitem{Wang-2004-CPAM}
 %\textsc{C. Y. Wang}, \emph{Stationary weakly 1/2-harmonic maps from $\mathbb{R}^n$ into a Riemannian manifold}, Comm. Pure Appl. Math. \textbf{57}  (2004), 419-444.

 %\bibitem{Wang-2004-CVPDE}
 %\textsc{C. Y. Wang}, \emph{Remarks on biharmonic maps into spheres}, Calc. Var. Partial Differential Equations. \textbf{21} (2004), no. 3, 221-242.

%\bibitem{Wang-Zheng-2012-JFA}
%\textsc{C. Y. Wang and S. Z. Zheng}, \emph{Energy identity of approximate biharmonic maps to Riemannian manifolds and its application}, J. Funct. Anal. \textbf{263} (2012), no. 4, 960-987.



%\bibitem{Wang-Zheng-2013-DCDS}
%\textsc{C. Y. Wang and S. Z. Zheng}, \emph{Energy identity for a class of approximate biharmonic maps in dimension four}, Discrete Contin. Dyn. Syst. \textbf{33} (2013).

%\bibitem{Wang-Zheng-2012-JFA}
%%\textsc{C. Y. Wang and S. Z. Zheng}, \emph{Energy identity of approximate biharmonic maps to Riemannian manifolds and its application}, J. Funct. Anal. \textbf{263} (2012), no. 4, 960-987.

%\bibitem{Ziemer-book-1989}
%\textsc{W. P. Ziemer}, \emph{Weakly differentiable functions.Sobolev spaces and functions of bounded variation}, Graduate Texts in Mathematics, \textbf{120}. Springer, New York, 1989.


%\bibitem{phar}
%A.Naber,D.Voltorta,G.Veronelli,:Quantitative regularity for p-harmonic maps

\bibitem{Wang-Xiang-Zheng-2026-JDE}
\textsc{Y.-Y. Wang, C.-L. Xiang and G.-F. Zheng}, {\it Quantitative regularity for minimizing intrinsic fractional harmonic maps}, J. Differential Equations \textbf{463} (2026), Paper No. 114243.
\end{thebibliography}
\end{document}